\documentclass[letterpaper, reqno,11pt]{article}
\usepackage[margin=1.0in]{geometry}
\usepackage{color,latexsym,amsmath,amssymb, amsthm, graphicx,subfigure,revsymb, enumerate,xcolor,mathtools}
\usepackage{stackengine}
 \usepackage{todonotes}

\usepackage{amsmath}\allowdisplaybreaks 
\numberwithin{equation}{section}

\renewcommand{\dim}{\mathrm{dim}\,}

\newcommand{\RR}{\mathbb{R}}

\newcommand{\cP}{A} 
\newcommand{\cS}{\mathcal{S}}

\newcommand{\TT}{\mathbb{T}}

\newcommand{\cE}{\mathcal{E}}
\newcommand{\eps}{\varepsilon}

\newcommand{\tubes}{\mathbb{T}}
\newcommand{\supp}{\operatorname{supp}}

\newcommand{\dimA}{\dim_{\! \!A}}

\newtheorem{thm}{Theorem}
\numberwithin{thm}{section}

\newtheorem{prop}[thm]{Proposition}
\newtheorem{defn}[thm]{Definition}
\newtheorem{lem}[thm]{Lemma}

\newtheorem{cor}[thm]{Corollary}

\newtheorem*{4wayDicotPropEnv}{Proposition \ref{4WayDichotProp}}

\theoremstyle{remark}
\newtheorem{rem}[thm]{Remark}

\title{The Assouad dimension of Kakeya sets in $\RR^3$}
\author{Hong Wang\thanks{Courant Institute of Mathematical Sciences, New York University. New York, NY, USA.} \and Joshua Zahl \thanks{Department of Mathematics, The University of British Columbia. Vancouver, BC, Canada.}}
\date{\today}
\begin{document}
\maketitle

\begin{abstract}
This paper studies the structure of Kakeya sets in $\mathbb{R}^3$. We show that for every Kakeya set $K\subset\mathbb{R}^3$, there exist well-separated scales $0<\delta<\rho\leq 1$ so that the $\delta$ neighborhood of $K$ is almost as large as the $\rho$ neighborhood of $K$. As a consequence, every Kakeya set in $\mathbb{R}^3$ has Assouad dimension 3 and every Ahlfors-David regular Kakeya set in $\mathbb{R}^3$ has Hausdorff dimension 3. We also show that every Kakeya set in $\mathbb{R}^3$ that has ``stably equal'' Hausdorff and packing dimension (this is a new notion, which is introduced to avoid certain obvious obstructions) must have Hausdorff dimension 3.  

The above results follow from certain multi-scale structure theorems for arrangements of tubes and rectangular prisms in three dimensions, and a generalization of  the sticky Kakeya theorem previously proved by the authors. 
\end{abstract}

\section{Introduction}\label{introSection}
A Kakeya set is a compact subset of $\RR^n$ that contains a unit line segment pointing in every direction. The Kakeya set conjecture asserts that every Kakeya set in $\RR^n$ has Minkowski and Hausdorff dimension $n$. This conjecture is proved in the plane \cite{Cor, Dav}, and is open in three and higher dimensions. See \cite{KT02, Wolff96} for a survey of progress on the Kakeya conjecture. 

In this paper, we study the structure of Kakeya sets in $\RR^3$. We show that for every Kakeya set $K\subset\mathbb{R}^3$, there exist well-separated scales $0<\delta<\rho\leq 1$ so that the $\delta$ neighborhood of $K$ is almost as large as the $\rho$ neighborhood of $K$. The precise statement is given in Theorem \ref{mainThmDiscretized} below. As a consequence, we prove several weaker variants of the Kakeya set conjecture: we prove that every Kakeya set in $\RR^3$ has Assouad dimension 3, and we prove that every Kakeya set in $\RR^3$ with ``stably equal'' Hausdorff and packing dimension (see Definition \ref{essentiallyEqualHausdorffPacking} below) must have Hausdorff and packing dimension 3. In order to explain these statements precisely, we require the following definitions.
\begin{defn}\label{defnAssouadDimension}
Let $E\subset\RR^n$ be non-empty. The \emph{Assouad dimension} of $E$, denoted $\dimA(E)$, is the infimum of all $\beta\geq 0$ for which there exist positive constants $C$ and $r_0$ so that for all $0<\rho<r\leq r_0$, we have
\[
\sup_{x\in \RR^n}\mathcal{E}_\rho \big(E \cap B(x,r)\big)\leq C(r/\rho)^\beta.
\]
In the above, $\mathcal{E}_\rho(X)$ denotes the $\rho$-covering number of $X$.
\end{defn}

The Assouad dimension is always at least as large as the Hausdorff and Minkowski dimensions, i.e.~if $E\subset\RR^n$ is bounded then we have $0\leq \dim_{\!H}(E)\leq\underline{\dim}_M(E)\leq \overline{\dim}_M(E)\leq\dim_A(E)\leq n$. If $E$ is bounded and Ahlfors-David regular, then all of these dimensions are equal. See \cite{Fra} for further details on Assouad dimension and its properties. 

Our first result is the following weak version of the Kakeya set conjecture in $\RR^3$.

\begin{thm}\label{assouadDimThm}
Every Kakeya set in $\RR^3$ has Assouad dimension 3.
\end{thm}
\noindent\emph{Remarks}\\
1. Theorem \ref{assouadDimThm} holds for a slightly more general class of sets, where the lines satisfy a mild strengthening of the Wolff axioms. See the remarks in Section \ref{proofOfAssouadDimThmSec} for details.
\medskip

\medskip
\noindent 2. Our proof of Theorem \ref{assouadDimThm} establishes a slightly stronger (though more technical) result, which asserts that the analogue of Theorem \ref{assouadDimThm} continues to hold if we restrict to pairs of scales $\rho<r$ in Definition \ref{defnAssouadDimension} that are ``exponentially separated,'' in the sense that $\rho < r^{1+\eps}$. This shows that every Kakeya set in $\RR^3$ has quasi-Assouad dimension 3. (The Assouad dimension of a set is always bounded below by its quasi-Assouad dimension; see \cite{Fra} for further details).

\medskip
\noindent 3.
Fraser, Olson and Robinson \cite{FOR} proved that every half-extended Kakeya set in $\RR^n$ (i.e.~a subset of $\RR^n$ containing a half-infinite line segment in every direction) has Assouad dimension $n$. However this is a somewhat different question. In brief, Fraser, Olson and Robinson showed that the Assouad dimension does not increase under a ``zooming out'' rescaling, which transforms an extended Kakeya set into one where all lines pass through the origin (the latter type of set has full Assouad dimension).

\medskip

Next, we consider Kakeya sets with equal Hausdorff and packing dimension. Recall that for $E\subset\RR^n$, the packing dimension is given by 
\[
\dim_{\!P}(E)=\inf\{\sup \overline {\dim}_M E_i\},
\]
where $\overline{\dim}_M$ denotes the upper Minkowski dimension, and the infimum is taken over all decompositions $E=\bigcup_i E_i$ into countably many sets. We have $\dim_{\!P}(E)\geq\dim_{\!H}(E)$ for every set $E\subset\RR^n$, and if $\dim_{\!P}(E)=\dim_{\!H}(E)=\alpha $, then for every $\eps>0$ there exists a constant $C$ and a measure $\mu$ supported on $E$ so that for $\mu$-a.e.~$x\in E$ we have the Frostman-type estimate 
\[
C^{-1}r^{\alpha+\eps}\leq\mu(B(x,r))\leq C r^{\alpha-\eps}\quad\textrm{for all}\ 0<r\leq 1. 
\]

We would like to say that every Kakeya set in $\RR^3$ with equal Hausdorff and packing dimension must have dimension 3. Unfortunately, we cannot prove this statement, and indeed this would imply that every Kakeya set in $\RR^3$ has packing dimension 3: let $K\subset\RR^3$ be a Kakeya set, and let $X\subset\RR^3$ be a compact set with equal Hausdorff and packing dimension $\dim_{\!H}(X) = \dim_{\!P}(X) = \dim_{\!P}(K)$. Then $K'=K\cup X$ is a Kakeya set with equal Hausdorff and packing dimension, and this common value is $\dim_{\!P}(K)$. In the above example, the Kakeya set $K'$ has equal Hausdorff and packing dimension for the somewhat trivial reason that it is a union of a Kakeya set with potentially unequal Hausdorff and packing dimension, and a set of equal Hausdorff and packing dimension\footnote{We could even select $X$ to be a union of unit line segments pointing in different directions (specifically a $\dim X - 1$ dimensional set of directions), and modify $K$ so that $K'$ is a union of unit line segments, with one line segment pointing in each direction.}. In particular, if $\dim_{\!H} K < \dim_{\!H} K'=\alpha$, then an $\alpha$-dimensional Frostman measure supported on $K'$ would be null on $K$. To exclude this type of situation, we introduce the following definition.

\begin{defn}\label{essentiallyEqualHausdorffPacking}
We say a set compact set $K\subset\RR^n$ is a \emph{Kakeya set with stably equal Hausdorff and packing dimension} if for all $\eps>0$, there is a constant $C$ and a measure $\mu$ supported on $K$ with the following properties:
\begin{itemize}
	\item For all $x\in\operatorname{supp}(\mu)$ and all $0<r\leq 1$, we have
	\[
		C^{-1} r^{\alpha+\eps}\leq \mu(B(x,r))\leq C r^{\alpha-\eps},
	\]
	where $\alpha=\dim_{\!H}(K)$.

	\item There is a positive (two-dimensional Lebesgue) measure set of directions $\Omega\subset S^2,$ so that for each direction $e\in\Omega$, there is a line $\ell_e$ pointing in direction $e$ for which $\ell_e\cap \supp(\mu)$ has positive one-dimensional Lebesgue measure.
	\end{itemize}
\end{defn}

In particular, if $K$ is a Kakeya set with unequal Hausdorff and packing dimension, then the set $K'=K\cup X$ described above will not have stably equal Hausdorff and packing dimension\footnote{Unless the set $X$ is itself a Kakeya set with stably equal Hausdorff and packing dimension.}.

\begin{thm}\label{equalHausdorffPackingKakeya}
Every Kakeya set in $\RR^3$ with stably equal Hausdorff and packing dimension has Hausdorff and packing dimension 3.
\end{thm}

Theorems \ref{assouadDimThm} and \ref{equalHausdorffPackingKakeya} will follow from a common discretized Kakeya estimate, which is the main technical contribution of this paper. In the statement below, $N_\rho(X)$ denotes the $\rho$-neighborhood of the set $X$; a $\delta$-tube is the $\delta$-neighborhood of a unit line segment; and we say two $\delta$-tubes are essentially distinct if neither tube is contained in the 2-fold dilate of the other. See Section \ref{notationDefinitionsSection} for precise definitions.
\begin{thm}\label{mainThmDiscretized}
For all $\eps>0$, there exists $\eta,\delta_0>0$ so that the following holds for all $\delta\in(0,\delta_0]$. Let $\tubes$ be a set of essentially distinct $\delta$-tubes in $B(0,1)\subset \RR^3$, and suppose that each convex set of volume $V$ contains at most $\delta^{-\eta}V(\#\tubes)$ tubes from $\tubes$. For each $T\in\tubes$, let $Y(T)\subset T$ be a measurable set, and suppose that $\sum_{\tubes}|Y(T)|\geq\delta^{\eta}\sum_{\tubes}|T|$. 

Then there exists $\rho,r\in [\delta,1]$ with $\rho\leq\delta^{\eta}r$ and a ball $B$ of radius $r$ so that
\begin{equation}\label{rhoRNbhdFullVol}
\Big| B \cap N_{\rho}\Big(\bigcup_{T\in\tubes}Y(T)\Big)\Big| \geq (\rho/r)^{\eps} |B|. 
\end{equation}
\end{thm}

\begin{rem}
In the above theorem, we have required that the tubes be essentially distinct and that every convex set $U\subset\RR^3$ contain at most $\delta^{-\eta}|U|(\#\tubes)$ tubes from $\tubes$. This condition is satisfied, for example, if $\#\tubes=\delta^{-2}$ and the tubes in $\tubes$ point in $\delta$-separated directions. Indeed; for the latter condition we need only consider convex sets $U\subset B(0,1)\subset\RR^3$ of diameter roughly 1, and every set of this form is comparable to an ellipsoid whose axes have lengths $s\leq t \leq w$ with $w\sim 1$. If $T$ is a $\delta$-tube contained in $U$, then the direction of $T$ lies in a rectangular sector of $S^2$ of dimensions roughly $s\times t$; if the tubes in $\tubes$ point in $\delta$-separated directions, we conclude that $\#\{T\in\tubes\colon T\subset U\}\lesssim st\delta^{-2}\lesssim |U|(\#\tubes)$. The anti-clustering condition imposed in Theorem \ref{mainThmDiscretized} will be discussed further in Section \ref{wolffAxiomsSection}.
\end{rem}

Theorem \ref{mainThmDiscretized} can be amplified to obtain the following variant, which says that if $\tubes$ is a collection of $\delta$-tubes satisfying the anti-clustering condition described above, then there are two well-separated scales $\rho < r$ so that the $\rho$-neighborhood of $\bigcup_{\tubes}T$ is almost as large as its $r$-neighborhood. 

\begin{cor}\label{KakeyaTwoScales}
For all $\eps>0$, there exists $\eta,\delta_0>0$ so that the following holds for all $\delta\in(0,\delta_0]$. Let $\tubes$ be a set of essentially distinct $\delta$-tubes in $B(0,1)\subset \RR^3$, and suppose that each convex set of volume $V$ contains at most $\delta^{-\eta}V(\#\tubes)$ tubes from $\tubes$. For each $T\in\tubes$, let $Y(T)\subset T$ be a measurable set, and suppose that $\sum_{\tubes}|Y(T)|\geq\delta^{\eta}\sum_{\tubes}|T|$.  

Then there exist sets $Y'(T)\subset Y(T),$ $T\in\tubes$, with $\sum_\tubes|Y'(T)|\gtrsim|\log\delta|^{-2}\sum_\tubes|Y(T)|$, and there exist scales $\rho,r\in [\delta,1]$ with $\rho\leq\delta^{\eta}r$ so that
\begin{equation}\label{KakeyaSameTwoScalesIneq}
\Big| N_{\rho}\Big(\bigcup_{T\in\tubes}Y'(T)\Big)\Big| \geq (\rho/r)^{\eps} \Big| N_{r}\Big(\bigcup_{T\in\tubes}Y'(T)\Big)\Big|. 
\end{equation}
\end{cor}

\subsection{Main ideas, and a sketch of the proof}\label{proofSketchSection}
In \cite{WZ}, the authors proved that every sticky Kakeya set in $\RR^3$ has Hausdorff and Minkowski dimension 3. In brief, a \emph{Sticky Kakeya Set} in $\RR^n$ is a Kakeya set that, when discretized at scale $\delta$, gives rise to a collection of $\delta$-tubes with the following two properties: 
\begin{itemize}
	\item[(a)] The tubes point in $\delta$-separated directions.
	\item[(b)]For each $\delta\leq \rho\leq 1$, the tubes can be covered by a set of $\rho$-tubes, with few (i.e.~about 1) $\rho$-tubes pointing in each direction.
\end{itemize} 

Sticky Kakeya sets have two useful properties. First, for $\delta\leq\rho\leq 1$, the $\rho$-neighborhood of a sticky Kakeya set at scale $\delta$ is a sticky Kakeya set at scale $\rho$. In other words, if we thicken a sticky Kakeya set, we obtain a sticky Kakeya set at a coarser scale. Second, for $\delta\leq\rho\leq 1$, the $\rho$-neighborhood of a typical tube in our sticky Kakeya set will contain about $(\rho/\delta)^{1-n}$ $\delta$-tubes; if we rescale this $\rho$-tube to be comparable to the unit ball, then we obtain a sticky Kakeya set at scale $\delta/\rho$. In other words, if we zoom-in and rescale a sticky Kakeya set, we obtain a sticky Kakeya set at a coarser scale.

The Sticky Kakeya Theorem (or more precisely, a generalization that we will prove in Section \ref{stickySelfSimThmSec}) will play an important role in the proof of Theorem \ref{mainThmDiscretized}. In Section \ref{selfSimSection} we show that every set of $\delta$-tubes satisfying the hypotheses of Theorem \ref{mainThmDiscretized} must either satisfy a suitable analogue of stickiness (and hence we can apply our generalization of the Sticky Kakeya Theorem and obtain the conclusion of Theorem \ref{mainThmDiscretized}), or else at least one of the following three things must happen: 
\begin{itemize}
	\item[(i)] After thickening the tubes at a suitable scale, we obtain a richer collection of tubes, i.e.~a collection of tubes that has similar structural properties, but larger cardinality (relative to the new scale). 
	\item [(ii)] After zooming-in and rescaling at a suitably chosen location and scale, we obtain a richer collection of tubes. 
	\item[(iii)] After possibly coarsening and zooming-in, the tubes arrange themselves into ``planks'' (like planks of wood, these sets are long (i.e.~length 1) in one direction, much shorter in a second direction, and shorter still in the final direction). 
\end{itemize}

In Section \ref{selfSimSection} we show that if Item (iii) occurs, then the corresponding collection of tubes satisfies the conclusion of Theorem \ref{mainThmDiscretized}. This observation is the main new geometric input in the proof of Theorem \ref{mainThmDiscretized}. In brief, the argument is as follows. Suppose each plank has dimensions roughly $s\times t\times 1$, with $\delta\leq s<\!\!<t\leq 1$. This means that each plank has an associated plane, which is the span of the two vectors pointing in the longest and second-longest directions.  First we consider the case in which a typical pair of intersecting planks have corresponding planes that make angle $\theta>\!\!>s/t$. If this occurs, then the union of the $s$-neighborhood of these planks will contain a ball of radius $r=\theta t.$ This is much larger than $s$, and hence \eqref{rhoRNbhdFullVol} holds. Second, we consider the case in which a typical pair of intersecting planks have corresponding planes that make angle $\leq s/t$. This forces the planks to cluster into bigger planks dimensions $s'\times t'\times 1$, where $s' \ll t'\leq 1$ are substantially larger than $s$ and $t$, respectively. We can then repeat this argument at a larger scale. After finitely many iterations of this argument, we must find ourselves in the First Case. This concludes the argument. 

Summarizing the arguments above, if $\tubes$ is a counter-example to Theorem \ref{mainThmDiscretized}, then there exists either a thickening or a zooming-in and rescaling that produces a richer collection of tubes. However, Theorem \ref{mainThmDiscretized} is agnostic to location and scale, and this means that if $\tubes$ is a counter-example to Theorem \ref{mainThmDiscretized}, then any thickening or zooming-in and rescaling of $\tubes$ remains a counter-example. Thus to prove Theorem \ref{mainThmDiscretized}, we show that if a counter-example to Theorem \ref{mainThmDiscretized} exists, then there must exist a (nearly) maximally rich counter-example. But this is impossible, since some thickening or zooming-in and rescaling of this (supposedly) maximally rich counter-example is an even richer counter-example.

In Section \ref{wolffAxiomsSection} we precisely define the relevant generalization of stickiness, and in Section \ref{AssouadDimensionSection} we define what it means for $\tubes$ to be a maximally rich counter-example to Theorem \ref{mainThmDiscretized}. In Section \ref{selfSimSection} we carry out the bulk of the geometric arguments described above. In Section \ref{kakeyaSelfSimWolffSection} we show that a (hypothetical) counterexample to Theorem \ref{mainThmDiscretized} must satisfy the hypotheses of (a generalization of) the sticky Kakeya Theorem. In Sections \ref{stickySelfSimThmSec} and \ref{twistedProjectionsTubeSec} we prove the necessary generalization of the Sticky Kakeya Theorem (this is Theorem \ref{stickySelfSimThm}), thereby proving Theorem \ref{mainThmDiscretized}. Finally, in Section \ref{proofOfAssouadDimThmSec} we show how Theorem \ref{mainThmDiscretized} implies Theorems \ref{assouadDimThm} and \ref{equalHausdorffPackingKakeya}.

\subsection{Thanks}
The authors would like to thank Larry Guth, Jonathan Fraser, Terence Tao, and the anonymous referee for comments, suggestions, and corrections on an earlier draft of this manuscript. The authors would like to thank Sam Craig for pointing out an error in the proof of Lemma 4.4 in an earlier draft of this manuscript. Hong Wang is supported by NSF CAREER DMS-2238818 and NSF DMS-2055544. Joshua Zahl is supported by an NSERC Discovery Grant and an NSERC Alliance Grant.

\section{Notation and definitions}\label{notationDefinitionsSection}

We write $A\lesssim B$ or $A = O(B)$ to mean there is a constant $C$ (which might depend on the ambient dimension $n$) so that $A\leq CB$. We write $A\sim B$ if $A\lesssim B$ and $B\lesssim A$. If the constant $C$ is allowed to depend on an additional parameter such as $\eps$, then we denote this by $A\lesssim_\eps B$.  Many of our results will involve a small positive parameter, which we will call $\delta$. We write $A\lessapprox B$ to mean that there is a constant $C$ so that $A\leq C (\log 1/\delta)^C B$. As above, if the constant $C$ is allowed to depend on an additional parameter such as $\eps$, then we denote this by $A\lessapprox_\eps B$. Finally, during proof sketches and other informal remarks, we will sometimes write $A \ll B$ to suggest to the reader that that the real number $A$ is much smaller than the real number $B$.

\subsection{Tubes and their shadings}\label{tubesAndShadingsSection}
Recall that a $\delta$-tube in $\RR^n$ is the $\delta$-neighborhood of a unit line segment. We will use $T$ to denote a $\delta$-tube, and $\ell(T)$ to denote its coaxial line. If $T$ is a $\delta$-tube (or more generally, the translation of a centrally symmetric convex set) and $R>0$, we define $RT$ to be the $R$-fold dilate of $T$ around its center. We say two tubes $T,T'$ are \emph{essentially distinct} if $T\not\subset 2T'$ and $T'\not\subset 2T$.

We will sometimes be concerned with collections of tubes, plus shadings (i.e.~distinguished subsets) of these tubes. To this end, we introduce the following definitions. 
\begin{defn}\label{def: uniform}
Let $\mathcal{S}$ be a collection of subsets of $\RR^n$, and let $\tau\in (0,1]$. A collection of sets $\{Y(S)\subset S\colon S\in\mathcal{S}\}$ is called a  \emph{$\tau$-dense shading} of $\mathcal{S}$ if $\sum_{S\in\mathcal{S}} |Y(S)|\geq\tau\sum_{S\in\mathcal{S}}|S|$. $\{Y(S),\ S\in\mathcal{S}\}$ is called a \emph{uniformly $\tau$-dense shading} of $\mathcal{S}$ if $|Y(S)|\geq\tau|S|$ for each $S\in\mathcal{S}$. 
\end{defn}
In practice, $\mathcal{S}$ will either be a collection of $\delta$-tubes, or a collection of $s\times t\times 1$ rectangular prisms. We will use the notation $(\tubes,Y)_\delta$ to refer to a collection of $\delta$-tubes and an associated shading $\{Y(T)\colon T\in\tubes\}$. 

\begin{defn}\label{regularShading}
Let $S\subset\RR^n$, let $Y(S)\subset S$, and let $\delta>0$. We say the shading $Y(S)$ is \emph{regular} at scales $\geq\delta$ if for each $x\in Y(S)$ and each $r\in[\delta, 1]$, we have
\begin{equation}\label{regularShadingIneq}
|Y(S)\cap B(x,r)| \geq (100^n\log(1/\delta))^{-1} |Y(S)|\Big(\frac{|B(x,r)\cap S|}{|S|}\Big).
\end{equation}
If the quantity $\delta$ is apparent from context, then we will omit it and say that $Y(S)$ is regular. 
\end{defn}
The following is Lemma 2.7 from \cite{KWZ}. 
\begin{lem}\label{findRegularShading}
Let $\delta>0$, let $S\subset B(0,1)\subset \RR^n$, and let $Y(S)$ be a shading. Then there is a shading $Y'(S)\subset Y(S)$ that is regular at scales $\geq\delta$ and satisfies $|Y'(S)|\geq\frac{1}{2}|Y(S)|$. 
\end{lem}
The proof in \cite{KWZ} considers the special case where $S$ is a unit line segment (and hence $|B(x,r)\cap S|/|S|\sim r$), but the proof is identical. In brief, we consider the quantity $|Y(S)\cap B(x,r)|$ for each dyadic scale $r=2^k\delta$ and each ball $B(x,r)$ aligned to the dyadic grid. We delete those balls for which $|Y(S)\cap B(x,r)|$ is too small (i.e.~for which \eqref{regularShadingIneq} fails), and we denote the surviving subset of $Y(S)$ by $Y'(S)$. At each dyadic scale we have only deleted a small fraction of $Y(S)$, so at least half of $Y(S)$ survives this process.

\subsection{Wolff axioms and their generalizations}\label{wolffAxiomsSection}
In this section we will define several anti-clustering conditions that we can impose on our collection of tubes. In \cite{Wolff95} (see also Definition 13.1 from \cite{KLT}), Wolff introduced what are now called the Wolff axioms; this is a non-concentration condition that forbids a collection of $\delta$-tubes from clustering inside a fatter $\rho$-tube, and also forbids a collection of $\delta$-tubes from clustering near (affine) 2-planes. The precise definition is as follows.

\begin{defn}
Let $\tubes$ be a set of $\delta$-tubes in $\RR^n$. We say $\tubes$ satisfies the \emph{Wolff axioms} if:
\begin{itemize}
\item Every $\rho$-tube in $\RR^n$ can contain at most $C(\rho/\delta)^{n-1}$ tubes from $\tubes$.
\item Every rectangular prism in $\RR^n$ of dimensions $2\delta\times 2\delta\times\cdots\times 2\delta\times \rho\times 2$ can contain at most $C(\rho/\delta)$ tubes from $\tubes$.
\end{itemize}
\end{defn}
\noindent\emph{Remarks}\\
1. In the above definition, the constant $C$ is generally chosen to be a large number depending only on the ambient dimension $n$. \\
2. Every set of $\delta$-tubes pointing in $\delta$-separated directions will satisfy the Wolff axioms with constant $C=100^n$. \\
3. Every set of $\delta$-tubes contained in $B(0,1)\subset\RR^n$ that satisfies the Wolff axioms has cardinality at most $100^nC\delta^{1-n}$. 

\medskip

In the arguments that follow, we will need several variants of the Wolff axioms. We define these below. 
\begin{defn}\label{convexWolffAxiomsDefn}
Let $C\geq 1$. We say a multi-set $\mathcal{S}$ of sets in $\RR^n$ satisfies the \emph{Convex Wolff Axioms with error} $C$ if for all convex sets $W\subset\RR^n$, we have
\begin{equation}\label{convexWolff}
\#\{S\in\mathcal{S}\colon S\subset W\} \leq C|W|(\#\mathcal{S}).
\end{equation}
\end{defn}
Informally, when $C$ has size close to 1, we will sometimes suppress the role of $C$ and say that $\mathcal{S}$ satisfies the Convex Wolff Axioms. In practice we will have $n=3$, and the sets in $\mathcal{S}$ will be either $\delta$-tubes, or $s\times t\times 1$ rectangular prisms (not necessarily axis-parallel) for some $\delta\leq s \leq t \leq 1$. 

\begin{rem} If the elements of $\mathcal S$ are convex, then  $\#\mathcal{S}\geq C^{-1}\big(\inf_{S\in\mathcal{S}}|S|\big)^{-1}$. In particular, a set of $\delta$-tubes in $\RR^3$ satisfying the Convex Wolff Axioms must have cardinality $\gtrsim \delta^{-2}$. 
\end{rem}

\begin{rem}\label{convexSupSet}
Let $\mathcal{S}$ be a multi-set of subsets of $\RR^n$. For each $S\in\mathcal{S}$, let $T(S)\supset S$, and define the multi-set $\mathcal{S}'=\{T(S)\colon S\in\mathcal{S}\},$ i.e.~the multi-sets $\mathcal{S}$ and $\mathcal{S}'$ have the same cardinality. If $\mathcal{S}$ satisfies the Convex Wolff Axioms with error $C$, then $\mathcal{S}'$ also satisfies the Convex Wolff Axioms with error $C$. 
\end{rem}

In our arguments below, we will often consider a Kakeya set (or more accurately a collection of $\delta$-tubes) at many different scales. Typically, for a collection of $\delta$-tubes $\tubes$, we will want to examine the coarsening of $\tubes$ at a larger scale $\rho$, and also examine the collection of $\delta$-tubes from $\tubes$ that are contained inside each of these coarser $\rho$-tubes. As observed in Remark \ref{convexSupSet} above, if $\tubes$ satisfies the Convex Wolff Axioms, then the coarsening of $\tubes$ at a larger scale will also satisfy these axioms\footnote{In contrast, the more traditional condition that the tubes in $\tubes$ point in $\delta$-separated directions is \emph{not} preserved after coarsening.}. However, the (rescaled) collection of $\delta$-tubes from $\tubes$ contained inside each coarser $\rho$-tube might not satisfy the Convex Wolff Axioms. Our next task is to introduce a slightly stricter variant of the Convex Wolff Axioms that addresses this issue.

\begin{defn}\label{unitRescalingDefn}
Let $U$ be a convex set and let $S\subset U$. Let $U^*\supset U$ be the John ellipsoid that circumscribes $U$, and let $\phi_U$ be an affine transformation that sends $U^*$ to the unit ball. The \emph{unit rescaling of $S$ relative to $U$} is the set $\phi_U(S)$. If $\mathcal{S}$ be a multi-set of subsets of $\RR^n$ that are contained in $U$, then the \emph{unit rescaling of $\mathcal{S}$ relative to $U$} is the multi-set $\{\phi_U(S)\colon S\in\mathcal{S}\}$. 
\end{defn}

\begin{rem}
In \cite{WZ}, the authors considered a similar definition in the special case where $\mathcal{S}$ is a set of $\delta$-tubes and $U$ is a $\rho$-tube. In order for this definition to be compatible with the definition from \cite{WZ} of a tube's shading, the map $\phi_U$ was defined slightly differently, so that image of $U^*$ was the ellipsoid $\{Cx^2+Cy^2+z^2\leq 1\}$ (for some fixed constant $C$), rather than the unit ball. 
\end{rem}

\begin{defn}\label{coversDefn}
Let $\mathcal{A}$ be a set and $\mathcal{B}$ be a multi-set of convex subsets of $\RR^n$. We say that \emph{$\mathcal{A}$ is a cover of $\mathcal{B}$} if each set $B\in\mathcal{B}$ is contained in at least one set from $\mathcal{A}$. For each $A\in\mathcal{A}$, we write $\mathcal{B}[A]$ to denote the multi-set $\{B\in\mathcal{B}\colon B \subset A\}$. We say that $\mathcal{A}$ is a \emph{$K$-uniform cover} of $\mathcal{B}$ if $\# \mathcal{B}[A]\leq K (\#\mathcal{B}[A'])$ for each pair of sets $A,A'\in\mathcal{A}$. When $K=1$, we will abbreviate this to a \emph{uniform cover}. Finally, we say that $\mathcal{A}$ is a \emph{partitioning cover} of $\mathcal{B}$ if $\mathcal{A}$ is a cover of $\mathcal{B}$, and for every pair of distinct $A,A'\in\mathcal{A}$, the sets $\mathcal{B}[2A]$ and $\mathcal{B}[2A']$ are disjoint, where $2A$ denotes the 2-fold dilate of $A$.  In practice, $\mathcal{A}$ and $\mathcal{B}$ will be collections of $\rho$-tubes and $\delta$-tubes respectively, for some $\rho>\delta$. 
\end{defn}

\begin{rem}\label{consequenceOFUnifCover}
Note that if $\mathcal{B}$ satisfies the Convex Wolff Axioms with error $C$, and if $\mathcal{A}$ is a $K$-uniform partitioning cover of $\mathcal{B}$, then $\mathcal{A}$ satisfies the Convex Wolff Axioms with error $KC$. 
\end{rem}

\begin{defn}\label{convexAtEveryScale}
Let $C\geq 1,\delta>0$. We say that a set $\tubes$ of $\delta$-tubes in $\RR^n$ satisfies the \emph{Convex Wolff Axioms at every scale with error $C$} if the tubes in $\tubes$ are essentially distinct, and for every $\rho_0\in [\delta,1]$, there exists $\rho\in [\rho_0, C\rho_0)$ and a set of $\rho$-tubes $\tubes_\rho$ that satisfies the following properties:
\begin{itemize}
	\item[(i)] $\tubes_\rho$ is a $C$-uniform partitioning cover of $\tubes$.
	\item[(ii)] For each $T_\rho\in\tubes_\rho$, the unit rescaling of $\tubes[T_\rho]$ relative to $T_\rho$ satisfies the Convex Wolff Axioms with error $C$.
\end{itemize}
\end{defn}

As the name suggests, Definition \ref{convexAtEveryScale} is a multi-scale property. In particular, if $\tubes$ satisfies the Convex Wolff Axioms at every scale, then we can cover $\tubes$ by fatter $\rho$-tubes; both the $\rho$-tubes and the (rescaled) tubes from $\tubes$ inside each $\rho$-tube will again satisfy the Convex Wolff Axioms at every scale. The precise statement is as follows.

\begin{lem}\label{multiScaleWolffLem}
Let $\tubes$ be a set of $\delta$-tubes in $\RR^n$ that satisfies the Convex Wolff Axioms at every scale with error $C$. Let $\rho_0\in[\delta,1]$. Then there exists $\rho\in [\rho_0, C\rho_0]$ and a set of $\rho$-tubes $\tubes_\rho$, so that $\tubes_\rho$ satisfies the Convex Wolff Axioms at every scale with error $O(C)$, and for each $T_\rho\in\tubes_\rho$, the unit rescaling of $\tubes[T_\rho]$ relative to $T_\rho$ satisfies the Convex Wolff Axioms at every scale with error $O(C)$.
\end{lem}
\begin{proof}
Let $N$ be the largest integer with $C^{N+1}\leq\delta^{-1}$. For $i=1,\ldots, N$, let $\tubes_{\tau_i},$ be a set of $\tau_i$-tubes, with $\tau_i\in [\delta C^i, \delta C^{i+1})$ that satisfies Items (i) and (ii) from Definition \ref{convexAtEveryScale}.  We claim that for $i<j$, if $T_{\tau_i}\in \tubes_{\tau_i}, T_{\tau_j}\in \tubes_{\tau_j}$, and $\tubes[T_{\tau_i}]\cap \tubes[T_{\tau_j}]\neq\emptyset$, then $\tubes[T_{\tau_i}]\subset\tubes[T_{\tau_j}]$. To see this, observe that if the intersection contains at least one tube $T\in\tubes$, then $2T_{\tau_j}\supset N_{\tau_j}(T) \supset T_{\tau_i}$. Since $\tubes_{\tau_j}$ covers $\tubes,$ and the sets $\{\tubes[2T_{\tau_j}']\colon T_{\tau_j}'\in\tubes_{\tau_j}\}$ are disjoint, we conclude that $\tubes[T_{\tau_i}]\subset\tubes[T_{\tau_j}]$. 

The consequence of the above observation is the following: if we partially order the sets $\{\tubes[T_{\tau_i}]\colon T_{\tau_i}\in\tubes_{\tau_i},\ i=1,\ldots,N\}$ under inclusion, then this forms a tree with $N$ levels---the vertices at level $i$ are precisely the sets $\tubes[T_{\tau_i}]$. By Item (i) of Definition \ref{convexAtEveryScale}, for each $i=1,\ldots,N$ the sets $\{\tubes[T_{\tau_i}]\colon T_{\tau_i}\in\tubes_{\tau_i}\}$ have comparable cardinality, up to a multiplicative factor of $C$. This tree is precisely what is needed to verify Lemma \ref{multiScaleWolffLem}: given $\rho_0\in[\delta,1]$, we select an index $i$ so that $\tau_i\in[\rho_0, C\rho_0)$, and we choose $\rho=\tau_i$, $\tubes_\rho=\tubes_{\tau_i}$; note that since $\tubes_\rho$ forms a partitioning cover of $\tubes$, the $\rho$-tubes in $\tubes_\rho$ are essentially distinct.
\end{proof}

Finally, we will introduce two additional anti-concentration conditions that will play a technical role in the arguments to follow. 
\begin{defn}
Let $C\geq 1$ and let $\tubes$ be a set of $\delta$-tubes in $\RR^n$. We say that $\tubes$ satisfies a \emph{Frostman condition at dimension $\sigma$ with error $C$} if the tubes in $\tubes$ are essentially distinct, and for every $\rho\geq\delta$ and every $\rho$-tube $T_\rho$, we have $\#\tubes[T_\rho]\leq C\rho^\sigma(\#\tubes)$.
\end{defn}

\begin{defn}\label{selfSimilarWolff}
Let $C\geq 1$ and let $\tubes$ be a set of $\delta$-tubes in $\RR^n$. We say that $\tubes$ satisfies the \emph{self-similar Convex Wolff Axioms with error $C$} if the tubes in $\tubes$ are essentially distinct, and for every $\rho_0\in [\delta,1]$, there exists $\rho\in [\rho_0, C\rho_0)$ and a set of $\rho$-tubes $\tubes_\rho$ that satisfies the following properties:
\begin{itemize}
	\item[(i)] $\tubes_\rho$ is a $C$-uniform partitioning cover of $\tubes$.
	\item[(ii)] For each $T_\rho\in\tubes_\rho$, the unit rescaling of $\tubes[T_\rho]$ relative to $T_\rho$ satisfies the Convex Wolff Axioms with error $C$.
	\item[(iii)] For each $T_\rho\in\tubes_\rho$, we have 
\begin{equation}\label{selfSimilarNumberTubes}
C^{-1}(\rho/\delta)^{\sigma}\leq \#\tubes[T_\rho]\leq C(\rho/\delta)^{\sigma},
\end{equation}
where $\sigma>0$ is the unique number satisfying $\#\tubes=\delta^{-\sigma}$.
\end{itemize}

\end{defn}
Note that Items (i) and (ii) above are identical to their counterparts in Definition \ref{convexAtEveryScale}. In particular, we have
\begin{equation}\label{CWAImplication}
\textrm{Self-similar CWA w.~err.~$C$} \implies \textrm{CWA at every scale w.~err.~$C$}\implies \textrm{CWA w/~err.~$C$}.
\end{equation}

\subsection{Discretized Assouad dimension}
The conclusion of Theorem \ref{mainThmDiscretized} involves a discretized analogue of Assouad dimension. We formalize this as follows.

\begin{defn}
Let $E\subset\RR^n$, let $\beta\in[0,n],$ let $\delta>0$, and let $A\geq 1$. We say that $E$ has \emph{discretized Assouad dimension at least $\beta$, at scale $\delta$ and scale separation $A$} if there exist scales $\delta\leq\rho\leq r\leq 1$ with $r\geq A\rho$, and a ball $B$ of radius $r$, so that
\[
\big| B \cap N_{\rho}(E)\big| \geq (\rho/r)^{n-\beta} |B|. 
\]
If the scale $\delta$ is apparent from context, then for brevity we may say ``$E$ has discretized Assouad dimension at least $\beta$ at scale separation $A$.''
\end{defn}

In this language, the conclusion of Theorem \ref{mainThmDiscretized} says that the set $\bigcup_{\tubes}Y(T)$ has discretized Assouad dimension at least $3-\eps$ at scale separation $\delta^{-\eta}.$

\section{Assouad-extremal Kakeya sets}\label{AssouadDimensionSection}

In this section we will suppose that Theorem \ref{mainThmDiscretized} is false, and we seek to construct a ``worst possible'' counter-example to Theorem \ref{mainThmDiscretized}. In this and later sections, we will restrict attention to $\RR^3$. Our example will be worst possible in two respects. First, Inequality \eqref{rhoRNbhdFullVol} will fail as dramatically as possible, i.e.~a reverse inequality will hold where the term $(\rho/r)^\eps$ on the RHS of \eqref{rhoRNbhdFullVol} will be replaced by $(\rho/r)^\omega$ for the largest possible value of $\omega>0$. Second, of all collections of tubes for which such a reverse inequality holds for this value of $\omega$, we will choose a set $\tubes$ of largest possible cardinality. The hypotheses of Theorem \ref{mainThmDiscretized} ensure that $\tubes$ has size at least $\delta^{-2}$; we will choose a set of tubes of size $\delta^{-\alpha}$, for the largest possible value of $\alpha$. 

We now turn to the task of carefully defining the quantities $\omega$ and $\alpha$; this will be done in Sections \ref{definingOmegaSection} and \ref{definingAlphaSection}, respectively. This task is complicated by the fact that Theorem \ref{mainThmDiscretized} involves a sequence of quantifiers, and we must unwind these quantifiers in the correct order.

\subsection{Defining $\omega$}\label{definingOmegaSection}

Let $(\tubes,Y)_\delta$ be a set of tubes and their associated shading. Let $\delta\leq\rho< r\leq 1$, and let $B$ be a ball of radius $r$ whose intersection with $\bigcup_\tubes Y(T)$ has positive measure.  Define $\zeta=\zeta( \tubes,Y,B, \rho)$ to be the unique number satisfying
\begin{equation}\label{defnZetaTubesRhoR}
\Big|B \cap  N_{\rho}\Big(\bigcup_{T\in\tubes}Y(T)\Big)\Big| = \Big(\frac{\rho}{r}\Big)^{\zeta} |B|.
\end{equation}

\noindent Next, for $\eta>0$, define
\begin{equation}\label{defnTubesYEta}
\zeta( \tubes,Y; \eta)=\inf_{B,\rho} \zeta( \tubes,Y,B, \rho),
\end{equation}
where the infimum is taken over all pairs $(B,\rho)$, where $\rho\in[\delta, \delta^{\eta}]$ and $B$ is a ball of radius $r \in [\delta^{-\eta}\rho, 1]$ whose intersection with $\bigcup_\tubes Y(T)$ has positive measure. Observe that as $\eta\searrow 0$, the set of admissible pairs $(B,\rho)$ becomes larger, and hence $\zeta( \tubes,Y; \eta)$ weakly decreases, i.e.~$\zeta( \tubes,Y; \eta)$ is a weakly increasing function of $\eta$. 

For $s,\eta,\delta>0$ and $u\geq 2$, define
\begin{equation}\label{defnOmegaSEtaDeltaU}
\omega(\eps,\eta,\delta;u)=\sup_{(\tubes,Y)_\delta} \zeta( \tubes,Y; \eta),
\end{equation}
where the supremum is taken over all choices of $(\tubes,Y)_\delta$ with the following properties:
\begin{itemize}
    \item $\tubes$ is a set of essentially distinct $\delta$-tubes contained in $B(0,1)$ that satisfies the Convex Wolff Axioms with error $\delta^{-\eps}$, and $Y$ is a $\delta^\eps$-dense shading.
    \item $\#\tubes \geq \delta^{-u+\eps}$.
\end{itemize}

Again, $\omega(\eps,\eta,\delta;u)$ is a weakly increasing function of $\eta$. It is also a weakly increasing function of $\eps$ and a weakly decreasing function of $u$, since as $\eps$  becomes larger or $u$ becomes smaller, the set of admissible pairs $(\tubes,Y)_\delta$ becomes larger.

Define
\[
\omega(u) = \lim_{\eps\to0^+}\ \lim_{\eta\to 0^+}\ \limsup_{\delta\to 0^+}\omega(\eps,\eta,\delta;u)= \inf_{\eps,\eta>0}\limsup_{\delta\to 0^+}\omega(\eps,\eta,\delta;u).
\]

\begin{defn}\label{defnOmega}
Define $\omega = \omega(2)$.
\end{defn}
Unwrapping the above definitions, we have the following.

\begin{lem}\label{roleOfOmega}
Let $\omega\geq 0$ be as in Definition~\ref{defnOmega}. For all $\eps>0$, there exists $\eta,\delta_0>0$ so that the following holds for all $\delta\in (0,\delta_0]$. Let $(\tubes,Y)_{\delta}$ be a set of essentially distinct tubes that satisfy the Convex Wolff Axioms with error $\delta^{-\eta}$, with $Y$ a $\delta^{\eta}$-dense shading.

Then $\bigcup_{\tubes}Y(T)$ has discretized Assouad dimension at least $3-\omega-\eps$ at scale separation $\delta^{-\eta}$.
\end{lem}

\begin{proof}
Select $\eta>0$ sufficiently small so that $\limsup_{\delta\to 0^+}\omega(\eta,\eta,\delta;2)\leq \omega+\eps/2.$ Then there exists $\delta_0>0$ so that $\omega(\eta,\eta,\delta;2)\leq\omega+\eps$ for all $\delta\in (0,\delta_0]$. The conclusion of Lemma \ref{roleOfOmega} now holds for this choice of $\eta$ and $\delta_0$.
\end{proof}

In light of Lemma \ref{roleOfOmega}, in order to prove Theorem \ref{mainThmDiscretized} it suffices to prove that $\omega=0$. 


\subsection{Defining $\alpha$}\label{definingAlphaSection}
Note that $\omega(u)$ is weakly monotone decreasing. We also have that $\omega(4)=0$, since if $\tubes$ is a set of $\delta^{-4}$ essentially distinct $\delta$-tubes in $B(0,1)$, then $\bigcup_\tubes T$ must have volume $\gtrsim 1$. Define
 \begin{equation}\label{defnAlphaEqn}
\alpha = \sup\{u\in [2,4] \colon \omega(u)=\omega(2)\}.
 \end{equation}
The quantity $\alpha$ has the following interpretation. We have defined $3-\omega$ to be the smallest possible discretized Assouad dimension of a collection of essentially distinct tubes that satisfy the Convex Wolff Axioms. Every such collection of tubes must have cardinality at least $\delta^{-2}$, and in general we might expect that larger collections of tubes should have larger discretized Assouad dimension. Any collection of essentially distinct tubes of cardinality substantially larger than $\delta^{-\alpha}$ that satisfies the Convex Wolff Axioms must have discretized Assouad dimension larger than $3-\omega$.

\subsection{The existence of Assouad-extremal Kakeya sets}
Now that we have defined $\omega$ and $\alpha$, we are ready to construct a (hypothetical) ``worst possible'' counter-example to Theorem \ref{mainThmDiscretized}, in the sense described at the beginning of Section \ref{AssouadDimensionSection}.

\begin{defn}\label{epsAssouadExtremalDefn}
For $\eps>0,$ we say a pair $(\tubes,Y)_\delta$ is \emph{$\eps$ Assouad-extremal} if it satisfies the following properties:
\begin{itemize}
    \item[(i)] $\#\tubes \geq\delta^{-\alpha+\eps}$.
    \item[(ii)] The tubes are contained in $B(0,1)$, essentially distinct, and satisfy the Convex Wolff Axioms with error $\delta^{-\eps}$.
    \item[(iii)] The shading $Y$ is $\delta^{\eps}$-dense. 
    \item[(iv)] For all $\rho,r\in [\delta,1]$ with $\rho\leq\delta^{\eps}r$ and all balls $B$ of radius $r$, we have
    \begin{equation}\label{volumeBdNearExtremal}
        \Big| B \cap N_\rho\Big( \bigcup_{T\in\tubes}Y(T)\Big)\Big| \leq (\rho/r)^{\omega-\eps}|B|.
    \end{equation}
\end{itemize}
\end{defn}

Note that if $\tubes$ is $\eps$ Assouad-extremal, then it is also $\eps'$ Assouad-extremal for all $\eps'\geq\eps$.

\begin{lem}\label{exitenceOfEpsExtremalKakeyaLem}
For all $\eps,\delta_0>0$, there exists $\delta\in(0,\delta_0]$ and an $\eps$ Assouad-extremal set of tubes $(\tubes,Y)_\delta.$
\end{lem}
\begin{proof}
Fix $u\geq \max(2, \alpha-\eps/2)$ with $\omega(u)=\omega(2)=\omega$. Since $\omega(s,\eta,\delta;u)$ is a weakly increasing function of $s$ and $\eta$, we have $\limsup_{\delta\to 0^+}\omega(\eps/2,\eps,\delta;u)\geq\omega(u)$. Thus there exists $\delta\in (0,\delta_0]$ with $\omega(\eps/2,\eps,\delta;u)\geq\omega-\eps/2$. 

In light of \eqref{defnOmegaSEtaDeltaU}, there exists a set of tubes and their associated shading $(\tubes,Y)_\delta$, so that the following holds:
\begin{itemize}
\item[(a)] $\#\tubes\geq\delta^{-u+\eps/2}\geq\delta^{-\alpha+\eps}$ (cf. Definition \ref{epsAssouadExtremalDefn}, Item (i)).
\item[(b)] The tubes are contained in $B(0,1)$, essentially distinct, and satisfy the Convex Wolff Axioms with error $\delta^{-\eps}$ (cf. Definition \ref{epsAssouadExtremalDefn}, Item (ii)).
\item[(c)] $Y$ is a $\delta^{\eps}$-dense shading (cf. Definition \ref{epsAssouadExtremalDefn}, Item (iii)).
\item[(d)] We have $\zeta( \tubes,Y; \eta)\geq \omega-\eps$. 
\end{itemize}
Thus by \eqref{defnTubesYEta}, for every pair of scales $\rho,r\in[\delta,1]$ with $\rho\leq\delta^{\eps}r$ and every ball $B$ of radius $r$, we have
\[
\zeta( \tubes,Y,B, \rho)\geq  \zeta( (\tubes,Y); \eta)\geq \omega-\eps.
\]
But this is precisely \eqref{volumeBdNearExtremal}, and thus $(\tubes,Y)_\delta$ satisfies Definition \ref{epsAssouadExtremalDefn}, Item (iv).
\end{proof}


\section{Convex Wolff Axioms and multi-scale structure}\label{selfSimSection}

In this section we will carefully execute the arguments described in Section \ref{proofSketchSection}. The main result is Proposition \ref{4WayDichotProp}, which says that under suitable hypotheses, for every collection of $\delta$-tubes $\tubes$, at least one of the following must occur: 
\begin{itemize}
	\item[(A)] $\tubes$ has discretized Assouad dimension close to 3 (as we will see later in this section, this is forced to happen if the tubes in $\tubes$ arrange themselves into planks).
	\item[(B)] $\tubes$ satisfies an analogue of stickiness.
	\item [(C)] A suitable thickening of $\tubes$ yields a richer collection of tubes.
	\item[(D)] A suitable zooming-in and rescaling of $\tubes$ yields a richer collection of tubes. 
\end{itemize}
The precise statement is as follows.

\begin{prop}\label{4WayDichotProp}
For all $\eps>0$, there exists $\tau>0$, so that for all $\eta_1>0$, there exists $\eta,\delta_0>0$ so that the following holds for all $\delta\in(0,\delta_0]$. Let $\beta\geq 2$ and let $(\tubes,Y)_\delta$ be a set of $\delta^{-\beta}$ essentially distinct $\delta$-tubes that satisfy the Convex Wolff Axioms with error $\delta^{\eta}$, with $Y$ a $\delta^\eta$-uniformly dense shading. Then at least one of the following must hold:
\begin{itemize}
\item[(A)] $\bigcup_{\tubes} Y(T)$ has discretized Assouad dimension at least $3-\eps$ at scale $\delta$, with scale separation $\delta^{-\eta}$. 
\item[(B)] There exists $\tubes'\subset\tubes$ with $\#\tubes'\geq\delta^{\eta_1}(\#\tubes)$, and $\tubes'$ satisfies the self-similar Convex Wolff Axioms with error $\delta^{-\eps}$.
\item[(C)] There exists $\rho\in[\delta,\delta^\eps]$ and a set $\tilde\tubes$ of $\rho$-tubes, with the following properties:
	\begin{itemize}
	\item[(C.i)] The tubes in $\tilde\tubes$ are essentially distinct, and $\#\tilde\tubes\geq\rho^{-\beta-\tau}$. 
	\item[(C.ii)] $\tilde\tubes$ satisfies the Convex Wolff Axioms with error $\leq\rho^{-\eta_1}$. 
	\item[(C.iii)] For each $T_\rho\in\tilde\tubes$, there exists $T\in\tubes$ with $T\subset T_\rho$.
	\end{itemize}

\item[(D)] There exists $\rho\in[\delta^{1-\eps},1]$, a $\rho$-tube $T_\rho$, and a set $\tilde\tubes\subset \tubes[T_\rho]$, so that the following holds: 
	\begin{itemize}
	\item[(D.i)]  $\# \tilde\tubes \geq(\delta/\rho)^{-\beta-\tau}.$ 
	\item[(D.ii)] The unit rescaling of $\tilde\tubes$ relative to $T_\rho$ satisfies the Convex Wolff Axioms with error $(\delta/\rho)^{-\eta_1}$. 
	\end{itemize}
\end{itemize}

\end{prop}

To prove Proposition \ref{4WayDichotProp}, we will consider arrangements $\tubes$ for which conclusions (B) and (C) fail. This will force the existence of a $\rho$-tube $T_\rho$ that satisfies Item (D.i). However, it is not clear that this $\rho$-tube will also satisfy Item (D.ii); the main concern is that there might exist many convex sets $W\subset T_\rho$ of dimensions $a\times b\times 1$ with $a<\!\!<b$ that contain an illegally large number of tubes from $\tubes[T_\rho]$ (convex sets of dimensions $a\times b\times 1$ with $a\sim b$ can be handled by replacing $\rho$ with a smaller scale). Our goal is to show that if this happens, then Conclusion (A) holds. To do this, we will show that if this happens, then the tubes in $\tubes[T_\rho]$ must cluster into planks; this step is straightforward and will be deferred to later. The second step is to show that under suitable hypotheses, every arrangement of planks must have discretized Assouad dimension close to 3; we will do this in the next sub-section. 


\subsection{Arrangements of rectangular prisms}\label{arrangeRecPrismSection}
We begin the process described at the end of the previous section. Let $\mathcal{R}$ be a set of planks, i.e.~rectangular prisms of dimensions roughly $s\times t\times 1$ for some $\delta\leq s\leq t\leq 1$. If a typical pair of planks from $\mathcal{R}$ intersects tangentially, then this forces the planks to cluster into wider (and somewhat thicker) planks; crucially, the aspect ratio (i.e. thickness vs. width) remains the same. If instead a typical pair of planks does not intersect tangentially, then the hairbrush of a typical plank $R\in\mathcal{R}$ (i.e.~the union of the other planks that intersect $R$) fills out a thickened neighbourhood of $R$, and this in turn means that the union of planks has discretized Assouad dimension close to 3. The precise statement is as follows.

\begin{lem}\label{biggerRectDichotomy}
For all $\eps,\tau\in (0,1]$, there exists $\delta_0,\eta>0$ so that the following is true for all $\delta\in(0,\delta_0]$. Let $\delta\leq s\leq t\leq 1$ with $s\leq\delta^{\eps}t$. Let $\mathcal{R}$ be a multi-set of $s\times t\times 1$ rectangular prisms contained in $B(0,1)$, and suppose that $\mathcal{R}$ satisfies the Convex Wolff Axioms with error $\delta^{-\eta}$. Let $\{Y(R), R\in\mathcal{R}\}$ be a $\delta^{\eta}$-dense shading.

\noindent Then at least one of the following two things must happen:

\begin{itemize}
	\item[(A)] $\bigcup_{\mathcal{R}}Y(R)$ has discretized Assouad dimension at least $3-\eps$ at scale $s$, with scale separation $\delta^{-\eta}$. This means that there exists $r\in[\delta^{-\eta}s,1]$ and a $r$-ball $B$ so that
\begin{equation}\label{plankRhoRNbhdSameVolInDichotomy}
\Big| B \cap  \bigcup_{R\in\mathcal{R}}Y(R)\Big| \geq (s/r)^{\eps}|B|.
\end{equation}

\item[(B)] There are scales $s'\in [\delta^{-\eps^2/4}s,\delta^{\eps/2}]$ and $t'=\min\big(1, t\frac{s'}{s}\big)$, and a multi-set $\mathcal{R}'$ of $s' \times t'\times 1$ prisms that satisfies the Convex Wolff Axioms with error $\delta^{-\tau}$, so that the shading 
\begin{equation}\label{YPrimeRPrimeShading}
Y'(R')=R'\cap N_{s'}\Big(\bigcup_{\substack{R\in\mathcal{R}\\ R \subset 2R'}}Y(R)\Big)
\end{equation}
is $\delta^{\tau}$-dense. 
\end{itemize}
\end{lem}

\begin{rem}
Note that as $\eps\searrow 0$, Conclusion (A) becomes stronger, while conclusion (B) becomes weaker (indeed, if $\eps=0$ then Conclusion (B) vacuously holds with $s'=s$ and $t'=t$). As $\tau\searrow 0$, Conclusion (A) is unchanged, while Conclusion (B) becomes stronger. 
\end{rem}

\begin{proof}$\phantom{1}$\\
{\bf Step 1: Robust broadness.} 
After replacing $\mathcal{R}$ by a subset of cardinality $\gtrsim\delta^{\eta}(\#\mathcal{R})$ if necessary, we may suppose that $|Y(R)|\geq \frac{1}{2}\delta^{\eta}|R|$ for each $R\in\mathcal{R}$. Let $Y_1(R)\subset Y(R)$ be a regular (in the sense of Definition \ref{regularShading}) refinement of $Y$, with $|Y_1(R)|\geq\frac{1}{2}|Y(R)|$ for each $R\in\mathcal{R}$. 


Next, we perform a ``two-ends'' style robust broadness reduction. Let $\eps_1>0$ be a small quantity to be chosen below. For each $x\in\bigcup_{R\in\mathcal{R}}Y_1(R)$, let $u_x\in [t, 1]$ and let $T_{u_x}$ be a $u_x$-tube that comes within a multiplicative factor of 2 of maximizing the quantity
\[
u_x^{-\eps_1}\#\{R\in\mathcal{R}\colon x\in Y_1(R),\ R\subset T_{u_x}\}.
\]
Since there exists a 1-tube that contains a $\gtrsim 1$ fraction of $\mathcal{R}$, and since $u_x\geq\delta$, we have
\[
\#\{R\in\mathcal{R}\colon x\in Y_1(R),\ R\subset T_{u_x}\}\gtrsim \delta^{\eps_1}\#\{R\in\mathcal{R}\colon x\in Y_1(R)\}.
\]

Define the shading
\[
Y_2(R)=\{x\in Y_1(R)\colon R\subset T_{u_x}\}.
\]
Then $\sum_{R\in\mathcal{R}}|Y_2(R)|\gtrsim \delta^{\eps_1}\sum_{R\in\mathcal{R}}|Y_1(R)|$. For each $x\in\bigcup_{R\in\mathcal{R}}Y_2(R)$, each $t\leq r\leq u_x$, and each $r$-tube $T_r$, we have
\begin{equation}\label{notTooMuchConcentrationInAnRTube}
\#\{R\in\mathcal{R}\colon x\in Y_2(R),\ R\subset T_r\}\leq 2(r/u_x)^{\eps_1}\#\{R\in\mathcal{R}\colon x\in Y_2(R)\}.
\end{equation}

By dyadic pigeonholing, we can select a multiplicity $\mu$ and a number $t\leq u\leq 1$ so that if we define
\[
Y_3(R) = Y_2(R) \cap \Big\{x \in\RR^3\colon \sum_{R\in\mathcal{R}}\chi_{Y_2(R)}(x)\in[\mu,2\mu),\ u_x \in [u/2, u)\Big\},
\]
then provided we choose $\delta_0$ sufficiently small and $\eta<\eps_1/2$, we have $\sum_{R\in\mathcal{R}}|Y_3(R)|\geq \delta^{2\eps_1}\sum_{R\in\mathcal{R}}|R|$.

We first consider the case where $\mu\leq\delta^{-\eps^2+2\eps_1}(st\#\mathcal{R})$. Then 
\[
\Big|\bigcup_{R\in\mathcal{R}}Y(R)\Big|\geq \mu^{-1}\sum_{R\in\mathcal{R}}|Y_3(R)|\geq \delta^{\eps^2-2\eps_1}\delta^{2\eps_1} \frac{\sum_{R\in\mathcal{R}}|R|}{st(\#\mathcal{R})}\geq\delta^{\eps^2}\geq s^{\eps},
\]
where the final inequality used the hypothesis that $s\leq\delta^{\eps}t$ and hence $s\leq\delta^\eps$. Thus Conclusion (A) holds with $r=1$ and $B$ the unit ball, and we are done.

Our choice of $\eps_1$ will satisfy $\eps_1\leq\eps^3/6$, and thus for the remainder of the proof we may suppose that
\begin{equation}\label{muFairlyLarge}
\mu>\delta^{-(2/3)\eps^2}(st\#\mathcal{R}).
\end{equation}

\medskip
\noindent {\bf Step 2: Localization.} 
Let $\tubes_u$ be a set of essentially distinct $u$ tubes with the property that for every $u$-tube $T_u$, there is at least one, and at most $O(1),$ tubes from $\tubes_u$ whose 2-fold dilate contains $T_u$. 
For each $x\in\bigcup_{R\in\mathcal{R}}Y_3(R)$, select a tube $T_u(x)\in\tubes_u$ with the property that $T_{u_x}\subset 2T_u$. Define the shading $Y(T_u) = \{x\in\RR^3\colon T_u(x) = T_u\}$. We have $\bigcup_{R\in\mathcal{R}}Y_3(R)=\bigsqcup_{T_u\in\tubes_u}Y(T_u)$. Note that technically, $Y(T_u)$ is a subset of $2T_u$, rather than a subset of $T_u$. Observe that if $R\in\mathcal{R}$ with $x\in Y_3(R)$, then $R\subset T_{u_x}\subset 2T_u(x)$.

Since the tubes in $\tubes_u$ are essentially distinct, for each $T_u\in\tubes_u$ there are $O(1)$ tubes $T_u'\in \tubes_u$ with the property that $2T_u\cap 2T_u'$ contains a unit line segment --- indeed, if $2T_u\cap 2T_u'$ contains a unit line segment, then $T_u$ and $T_u'$ must make angle $\lesssim u$. Thus by pigeonholing, we can replace $\tubes_u$ by a $O(1)$ refinement $\tubes_u'$ with the following properties:
\begin{itemize}
\item For each pair of distinct tubes $T_u,T_u'\in\tubes_u'$, $2T_u\cap 2T_u'$ does not contain a unit line segment. In particular, each $R\in\mathcal{R}$ is contained in at most one dilated tube $2T_u$.

\item
\[
\Big|\bigsqcup_{T_u\in\tubes_u'}Y(T_u)\Big|\gtrsim \Big|\bigsqcup_{T_u\in\tubes_u}Y(T_u)\Big| = \Big|\bigcup_{R\in\mathcal{R}}Y_3(R)\Big|.
\]
\end{itemize}

For each $R\in\mathcal{R}$, define $Y_4(R) = Y_3(R) \cap Y(T_u)$, where $T_u$ is the unique tube from $\tubes_u'$ for which $R\subset 2T_u$. If no such tube exists, then define $Y_4(R) = \emptyset$. We have
\[
\sum_{R\in\mathcal{R}}|Y_4(R)|\geq \mu \sum_{T_u\in\tubes_u'}|Y(T_u)|\gtrsim  \mu \sum_{T_u\in\tubes_u}|Y(T_u)|\geq\frac{1}{2}\sum_{R\in\mathcal{R}}|Y_3(R)|\geq \delta^{2\eps_1}\sum_{R\in\mathcal{R}}|R|.
\]

Finally, let $\tubes_u''\subset\tubes_u'$ consist of those $T_u\in\tubes_u'$ for which
\begin{equation}\label{keepcDelta2eps1Fraction}
\sum_{R\in\mathcal{R}[2T_u]}|Y_4(R)| \geq c\delta^{2\eps_1} \sum_{R\in\mathcal{R}[2T_u]}|R|.
\end{equation}
We will choose the constant $c\sim 1$ sufficiently small so that 
\[
\sum_{T_u\in\tubes_u''}\sum_{R\in\mathcal{R}[2T_u]}|Y_4(R)|\gtrsim \sum_{T_u\in\tubes_u'}\sum_{R\in\mathcal{R}[2T_u]}|Y_4(R)|.
\]

In summary, the situation is as follows: For each $T_u\in\tubes_u''$, we have 
\begin{equation}\label{denseInsideR}
\sum_{R\in\mathcal{R}[2T_u]}|Y_4(R)|\gtrsim\delta^{2\eps_1}\sum_{R\in\mathcal{R}[2T_u]}|R|, 
\end{equation}
and for each $x\in\bigcup_{R\in\mathcal{R}[2T_u]}Y_4(R),$ we have
\begin{equation}\label{allPrismsKeptInsideR}
\{R\in\mathcal{R}[2T_u]\colon x\in Y_4(R)\}=\{R\in\mathcal{R}\colon x\in Y_3(R)\}=\{R\in\mathcal{R}\colon x\in Y_2(R)\}.
\end{equation}
The cardinality of the above set is between $\mu$ and $2\mu$.

For each $T_u\in\tubes_u''$ and each $R\in\mathcal{R}[2T_u]$, the unit rescaling of $R$ relative to $2T_u$ is comparable to a $s/u\times t/u\times 1$ prism. Define $\tilde s = s/u$ and $\tilde t = t/u$. Abusing notation slightly, we will call this set a $\tilde s\times\tilde t\times 1$ prism, and we will denote the (multi) set of these prisms by $\tilde{\mathcal{R}}^{T_u}$. Let $\tilde Y_1$ (resp.~$\tilde Y_4$) be the images of the shadings $Y_1$ (resp.~$Y_4$) under the unit rescaling described above. Thus for each $T_u\in\tubes_u$, we have that $(\tilde{\mathcal{R}}^{T_u},\tilde Y_1)$ is a set of $\tilde s\times\tilde t\times 1$ prisms and a shading on these prisms, and similarly for $(\tilde{\mathcal{R}}^{T_u},\tilde Y_4)$. By \eqref{allPrismsKeptInsideR} and our definition of $Y_3$, for each $T_u\in\tubes_u$ and each $x\in\bigcup_{\tilde R\in\tilde{\mathcal{R}}}\tilde Y_4(\tilde R)$, we have
\begin{equation}\label{rescaledMultiplicityBd}
\mu\leq \#\{\tilde R\in\tilde{\mathcal{R}}^{T_u}\colon x\in \tilde Y_4(\tilde R)\}<2\mu.
\end{equation}

For each $\tilde R\in\tilde{\mathcal{R}}^{T_u}$, let $v(\tilde R)$ be the direction of a tube of minimal diameter that contains $\tilde R$. Thus $v(R)$ is the ``direction'' of the tube $\tilde R$. We have that $v(R)$ is meaningfully defined up to accuracy $\tilde t$, in the sense that if $\tilde R$ and $\tilde R'$ are comparable, then $\angle(v(\tilde R),\ v(\tilde R')) = O(\tilde t)$. 

As a consequence of \eqref{notTooMuchConcentrationInAnRTube}, we have that for each $x\in\bigcup_{\tilde R\in\tilde{\mathcal{R}}^{T_u}}\tilde Y_4(\tilde R)$, each $\tilde t \leq r\leq 1$, and each unit vector $v$, we have
\begin{equation}\label{robustBroadnessRescaled}
\#\{\tilde R\in\tilde{\mathcal{R}}^{T_u}\colon x\in \tilde Y_4(\tilde R),\ \angle(v(\tilde R),\ v)\leq r\}\lesssim r^{\eps_1}\mu.
\end{equation}

\medskip

\noindent {\bf Step 3: Dealing with transverse intersection.}  Fix a $u$-tube $T_u\in\tubes_u''$, and define $\tilde{\mathcal{R}}=\tilde{\mathcal{R}}^{T_u}$. Each $\tilde R\in\tilde{\mathcal{R}}$ is comparable to a rectangular prism of dimensions $\tilde s\times\tilde t\times 1$, i.e.~$\tilde R$ is comparable to the intersection of the following three sets: the $\tilde s$ neighbourhood of a plane $\Pi(\tilde R)$; the $\tilde t$ neighbourhood of a tube $T(\tilde R)$; and the unit ball. Let $n(\tilde R)$ be the unit normal vector of $\Pi(\tilde R)$; this is meaningfully defined up to accuracy $\tilde s/ \tilde t=s/t$, in the sense that if $\tilde R$ and $\tilde R'$ are comparable, then $\angle(n(\tilde R),\ n(\tilde R')) = O(\tilde s/\tilde t)$ (here $\angle(\cdot,\cdot)$ refers to unsigned angle, since we do not distinguish between the vectors $n(\tilde R)$ and $-n(\tilde R)$).

For each $x\in\bigcup_{\tilde R\in\tilde{\mathcal{R}}}\tilde Y_4(R)$, define the ``plane spread''
\[
\operatorname{P-spread}(x) = \max\Big\{\angle\big(n(\tilde R), n(\tilde R')\big)\Big\},
\]
where the maximum is taken over all pairs $\tilde R,\tilde R'$ for which $x\in \tilde Y_4(R)$ and $x\in \tilde Y_4(R)$. Since $n(\tilde R)$ and $n(\tilde R')$ are defined up to uncertainty $O(\tilde s/\tilde t)=O(s/t)$, $\operatorname{P-spread}(x)$ is meaningfully defined if it is substantially larger than $s/t$. If $x\not\in \bigcup_{\tilde R\in\tilde{\mathcal{R}}}\tilde Y_4(\tilde R)$, then define $\operatorname{P-spread}(x) =0$.

Suppose there exists a prism $\tilde R_0\in\tilde{\mathcal{R}}$ and a number $\theta\geq \delta^{-7\eps_1/\eps}(\tilde s /\tilde t\phantom{.})$ so that
\begin{equation}\label{largeSetLargeSpread}
\big|\{x \in \tilde R_0 \colon \theta\leq \operatorname{P-spread}(x)<2\theta\}\big|\geq \frac{c}{100}|\log\delta|^{-2}\delta^{2\eps_1}|\tilde R_0|,
\end{equation}
where $c\sim 1$ is the constant from \eqref{keepcDelta2eps1Fraction}.

Since $\tilde R_0$ is a $\tilde s\times \tilde t\times 1$ prism, for each $x_0\in \tilde R_0$ we have $|\tilde R_0 \cap B(x_0, \tilde t)|\sim \tilde s\tilde t\phantom{.}^2$. By pigeonholing, we can select a point $x_0\in \tilde R_0$ so that
\begin{equation}\label{largeSetLargeSpreadBall}
\big|\{x \in \tilde R_0 \cap B(x_0, \tilde t) \colon  \theta\leq \operatorname{P-spread}(x)<2\theta\}\big|\gtrapprox \delta^{2\eps_1}\tilde s \tilde t\phantom{.}^2.
\end{equation}

After a harmless translation and rotation, we may suppose that  $\Pi(\tilde R_0)$ is the $yz$-plane, i.e.~$\operatorname{span}\{e_2,e_3\}$, and that $\tilde R_0\cap B(x_0,\tilde t)$ is contained in the ``square'' $S=[-\tilde s, \tilde s]\times [-\tilde t, \tilde t\phantom{.}]\times[-\tilde t, \tilde t\phantom{.}]$ (note as well that $|S|\sim |\tilde R_0\cap B(x_0,\tilde t)|$). Let $S^{\dag}$ denote the thickened square $[-10\tilde t\theta, 10\tilde t\theta]\times [-10\tilde t,10\tilde t\phantom{.}]\times[-10\tilde t,10\tilde t\phantom{.}]$, i.e.~$S^{\dag}$ is comparable to the $\tilde t\theta$ neighborhood of $S$.

Denote the set on the LHS of \eqref{largeSetLargeSpreadBall} by $E$; then for each $x\in E$ there is a prism $\tilde R_x\in\mathcal{R}$ with $x\in \tilde Y_4(\tilde R)$ and $\angle(n(\tilde R_x), e_1)\in [\theta, 2\theta)$. This implies that $\tilde R_x \cap B(x,\tilde t) \subset \tilde R_x \cap B(x_0,10\tilde t)\subset S^{\dag}$. Recall that the shading $\tilde Y_1(\tilde R_x)$ is regular, and thus we have
\begin{equation}\label{lowerBdY2RxCap}
|\tilde Y_1(\tilde R_x)\cap B(x_0,2\tilde t)|\gtrapprox |\tilde Y_1(\tilde R_x)|\frac{|\tilde R_x\cap B(x_0,2\tilde t)|}{|\tilde R_x|}\gtrsim \delta^{2\eps_1}\tilde s\tilde t\phantom{.}^2.
\end{equation}

Define the linear map $\tilde A(x,y,z)=\Big( \frac{ x}{10\tilde t\theta},\ \frac{y}{10\tilde t},\ \frac{z}{10\tilde t}\Big)$, so $A(S^{\dag})=[-1,1]^3$. Define 
\[
\hat S=A(S)=[-\frac{\tilde s}{10\tilde t\theta},\frac{\tilde s}{10\tilde t\theta}]\times[-1/10,1/10]\times[-1/10,1/10],
\] 
and let $\hat E = A(E)\subset\hat S$. For each $x\in E$, let $\hat R_x = A(\tilde R_x\cap S^\dag)$; this is comparable to a rectangular prism of dimensions roughly $\frac{\tilde s}{\tilde t\theta}\times 1\times 1$, and this prism makes angle $\sim 1$ with the $yz$-plane, i.e.~there is a constant $c_0\sim 1$ (independent of the choice of $x\in E$) so that the prism has a normal direction $ n(\hat R_x)$ (defined up to uncertainty $O(\frac{\tilde s}{\tilde t\theta})$) with $\angle(n(\hat R_x), e_1)\geq c_0$. By \eqref{lowerBdY2RxCap} we have
\begin{equation}\label{lowerBdY2RxHat}
|\hat Y_2(\hat R_x)|\gtrapprox \delta^{2\eps_1}\frac{\tilde s}{\tilde t\theta}.
\end{equation}

By pigeonholing, we can find a unit vector $e\in\RR^3$ and a set $\hat E'\subset\hat E$ with $|\hat E'|\sim |\hat E|$, so that $\angle\big(e,  n(\hat R_x)\big)\leq c_0/2$ for each $x\in \hat E'$. In particular this means that $\angle(e,e_1)\geq c_0/2\gtrsim  1$,  and hence the projection of $e$ to the $yz$-plane has magnitude $\sim 1$; denote the image of $e$ under this projection by $e'$. 

\eqref{largeSetLargeSpreadBall} says that $|E|\gtrapprox \delta^{2\eps_1}\tilde s\tilde t\phantom{.}^2$ and hence $|\hat E'|\gtrsim |\hat E|\gtrapprox \delta^{\eta}\frac{s}{t\theta}$. On the other hand, $\hat E'\subset \hat S$, and the latter set has volume $\sim \frac{\tilde s}{\tilde t\theta}$. Since $\hat S=[-\frac{\tilde s}{10\tilde t\theta},\frac{\tilde s}{10\tilde t\theta}]\times[-1/10,1/10]\times[-1/10,1/10]$ and $e'$ is contained in the $yz$ plane, by Fubini we can select a line $\ell$ pointing in direction $e'$ that satisfies $|\ell\cap \hat E'|\gtrapprox \delta^{2\eps_1}$. 

Select a $\frac{\tilde s}{\tilde t\theta}$-separated set of points $x_1,\ldots,x_N$ (here $N\gtrapprox\delta^{2\eps_1}\frac{\tilde t\theta}{\tilde s}$)  from $|\ell\cap \hat E'|$. We will index these points so that $|x_i-x_j|\geq \frac{\tilde s}{\tilde t\theta}|i-j|$. Note that $\hat R_{x_i}\cap\ell$ is a line segment of length between $\frac{\tilde s}{\tilde t\theta}$ and $10\frac{\tilde s}{\tilde t\theta}$, and hence if $i,j$ is a pair of indices with $|x_i-x_j|\geq 100\frac{\tilde s}{\tilde t\theta}$, then $\operatorname{dist}(\hat R_{x_i}\cap\ell, \hat R_{x_j}\cap\ell)\gtrsim \frac{\tilde s}{\tilde t\theta}|i-j|$. Since the line $\ell$ points in direction $e'$, and $|e'\cdot  n(\hat R_{x_j})|\sim 1$, we conclude that $\operatorname{dist}(x_i, \hat R_{x_j})\gtrsim \frac{\tilde s}{\tilde t\theta}|i-j|$, so in particular $\hat R_{x_i}$ contains a point that has distance $\gtrsim \frac{\tilde s}{\tilde t\theta}|i-j|$ from $\hat R_{x_j}$. Thus either $\hat R_{x_i}$ and $\hat R_{x_j}$ do not intersect, or if they do intersect then their normal vectors satisfy $\angle\big( n(\hat R_{x_i}),\ n(\hat R_{x_j})\big)\gtrsim \frac{\tilde s}{\tilde t\theta}|i-j|.$ In either case, we have
\begin{equation}\label{rectangleIntersectionBdSize}
|\hat R_{x_i}\cap \hat R_{x_j}|\lesssim \frac{(\frac{\tilde s}{\tilde t\theta})^2}{\frac{\tilde s}{\tilde t\theta}(1+|i-j|)}=\frac{\tilde s}{\tilde t\theta(1+|i-j|)}.
\end{equation}
The above inequality vacuously holds if $i,j$ are a pair of indices with $|x_i-x_j|< 100\frac{\tilde s}{\tilde t\theta}$, since in this case $|i-j|\sim 1$ and we can use the trivial bound $|\hat R_{x_i}\cap \hat R_{x_j}|\leq |\hat R_{x_i}|$.\\

We conclude that
\begin{equation}\label{L2Bd}
\Big\Vert \sum_{i=1}^N\chi_{\hat Y_2(\hat R_i)}\Big\Vert_2^2 \leq \sum_{i,j=1}^N|\hat R_i\cap \hat R_j| \lesssim  N \frac{\tilde s}{\tilde t\theta} \big|\log\big(\frac{\tilde s}{\tilde t\theta}\big)\big|,
\end{equation}
and hence
\[
\Big|\bigcup_{i=1}^N \hat Y_2(\hat R_i)\Big| 
\geq \frac{\Big(\sum_{i}|\hat Y_2(\hat R_i)|\Big)^2}{\Big\Vert \sum_{i}\chi_{\hat Y_2(\hat R_i)}\Big\Vert_2^2}
\gtrapprox \frac{\Big(N \delta^{2\eps_1}\frac{\tilde s}{\tilde t\theta}\Big)^2}{N \frac{\tilde s}{\tilde t\theta} \big|\log\big(\frac{\tilde s}{\tilde t\theta}\big)\big|}
\gtrapprox \delta^{4\eps_1}N\frac{\tilde s}{\tilde t\theta}\gtrapprox\delta^{6\eps_1},
\]
where the second inequality used \eqref{lowerBdY2RxHat} and \eqref{L2Bd}. Undoing the scaling transformation $A$, we conclude that
\[
\Big| S^\dag\cap\bigcup_{\tilde R\in\tilde{\mathcal{R}}}\tilde Y_1(\tilde R)\Big|\gtrapprox \delta^{6\eps_1} |S^\dag|.
\]
Recall that $ S^\dag$ is a rectangular prism of dimensions roughly $\tilde t\theta \times \tilde t\times \tilde t$. 

Next, undo the unit rescaling with respect to $2T_u$; recall that this unit rescaling anisotropically dilated two directions by a multiplicative factor of roughly $u^{-1}$, and left the third direction unchanged. Thus the preimage of $S^\dag$ under this rescaling is comparable to a rectangular prisms whose smallest dimension is at least $\gtrsim \tilde t\theta u=t\theta\geq \delta^{-7\eps_1/\eps}s$.

By pigeonholing, we can find a ball $ B$ of radius $r\gtrsim \delta^{-7\eps_1/\eps}s$ contained in the preimage of $S^\dag$, so that
\begin{equation}\label{BBallDensityBd}
\Big| B\cap\bigcup_{R\in\mathcal{R}}Y(R)\Big| \gtrapprox \delta^{6\eps_1}|B|\geq \delta^{-\eps_1}(s/r)^\eps|B|.
\end{equation}
If $\delta_0$ is selected sufficiently small, then the LHS of \eqref{BBallDensityBd} is at least $(s/r)^\eps|B|$, and thus Conclusion (A) holds.

\medskip

\noindent {\bf Step 4: Ensuring that $u$ is large.}
At this point we may assume that for every $T_u\in\tubes_u''$, every prism $\tilde R_0\in\tilde{\mathcal{R}}^{T_u}$, and every $\theta\geq \delta^{-7\eps_1/\eps}(\tilde s/\tilde t)$, we have that \eqref{largeSetLargeSpread} fails. 

For each $T_u\in\tubes_u''$ and each $\tilde R\in\tilde{\mathcal{R}}^{T_u}$, define  
\[
\tilde Y_5(\tilde R)=\{x\in \tilde Y_4(R)\colon \operatorname{P-spread}(x)\leq \delta^{-7\eps_1/\eps}(\tilde s/\tilde t).
\]
For each $R\in\bigcup_{T_u\in\tubes_u''}\mathcal{R}[2T_u]$, define $Y_5(R)\subset Y_4(R)$ to be the preimage of $\tilde Y_5(\tilde R)$ under the unit rescaling associated to $2T_u$. Define $Y_5(R) = \emptyset$ if $R\not\in \bigcup_{T_u\in\tubes_u''}\mathcal{R}[2T_u]$. We have that 
\begin{itemize}
	\item For each $T_u\in\tubes_u,$ 

	\begin{equation}\label{densityOfY5}
	\sum_{R\in\mathcal{R}[2T_u]}|Y_5(R)|\gtrsim\sum_{R\in\mathcal{R}[2T_u]}|Y_4(R)|\gtrsim \delta^{2\eps_1}\sum_{R\in\mathcal{R}[2T_u]}|R|.
	\end{equation}
	\item $\sum_{R\in\mathcal{R}}|Y_5(R)|\gtrsim\sum_{R\in\mathcal{R}}|Y_4(R)|\gtrsim \delta^{2\eps_1}\sum_{R\in\mathcal{R}}|R|.$
	\item For each $T_u\in\tubes_u''$ and each $x\in\bigcup_{R\in\mathcal{R}[2T_u]}Y_5(R)$, we have
		\begin{equation}\label{numberOfPrismsThruPtY5}
		\mu\leq \#\{R\in\mathcal{R}[2T_u]\colon x\in Y_5(R)\}<2\mu.
		\end{equation}
\end{itemize}
We claim that 
\begin{equation}\label{uFairlyLarge}
u \geq \delta^{-\eps^2/4}t.
\end{equation}
To verify this claim, fix a choice of $T_u\in\tubes_u''$ and a choice of $x\in \bigcup_{R\in\mathcal{R}[2T_u]}Y_5(R)$. The prisms in the set
\eqref{numberOfPrismsThruPtY5} are contained in a prism of volume $\lesssim \delta^{-7\eps_1/\eps} s u^2/t$; this is because the unit rescaling of these prisms are contained in a rectangular prism of dimensions comparable to $\delta^{-7\eps_1/\eps} (\tilde s/\tilde t)\times 1\times 1$. This latter set has volume roughly $\delta^{-7\eps_1/\eps} s/t$, and undoing the unit rescaling distorts volume by a factor of $u^2$.

On the other hand by \eqref{uFairlyLarge} and \eqref{muFairlyLarge}, at least $\mu\geq \delta^{-(2/3)\eps^2}(st\#\mathcal{R})$ prisms from $\mathcal{R}$ must be contained in this set. Since $\mathcal{R}$ satisfies the Convex Wolff Axioms with error $\delta^{-\eta}$, we conclude that 
\[
\delta^{-(2/3)\eps^2}(st\#\mathcal{R}) \leq \delta^{-\eta}(\delta^{-7\eps_1/\eps}s u^2/t) (\#\mathcal{R}).
\]
re-arranging, we obtain \eqref{uFairlyLarge} provided we select $\eta$ and $\eps_1$ sufficiently small.

\medskip

\noindent {\bf Step 5: Constructing larger prisms.} 
Define $s_1 = \delta^{-7\eps_1/\eps}s$. Since $s\leq\delta^\eps t$, if we select $\eps_1$ sufficiently small (depending on $\eps$), then $ s_1\leq  t$.
For each prism $R\in \mathcal{R}_0$, let $\hat R$ be the prism of dimensions $ \frac{s_1u}{t}  \times u \times 1$ obtained by anisotropically dilating $R$. We claim that if $R_0,R\in\mathcal{R}$ are prisms and if $x\in Y_5(R_0)\cap Y_4(R)$, then $R\subset 2\hat R_0$. We verify this claim as follows: we have that $R_0$ and $R$ are contained in the 2-fold dilate of a common $u$-tube $T_u$, and furthermore $\angle(n(\tilde R_0), n(\tilde R))\leq s_1/t$ (here $\tilde R_0$ and $\tilde R$ are the unit rescalings of $R_0$ and $R$ with respect to $2T_u$). This means that $\tilde R$ is contained in the $s_1/t$ neighbourhood of the plane $\Pi(\tilde R_0)$; but after undoing the unit rescaling with respect to $2T_u$, the preimage of $N_{s_1/t}(B(0,1)\cap \Pi(\tilde R_0))$ is comparable to $\hat R$. 

In fact a superficially stronger statement is true --- if we define $w=\max\big(\frac{s_1 u}{t},t\big)$, then $N_{s_1 u/t}(R)$ is comparable to a prism of dimensions $\frac{s_1 u}{t}\times w\times 1$, and this prism is contained in $2\hat R_0$. Furthermore, $N_{s_1 u/t}(R)\cap R_0$ is comparable to a rectangular prism of dimensions $s\times t\times w/u$.

Let $\mathcal{R}_0\subset\mathcal{R}$ consist of those rectangular prisms for which $|Y_5(R)|\geq c_1 \delta^{2\eps_1}|R|$. We will choose $c_1\sim 1$ sufficiently small so that $\sum_{R\in\mathcal{R}_0}|Y_5(R)|\gtrsim \sum_{R\in\mathcal{R}}|Y_5(R)|$. Let $R_0\in\mathcal{R}_0$. Recall that $R_0$ is a $s\times t\times 1$ rectangular prism; for notational simplicity we will suppose $R_0 = [-s/2,s/2]\times[-t/2,t/2]\times [-1/2, 1/2]$, and that the intersection of $Y_5(R_0)$ with the $z$-axis has measure $\gtrapprox\delta^{\eta}$. Then $2\hat R_0= [-\frac{s_1u}{t}, \frac{s_1u}{t}]\times [-u,u]\times[-1,1]$.

Select $\frac{w}{u}$-separated points $z_1,\ldots,z_N\in Y_5(R_0)$ on the $z$-axis, with $N\gtrsim \delta^{2\eps_1}\frac{u}{w}$. We will label the points so that $|z_i-z_j|\geq \frac{w}{u}|i-j|$. For each index $i$, select a prism $R_i\in\mathcal{R}$ with $z_i\in Y_4(R_i)$ and $\angle((0,0,1),\ell(R))\in[c_2u, u)$. If we select $c_2\sim 1$ sufficiently small, then such a prism $R_i$ always exists by \eqref{rescaledMultiplicityBd} and \eqref{robustBroadnessRescaled}.

For each index $i$, let $R_i^\dag$ be the $\frac{s_1 u}{t}$ neighborhood of $R_i$, and let $Y_1^\dag( R_i^\dag) = N_{\frac{s_1u}{t}}(Y_1(R_i))$. Since $|Y_1(R_i)|\geq \frac{1}{2}\delta^{\eta}|R_i|$, we have 
\begin{equation}\label{tildeY3Full}
| Y_1^\dag(R_i^\dag)|\gtrsim \delta^{\eta}| R_i^\dag|.
\end{equation}

$R_i^\dag$ is comparable to a prism of dimensions $\frac{s_1 u}{t}\times w\times 1$ whose coaxial line makes angle $\sim u$ with the $z$ axis. We have 
\[
|R_i^\dag\cap R_j^\dag|\lesssim \big(\frac{s_1 u}{t}\big)\frac{w}{1+|i-j|},
\]
and hence
\[
\Big\Vert \sum_{i=1}^N \chi_{Y_1^\dag(R_i^\dag)}\Big\Vert_2^2 \leq \sum_{i,j=1}^N|R_i^\dag\cap R_j^\dag|\lesssim N\frac{s_1 u w}{t}|\log(u/w)|.
\]

We conclude that 
\begin{equation}\label{CordobaTypeComputation}
\Big|\bigcup_{i=1}^N  Y_1^\dag( R_i^\dag) \Big| 
\geq \frac{\Big(\sum  |Y_1^\dag( R_i^\dag)|\Big)^2}{\Big\Vert \sum \chi_{Y_1^\dag(R_i^\dag)}\Big\Vert_2^2}
\gtrapprox \frac{\Big(N \delta^{2\eps_1} \frac{s_1 u}{t} w\Big)^2}{N\frac{s_1 u w}{t}}
=N\delta^{4\eps_1}\frac{s_1 u}{t} w
\gtrapprox \delta^{6\eps_1}\frac{s_1}{t}u^2.
\end{equation}
On the other hand, each set $R_i'$ is contained in $\hat R_0$ (a set of volume roughly $s_1u^2/t$), and thus
\begin{equation}\label{hatR0Full}
\Big| 2\hat R_0 \cap N_{\frac{s_1 u}{t}}\Big(\bigcup_{\substack{R\in\mathcal{R} \\ R\subset \hat R_0}}Y(R)\Big)\Big| \gtrapprox \delta^{6\eps_1}|2\hat R_0|.
\end{equation}
Recall that $2\hat R_0$ is a prism of dimensions $2\frac{s_1 u}{t}  \times 2 u\times 2$, and $s_1 = \delta^{-7\eps_1/\eps}s$. Define $s'=\frac{s_1u}{t}$ and $t'=\min(1, t\frac{s'}{s})=\min(1, \delta^{-7\eps_1/\eps}u)$. By \eqref{uFairlyLarge}, we have $s'\in [\delta^{-\eps^2/4}s, \frac{s_1}{t}]$. Since $s\leq\delta^\eps t$, we can select $\eta$ sufficiently small so that $\frac{s_1}{t}\leq \delta^{\eps/2}$, and hence $s'\in [\delta^{-\eps^2/4}s, \delta^{\eps/2}]$.

By \eqref{hatR0Full} we conclude that for each $R_0\in\mathcal{R}_0$, if $\hat R_0'$ is the corresponding $s'\times t'\times 1$ prism and $\eps_1$ is selected sufficiently small compared to $\eps$ and $\tau$, then the shading $\hat Y'(\hat R_0')$ given by \eqref{YPrimeRPrimeShading} satisfies 
\begin{equation}\label{shadingOfYPrimeRPrime}
|\hat Y'(\hat R_0')|\gtrapprox \delta^{6\eps_1}|\hat R_0|\geq \delta^{6\eps_1+7\eps_1/\eps}|\hat R_0'|\geq\delta^\tau|\hat R_0'|.
\end{equation}
To finish the proof, define 
\[
\mathcal{R}'=\{\hat R_0'\colon \hat R_0\in\mathcal{R}_0\}.
\]
\eqref{shadingOfYPrimeRPrime} ensures that $|Y'(R')|\geq \delta^{\tau}|R'|$ for each $R'\in\mathcal{R}'.$ Since $\mathcal{R}$ satisfies the Convex Wolff axioms with error $\delta^{-\eta}$ and $\#\mathcal{R}_0\gtrapprox \delta^{2\eps_1}\#\mathcal{R}$, we have that $\mathcal{R}_0$ satisfies the Convex Wolff axioms with error $\lessapprox \delta^{-2\eps_1-\eta}$, and hence by Remark \ref{convexSupSet} we have that $\mathcal{R}'$ satisfies the Convex Wolff Axioms with error $\lessapprox \delta^{-2\eps_1-\eta}$. If $\eps_1,\eta,$ and $\delta_0$ and $\eta$ is selected sufficiently small, then this is $\leq\delta^{-\tau}$. 
\end{proof}

Soon, we will iterate Lemma \ref{biggerRectDichotomy} multiple times. If Conclusion (B) happens repeatedly, then eventually our planks will have width 1, i.e.~our planks will become rectangular slabs of dimensions $s\times 1\times 1$ for some $s<\!\!<1$. The next lemma says that the union of such a collection of slabs has volume close to 1.

\begin{lem}\label{thickenedPlanesAssouad}
Let $0<s \leq 1/2$ and $K\geq 1$. Let $\mathcal{R}$ be a multi-set of $s \times 1\times 1$ rectangular prisms, and suppose that $\mathcal{R}$ satisfies the Convex Wolff Axioms with error $K$. Let $\{Y(R), R\in\mathcal{R}\}$ be a $K^{-1}$-dense shading. 
Then 
\begin{equation}\label{platesLargeVolume}
\Big|  \bigcup_{R\in\mathcal{R}}Y(R)\Big| \gtrsim K^{-3}(\log 1/s)^{-1}.
\end{equation} 
\end{lem}

\begin{proof}
Let $E=\bigcup_{R\in\mathcal{R}}Y(R)$. Since the shading $Y$ is $K^{-1}$ dense, by Cauchy-Schwartz we have
\begin{equation}\label{CSTypeArgument}
K^{-1}s(\#\mathcal{R})\leq \int \chi_E \cdot \sum_{R\in\mathcal{R}}\chi_{Y(R)}\leq \int \chi_E \cdot \sum_{R\in\mathcal{R}}\chi_{R} \leq |E|^{1/2} \Big\Vert \sum_{R\in\mathcal{R}}\chi_R\Big\Vert_2.
\end{equation}
In the argument that follows, we will obtain an upper bound for $\big\Vert \sum_{\mathcal{R}}\chi_R\big\Vert_2$. 

For each prism $R\in\mathcal{R}$, there is an associated normal vector $v(R)\in S^2$, which is well-defined up to uncertainty $O(s)$. For $R,R'\in\mathcal{R}$, define $\angle(R,R')$ to be the (unsigned) angle between $v(R)$ and $v(R')$. 

 Fix $R\in\mathcal{R}$. Observe that if $R\cap R'$ is non-empty and $\angle(R,R')=\theta\sim 2^j s$, then $R'$ is contained in the $O(\theta)\times 1\times 1$ rectangular prism concentric with $R$. Since $\mathcal{R}$ obeys the Convex Wolff Axioms, at most $K(2^j s)(\#\mathcal{R})$ slabs $R'\in\mathcal{R}$ can have this property.

Let $j_1\sim\log(1/ s)$ be the smallest integer with $2^{j_1} s\geq 1$. Using the above observation, we compute
\begin{equation}\label{L2BdSlabs}
\begin{split}
\Big\Vert \sum_{R\in\mathcal{R}}\chi_R\Big\Vert_2^2  & = \sum_{R'\in\mathcal{R}}|R\cap R'| = \sum_{j=0}^{j_1}\sum_{\substack{R'\colon R\cap R'\neq\emptyset,\\ \angle(R,R')\sim 2^j s}}|R\cap R'|\\
& \lesssim \sum_{j=0}^{j_1}\Big(K (2^j s\#\mathcal{R}) \Big)\Big(\frac{s^2}{2^j s}\Big)\lesssim K s^2(\#\mathcal{R})\log(1/s).
\end{split}
\end{equation}
Chaining \eqref{CSTypeArgument} and \eqref{L2BdSlabs}, we obtain \eqref{platesLargeVolume}.
\end{proof}

The culmination of Section \ref{arrangeRecPrismSection} is the following result, which says that every arrangement of planks satisfying the Convex Wolff Axioms must have discretized Assouad dimension close to 3.

\begin{prop}\label{planksBigVolSomeScale}
For all $\eps>0$, there exists $\eta,\delta_0>0$ so that the following holds for all $\delta\in(0,\delta_0]$. Let $s,t\in [\delta,1]$ with $s\leq\delta^{\eps}t$. Let $\mathcal{R}$ be a multi-set of $s\times t\times 1$ prisms contained in $B(0,1)$, and suppose that $\mathcal{R}$ satisfies the Convex Wolff Axioms with error $\delta^{-\eta}$. Let $\{Y(R), R\in\mathcal{R}\}$ be a $\delta^{\eta}$-dense shading. 

\noindent Then $\bigcup_{\mathcal{R}}Y(R)$ has discretized Assouad dimension at least $3-\eps$ at scale $\delta$ and scale separation $\delta^{-\eta}$.
\end{prop}

In brief, we prove Proposition \ref{planksBigVolSomeScale} by iterating Lemma \ref{biggerRectDichotomy}. If Conclusion (A) holds at any step in the iteration, then our union of planks has discretized Assouad dimension close to 3, and we are done. If Conclusion (B) holds, then we repeat the process with the resulting wider collection of planks. Eventually, our planks will have width 1, at which point we can apply Lemma \ref{thickenedPlanesAssouad} to again conclude that our union of planks has discretized Assouad dimension close to 3. We now turn to the details.

\begin{proof}
Fix $\eps>0$ and let $\eta,\delta_0$ be small quantities to be determined below. Let $\delta\in(0,\delta_0]$, $s,t\in[\delta,1]$ with $s\leq\delta^{\eps}t$; let $\mathcal{R}$ be a set of $s\times t\times 1$ prisms that satisfy the Convex Wolff Axioms with error $\delta^{-\eta}$; and let $Y(R)$ be a $\delta^{\eta}$-dense shading.

Define $\mathcal{R}_1 = \mathcal{R}$ and define the shading $Y_1=Y$. Let $s_1=s$ and $t_1=t$. Let  $\eta=\tau_1<\tau_2<\tau_3<\ldots<\tau_N,$ $N = \lfloor 4/\eps^2\rfloor$ be small quantities to be specified below. We will perform the following iterative process for each $i=1,2,\ldots,N$. At the beginning of step $i$, the following properties will hold:
\begin{itemize}
\item[(a)] We have a set $\mathcal{R}_i$ of $s_i\times t_i\times 1$ prisms, with $s_i\leq \delta^{\eps}t_i$.
\item[(b)] $\mathcal{R}_i$ satisfies the Convex Wolff axioms with error $\delta^{-\tau_{i}}$.
\item[(c)] We have a shading $\{Y_i(R)\colon R\in\mathcal{R}_i\}$ that is $\delta^{\tau_{i}}$ dense.
\item[(d)] For each $R_i\in\mathcal{R}_i$, we have 
\begin{equation}\label{shadingRiProperty}
Y_i(R_i)\subset N_{s_i}\Big(\bigcup_{R\in\mathcal{R}}Y(R)\Big).
\end{equation}
\end{itemize}
When $i=1$, the hypotheses of Proposition \ref{planksBigVolSomeScale} guarantee that the above properties hold.

Suppose that the above properties hold at step $i$. Apply Lemma \ref{biggerRectDichotomy} to $(\mathcal{R}_i, Y_i)$ with $\eps$ as above and $\tau_{i+1}$ in place of $\tau$. If $\tau_i>0$ is sufficiently small depending on $\eps$ and $\tau_{i+1}$, then the Lemma can be applied (indeed, $\tau_i$ must be less than or equal to the quantity ``$\eta$'' from the conclusion of Lemma \ref{biggerRectDichotomy}). The quantity $\delta_0$ (and hence $\delta$) will be chosen sufficiently small so that Lemma \ref{biggerRectDichotomy} can be applied at each step.

\medskip

\noindent\textbf{Case (A).} Suppose that Conclusion (A) from Lemma \ref{biggerRectDichotomy} holds, i.e.~there exists $r\in [\delta^{-\tau_i}s_i, 1]$ and an $r$-ball $B$ so that 
\[
\Big| B \cap  \bigcup_{R\in\mathcal{R}_i}Y_i(R_i)\Big| \geq (s_i/r)^{\eps}|B|.
\]
Using \eqref{shadingRiProperty}, we conclude that
\[
\Big| B \cap  N_{s_i}\Big(\bigcup_{R\in\mathcal{R}}Y(R)\Big) \Big| \geq (s_i/r)^{\eps}|B|,
\]
i.e.~$\bigcup_{\mathcal{R}}Y(R)$ has discretized Assouad dimension $\geq 3-\eps$ at scale $\delta$ and scale separation $\delta^{-\tau_i}$. If this occurs, then we halt the iteration and conclude the proof.

\medskip

\noindent\textbf{Case (B).} Next, suppose that Conclusion (B) from Lemma \ref{biggerRectDichotomy} holds. Define $s_{i+1}=s_i'$ and $t_{i+1}=t_i'$ (here $s_i'$ and $t_i'$ are the output of Lemma \ref{biggerRectDichotomy}, Conclusion (B)). Let $\mathcal{R}_{i+1}$ be the set of $s_{i+1}\times t_{i+1}\times 1$ prisms, with shading $Y_{i+1}=Y_i'$, as described by Lemma \ref{biggerRectDichotomy}, Conclusion (B). We have that $\mathcal{R}_{i+1}$ satisfies the Convex Wolff Axioms with error $\delta^{-\tau_{i+1}}$, and $Y_{i+1}$ is $\delta^{\tau_{i+1}}$ dense. Furthermore, the shading $Y_{i+1}$ satisfies \eqref{shadingRiProperty} with $i+1$ in place of $i$. In particular, $\mathcal{R}_{i+1}$ and $Y_{i+1}$ satisfy our induction hypotheses (b), (c), and (d) above. 

\medskip

\noindent\textbf{Case (B.i).} Suppose that Conclusion (B) holds and that $t_{i+1}=1$. Recall that $s_{i+1}\leq\delta^{\eps/2}.$ Thus we can apply Lemma \ref{thickenedPlanesAssouad} with $s_{i+1}$ in place of $\delta$. We have
\begin{equation}\label{volumeBdStepiP1}
\Big|N_{s_{i+1}}\Big(\bigcup_{R\in\mathcal{R}}Y(R)\Big)\Big| \geq \Big|\bigcup_{R\in\mathcal{R}_{i+1}}Y_{i+1}(R)\Big|\gtrsim \delta^{3\tau_{i+1}}|\log s_{i+1}|^{-1}\gtrsim s_{i+1}^{7\tau_{i+1}/\eps}.
\end{equation}
We can select $\tau_{i+1}$ and $\delta_0$ sufficiently small so that the LHS of \eqref{volumeBdStepiP1} has size at least $s_{i+1}^\eps|B(0,1)|$. Since the set on the LHS of \eqref{volumeBdStepiP1} is contained in $B(0,1)$, we conclude that $\bigcup_{\mathcal{R}}Y(R)$ has discretized Assouad dimension $\geq 3-\eps$ at scale $\delta$ and scale separation $\delta^{-\eps/2}\geq\delta^{-\eta}.$ If this occurs, then we halt the iteration and conclude the proof.

\medskip

\noindent\textbf{Case (B.ii).} Suppose $t_{i+1}<1$; then $s_{i+1}/t_{i+1}=s_i/t_i$, so in particular $s_{i+1}\leq \delta^{\eps}t_{i+1}$. This means that the induction hypothesis (a) also holds, and thus we have verified all of the properties required to continue the iteration, with $i+1$ in place of $i$. Note that $s_{i+1}\geq\delta^{-\eps^2/4}s_i$; $t_{i+1}\geq\min(1, \delta^{\eps}s_{i+1})$, and the iterative process halts if $t_{i+1}=1$. We conclude that this iterative process will halt after at most $N = \lfloor 4/\eps^2\rfloor$ steps. 
\end{proof}


\subsection{Arrangements of tubes that cluster into rectangular prisms}
In Section \ref{arrangeRecPrismSection}, we showed that unions of planks have discretized Assouad dimension close to 3. In this section we will use these results to show that if an arrangement of $\delta$-tubes clusters into planks, then the union of these tubes must have discretized Assouad dimension close to 3. This is made precise in Lemmas \ref{convexWolffImpliesFewTubesInRectangleLem} and \ref{correctNumberTubesInRectLem} below. 

As a first step, the following lemma says that if $\tubes$ is a set of $\delta$ tubes of cardinality $>\!\!>\rho/\delta$, all of which are contained in a common $\delta\times\rho\times 1$ prism, then these tubes must cluster into $\delta\times\omega\times1$ planks (for some $\delta<\!\!<\omega\leq\rho$), and the union of these tubes must fill out a substantial fraction of each $\delta\times\omega\times 1$ plank. Later, we will show that the existence of such (mostly full) $\delta\times\omega\times 1$ planks will force $\bigcup_{\tubes}T$ to have discretized Assouad dimension close to 3. 

\begin{lem}\label{heavyRectangleYieldsFullRectangleLem}
For all $\eps>0$, there exists $\eta,\delta_0>0$ so that the following holds for all $\delta\in(0,\delta_0]$. Let $\rho\in[20\delta,1]$ and let $K\geq 1$. Let $(\tubes,Y)_\delta$ be a set of essentially distinct $\delta$-tubes, with $\#\tubes\geq K\rho/\delta$, and $Y$ a uniformly $\delta^\eta$-dense shading. Let $R_0$ be a $10\delta\times \rho \times 1$ rectangle, and suppose each $T\in\tubes$ is contained in $R_0$. 

Then there exists a number $K\delta\lesssim \omega\leq\rho$ and a set $\mathcal{R}$ of essentially distinct $10\delta\times\omega\times 1$ prisms, so that 
\begin{equation}\label{mostTubesInAPrism}
\#\Big(\bigcup_{R\in\mathcal{R}}\{T\in\tubes\colon T\subset R\}\Big)\gtrapprox\#\tubes,
\end{equation}
and for each prism $R\in\mathcal{R}$ we have
\begin{equation}\label{prismFull}
\Big|\bigcup_{\substack{T\in\tubes \\ T\subset R}}Y(T)\Big|\geq\delta^\eps|R|.
\end{equation}
\end{lem}

\begin{proof}
For each $x\in\RR^3$, define the ``line spread'' 
\[
\operatorname{L-spread}(x) = \max\big\{\angle(T,T')\big\},
\]
where the maximum is taken over all pairs of tubes $T,T'$ intersecting $x$. We define $\operatorname{L-spread}(x)=0$ if at most one tube intersects $x$. Let $C$ be a large number to be determined below. We claim that there exists a choice of $\theta\in [\frac{K\delta}{C},\rho]$ so that the shading 
\[
Z(T)=\{x\in T\colon \operatorname{L-spread}(x)\in [\theta,2\theta)\}
\] 
satisfies 
\begin{equation}\label{ShadingZFull}
\sum_{T\in\tubes}|Z(T)|\geq (C|\log\delta|)^{-1} \delta^2(\#\tubes).
\end{equation}
(Note that the shading $Z(T)$ is not necessarily a subset of $Y(T)$). We justify this claim as follows. Since there are at most $10\log(C/\delta)$ dyadic values of $\theta$ in the interval $[\frac{K\delta}{C},\rho]$, if we could not select a $\theta\in [\frac{K\delta}{C},\rho]$ so that \eqref{ShadingZFull} holds, then (provided $C$ is selected sufficiently large) the shading 
\[
Z'(T)=\Big\{x\in T\colon \operatorname{L-spread}(x)\leq \frac{K\delta}{C}\Big\}
\] 
satisfies $\sum_{\tubes}|Z'(T)|\geq \frac{1}{100}\delta^2(\#\tubes)$. But for each $x\in\RR^3$, the set of tubes $T\in\tubes$ with $x\in Z(T)$ are contained in the $10\delta$-neighborhood of a plane (since they are contained in $R_0$), and also make angle $\leq \frac{K\delta}{C}$ with a common unit vector; since the tubes in $\tubes$ are essentially distinct, we conclude that $O(K/C)$ tubes $T\in\tubes$ can satisfy $x\in Z'(T)$, and hence
\begin{equation}\label{volumeBdYPrimeT}
\Big| \bigcup_{T\in\tubes}T\Big| \geq \Big| \bigcup_{T\in\tubes}Z'(T)\Big| \gtrsim (C/K)\sum_{T\in\tubes}|Z'(T)|\gtrsim C \rho\delta.
\end{equation}
If $C$ is selected sufficiently large then this is impossible, since the set on the LHS of \eqref{volumeBdYPrimeT} is contained in $R_0$, which has volume $10\rho\delta$.

At this point we have established the existence of an angle $\theta\in [\frac{K\delta}{C},\rho]$ so that the associated shading $Z(T)$ satisfies \eqref{ShadingZFull}. In the arguments that follow, the quantity $C$ has been fixed, and all implicit constants are allowed to depend on $C$. Let 
\begin{equation}\label{tubes1}
\tubes_1=\big\{T\in\tubes\colon |Z(T)|\geq (C|\log\delta|)^{-1} |T|\big\}.
\end{equation}
Then \eqref{ShadingZFull} guarantees that $\#\tubes_1\gtrapprox\#\tubes$.  For each $T_1\in \tubes_1$, let $x_2,\ldots,x_N,\ N\gtrapprox \theta/\delta$ be a set of $\delta/\theta$-separated points from $Y_1(T_1)$. For each $x_i$, select $T_i\in\tubes$ with $x_i\in T_i$ and $\angle(T_1,T_i)\in[\theta,2\theta)$. An $L^2$ argument (see e.g.~\eqref{L2Bd}) plus the fact that $|Y(T_i)|\geq \delta^\eta|T|$ shows that
\begin{equation}\label{tubesNearT1}
\Big|\bigcup_{i=2}^N Y(T_i)\Big|\gtrapprox \delta^{2\eta}\sum_{i=2}^N|T_i| \gtrsim \delta^{2\eta}\delta\theta.
\end{equation}
On the other hand, the set on the LHS of \eqref{tubesNearT1} is contained in $R_0\cap N_{2\theta}(T_1)$, which in turn is contained in a rectangular prism of dimensions $10\delta\times 4\theta\times 1$; denote this prism by $R(T_1)$. We have just shown that
\begin{equation}\label{prismMostlyFullIneq}
\Big|\bigcup_{\substack{T\in\tubes \\ T\subset R(T_1)}}Y(T)\Big|\gtrapprox \delta^{2\eta}|R(T_1)|.
\end{equation}

Finally, let $\omega=4\theta$ and let $\mathcal{R}_1=\{R(T_1)\colon T_1\in\tubes_1\}$. We claim that \eqref{mostTubesInAPrism} holds for $\mathcal{R}_1$. Indeed, this follows from the fact that $T_1\subset R(T_1)$ for each $T_1\in\tubes_1$, and $\#\tubes_1\gtrapprox\#\tubes$. By \eqref{prismMostlyFullIneq} we see that \eqref{prismFull} holds for each $R\in\mathcal{R}_1$, provided we select $\eta$ and $\delta_0$ sufficiently small. To conclude the proof, greedily select a set $\mathcal{R}\subset\mathcal{R}_1$ of essentially distinct rectangular prisms; we can perform this selection so that \eqref{mostTubesInAPrism} continues to hold.
\end{proof}


The following lemmas says that if $\tubes$ clusters into planks, then it must have discretized Assouad dimension close to 3. 

\begin{lem}\label{convexWolffImpliesFewTubesInRectangleLem}
For all $\eps>0$, there exists $\eta,\delta_0>0$ so that the following holds for all $\delta\in(0,\delta_0]$. Let $\tubes$ be a set of essentially distinct $\delta$-tubes that satisfies the Convex Wolff Axioms with error $\delta^{-\eta}$. Let $Y$ be a uniformly $\delta^\eta$-dense shading. Then at least one of the following must occur:
\begin{itemize}
	\item[(A)] $\bigcup_{\tubes} Y(T)$ has discretized Assouad dimension at least $3-\eps$ at scale $\delta$ and scale separation $\delta^{-\eta}$. 

	\item[(B)] There is a set $\tubes'\subset\tubes$ with $\#\tubes'\geq\frac{1}{2}(\#\tubes)$ so that for every $\delta\leq\rho\leq 1$ and every $\delta\times\rho\times 1$ rectangle $R$, at most $\delta^{-\eps}(\rho/\delta)$ tubes from $\tubes'$ are contained in $100R$. 
\end{itemize}
\end{lem}
\begin{rem}
Since $\tubes$ satisfies the Convex Wolff Axioms, we have that at most $100^3\delta^{-\eta}(\rho\delta)(\#\tubes)$ tubes are contained in each $100\delta\times 100\rho\times 100$ rectangle. If $\#\tubes$ is substantially larger than $\delta^{-2}$, then the estimate from Conclusion (B) is stronger. 
\end{rem}
\begin{proof}
We say a $100\delta\times100 \rho\times 100$ rectangle is ``heavy'' if it contains more than $\delta^{-\eps}(\rho/\delta)$ tubes from $\tubes$. If fewer than $\frac{1}{2}(\#\tubes)$ of the tubes are contained in heavy rectangles, let $\tubes'$ be the set of tubes that are not contained in heavy rectangles. Then Conclusion (B) holds, and we are done.

Suppose instead that at least $\frac{1}{2}(\#\tubes)$ tubes are contained in heavy rectangles. Since the tubes in $\tubes$ are essentially distinct, each $100\delta\times100 \rho\times 100$ rectangle can contain at most $O\big((\rho/\delta)^2\big)$ tubes from $\tubes$. Thus if a $100\delta\times100 \rho\times 100$ rectangle is heavy, then we must have $\rho\gtrsim\delta^{1-\eps}$. Let $\mathcal{R}_0$ denote this set of heavy rectangles (note that the rectangles in $\mathcal{R}_0$ can have different dimensions). Let $\eps_1$ be a small quantity to be chosen below. Apply Lemma \ref{heavyRectangleYieldsFullRectangleLem} to each rectangle $R_0\in\mathcal{R}_0$, with $\eps_1$ in place of $\eps$, and $K=\delta^{-\eps}$. We obtain a number $\omega(R_0)\gtrsim \delta^{1-\eps}$ and a set $\mathcal{R}(R_0)$ of $\delta\times\omega(R_0)\times O(1)$ rectangles. We have
\begin{equation}\label{mostTubesInRectangles}
\#\Big(\bigcup_{R_0\in\mathcal{R}_0}\bigcup_{R\in\mathcal{R}(R_0)}\{T\in\tubes\colon T\subset R\}\Big)\gtrapprox\#\tubes.
\end{equation}
After dyadic pigeonholing, we can select $\omega_0\in[\delta^{1-\eps},1]$ so that if we define $\mathcal{R}_1=\{R_0\in\mathcal{R}_0\colon \omega(R_0)\in[\omega_0, 2\omega_0)\}$, then \eqref{mostTubesInRectangles} continues to hold with $\mathcal{R}_1$ in place of $\mathcal{R}_0$. 

Define the multi-set
\[
\mathcal{R}_2=\bigsqcup_{R_1\in\mathcal{R}_1}\mathcal{R}(R_1).
\]
(In the above, $\bigsqcup$ is used to suggest that if the same rectangle is present in multiple sets $\mathcal{R}(R_1)$, then this rectangle should occur multiple times in $\mathcal{R}_2$; i.e.~the union might create a multi-set). Let
\[
\tubes_2=\bigcup_{R\in\mathcal{R}_2}\{T\in\tubes\colon T\subset R\}.
\]
Since $\tubes$ satisfies the Convex Wolff Axioms with error $\delta^{-\eta}$, by \eqref{mostTubesInRectangles} we have that the set (i.e.~not multi-set) $\tubes_2$ satisfies the Convex Wolff Axioms with error $\lessapprox \delta^{-\eta}$. Finally, for each $T\in\tubes_2$, let $R(T)$ be some rectangle in $\mathcal{R}_2$ with $T\subset R$. Define the multi-set $\mathcal{R}=\{R(T)\colon T\in\tubes_2\}$, i.e.~$\mathcal{R}$ is a multi-set with the same cardinality as $\tubes_2$. By Remark \ref{convexSupSet}, we conclude that $\mathcal{R}$ satisfies the Convex Wolff Axioms with error $\lessapprox\delta^{-\eta}$.  

For each $R\in\mathcal{R}$, define the shading
\[
Y(R)=R\cap\bigcup_{T\in\tubes}Y(T).
\]
By Lemma \ref{heavyRectangleYieldsFullRectangleLem}, we have $|Y(R)|\geq\delta^{\eps_1}|R|$ for each $R\in\mathcal{R}$. To conclude the proof, apply Proposition \ref{planksBigVolSomeScale} with $\eps$ as above, $s=\delta$, and $t=\omega_0$; to ensure that Proposition \ref{planksBigVolSomeScale} can be applied, we must select $\eps_1$ sufficiently small depending on $\eps$ (our choice of $\eps_1$ depends on the quantity ``$\eta$'' from Proposition \ref{planksBigVolSomeScale}), and then $\eta$ sufficiently small depending on $\eps_1$. Since $\bigcup Y(T)\subset \bigcup Y(R),$ we conclude that Conclusion (A) holds. 
\end{proof}

Lemma \ref{convexWolffImpliesFewTubesInRectangleLem} says that either a union of $\delta$-tubes $\tubes$ has large discretized Assouad dimension, or the tubes in $\tubes$ cannot concentrate into $\delta\times\rho\times 1$ rectangular prisms. The next lemma is similar, except the latter conclusion is strengthened: either a union of $\delta$-tubes has large discretized Assouad dimension, or these tubes cannot concentrate into $s\times t\times 1$ rectangular prisms, for any $\delta\leq s\leq t\leq 1$. To prove this result, we apply Lemma \ref{convexWolffImpliesFewTubesInRectangleLem} to the $s$-thickening of the tubes in our arrangement, for many different values of $s$. The precise statement is as follows. 

\begin{lem}\label{correctNumberTubesInRectLem}
For all $\eps=1/N>0$, there exists $\eta,\delta_0>0$ so that the following holds for all $\delta\in(0,\delta_0]$. Let $(\tubes,Y)_\delta$ be a set of essentially distinct $\delta$-tubes that satisfies the Convex Wolff Axioms with error $\delta^{-\eta}$, and let $Y$ be a uniformly $\delta^{\eta}$-dense shading. Then at least one of the following must occur:
\begin{itemize}
	\item[(A)] $\bigcup_\tubes Y(T)$ has discretized Assouad dimension at least $3-\eps$ at scale $\delta$ and scale separation $\delta^{-\eta}$. 

	\item[(B)] There is a set $\tubes'\subset\tubes$ with $\#\tubes'\geq \delta^{\eps} (\#\tubes)$ so that the following holds. For every number $\delta\leq s\leq 1$ of the form $\delta^{i\eps}$, there is a uniform partitioning cover $\tubes_s$ of $\tubes'$ (see Definition \ref{coversDefn}), where $\tubes_s$ is a set of essentially distinct $s$-tubes. Furthermore, for every $t\in[s,1]$ and every $s\times t\times 1$ rectangular prism $R$, at most $\delta^{-\eps}\frac{t}{s}\frac{(\#\tubes')}{(\#\tubes_s)}$ tubes from $\tubes'$ are contained in $100R$. 
\end{itemize}
\end{lem}

\begin{proof}
\textbf{Step 1.}\\
Our first step is to regularize the set $\tubes$ at scales of the form $s_i=\delta^{i/N}$, with $i=1,\ldots,N$ (we do not need to consider $s_0 = 1$, since in that case there is nothing to prove). After repeated pigeonholing, we can select a set $\tubes^\dag \subset\tubes$ with $\#\tubes^\dag\geq (100\log(1/\delta))^{-N}(\#\tubes)$ and sets of $s_i$-tubes $\tubes_{s_i},\ i=0,\ldots,N$, so that $\tubes_{s_N}=\tubes^\dag$, and for each $i=1,\ldots,N-1$, the set $\tubes_{s_i}$ is a uniform partitioning cover of $\tubes_{s_{i+1}}$ (and hence also a uniform partitioning cover of $\tubes_{s_j}$ for all $j>i$). See e.g.~Lemma 3.4 from \cite{KS} for an explicit description of this iterated pigeonholing process. We will choose $\delta_0$ sufficiently small so that $\#\tubes^\dag\geq\delta^{\eta/2}(\#\tubes)$. In particular, by Remark \ref{consequenceOFUnifCover} this means that $\tubes^\dag$ and each set $\tubes_{s_i},\ i=1,\ldots,N$ satisfies the Convex Wolff Axioms with error $\delta^{-2\eta}\leq s_i^{-2\eta N}$. For each index $i$ and each $T_{s_i}\in\tubes_{s_i}$, define the shading $Y_{s_i}(T_{s_i})=T_{s_i}\cap N_{s_i}(Y(T))$, where $T$ is a tube from $\tubes[T_{s_i}]$. In particular, $|Y_{s_i}(T_{s_i})|\gtrsim \delta^\eta|T_{s_i}|\gtrsim s_i^{\eta N}|T_{s_i}|$, i.e. $Y_{s_i}$ is a uniformly $s_i^{\eta N}$-dense shading of $\tubes_{s_i}$.

\medskip

\noindent \textbf{Step 2.}\\
For each $i=1,\ldots,N$, apply Lemma \ref{convexWolffImpliesFewTubesInRectangleLem} to $(\tubes_{s_i},Y_{s_i})_{s_i}$ with $\eps$ as above; we can do this provided $\eta>0$ and $\delta_0$ are selected sufficiently small depending on $\eps$ (recall that $N=1/\eps$). If Conclusion (A) of Lemma \ref{convexWolffImpliesFewTubesInRectangleLem} holds, then $\bigcup_{\tubes_{s_i}}Y_{s_i}(T_{s_i})$ has discretized Assouad dimension at least $3-\eps$ at scale $s_i$ and scale separation $s_i^{-\eta'}$, for some $\eta'$ depending only on $\eps$. This in turn implies that $\bigcup_{\tubes}Y(T)$ has discretized Assouad dimension at least $3-\eps$ at scale $\delta$ (this is true since $\delta\leq s_i$) and scale separation $\delta^{-\eta'/N}$. Thus Conclusion (A) of Lemma \ref{correctNumberTubesInRectLem} holds, provided we select $\eta\leq \eta'/N = \eta'\eps$. If this happens then we are done.

Suppose instead that Conclusion (B) of Lemma \ref{convexWolffImpliesFewTubesInRectangleLem} holds. This gives us a set $\tubes_{s_i}'\subset \tubes_{s_i}$  with $\#\tubes_{s_i}'\geq\frac{1}{2}(\#\tubes_{s_i})$ so that for every $s_i\leq t \leq 1$ and every $s_i\times t\times 1$ rectangular prism $R$, at most $s_i^{-\eps}(t/s_i)\leq\delta^{-\eps}(t/s_i)$ tubes from $\tubes_{s_i}'$ are contained in $100R$. Abusing notation slightly, we will replace each set $\tubes_{s_j},\ j>i$ with $\bigcup_{\tubes_{s_i}'}\tubes_{s_j}[T_{s_i}]$, i.e.~we only keep those tubes from $\tubes_{s_j}$ that are contained in some tube from $\tubes_{s_i}'$. This process discards at most half the tubes from $\tubes_{s_j}$, so at each stage in this process, $\tubes_{s_j}$ still satisfies the convex Wolff axioms with error at most $2^N s_j^{-2\eta N}$ and hence we can continue to iterate Step 2.

\medskip

\noindent \textbf{Step 3.}\\
Suppose that Conclusion (B) of Lemma \ref{convexWolffImpliesFewTubesInRectangleLem} holds for each $i=1,\ldots,N$ in Step 2. While each set $\tubes_{s_i}$ is a uniform partitioning cover of $\tubes^\dag$, this property might no longer hold for the sets $\tubes_{s_i}'$. We will fix this with another round of pigeonholing. This gives us sets $\tubes_{s_i}'\subset\tubes_{s_i}$, $i=1,\ldots,N$ and a set $\tubes'\subset\tubes^\dag$, so that $\#\tubes' \geq (100 \log(1/\delta))^{-N}(\#\tubes^\dag)$; each set $\tubes_{s_i}'$ is a uniform partitioning cover of $\tubes'$; and for each index $i$ and each $s_i\times t\times 1$ prism $R$ we have
\[
\#\tubes'[100R] \leq \#\tubes_{s_i}'[100R] \frac{\#\tubes'}{\#\tubes_{s_i}'} \leq \delta^{-\eps}\frac{t}{s_i}\frac{(\#\tubes')}{(\#\tubes_{s_i}')}.
\]
Thus $\tubes'$ and $\tubes_{s_i}',\ i=1,\ldots,N$ satisfy Conclusion (B) of Lemma \ref{correctNumberTubesInRectLem}.
\end{proof}

\subsection{Proof of Proposition \ref{4WayDichotProp}} 
We are now ready to prove Proposition \ref{4WayDichotProp}. For the reader's convenience, we reproduce it here
\begin{4wayDicotPropEnv}
For all $\eps>0$, there exists $\tau>0$, so that for all $\eta_1>0$, there exists $\eta,\delta_0>0$ so that the following holds for all $\delta\in(0,\delta_0]$. Let $\beta\geq 2$ and let $(\tubes,Y)_\delta$ be a set of $\delta^{-\beta}$ essentially distinct $\delta$-tubes that satisfy the Convex Wolff Axioms with error $\delta^{\eta}$, with $Y$ a $\delta^\eta$-uniformly dense shading. Then at least one of the following must be true:
\begin{itemize}
\item[(A)] $\bigcup_\tubes Y(T)$ has discretized Assouad dimension at least $3-\eps$ at scale $\delta$, with scale separation $\delta^{-\eta}$. 
\item[(B)] There exists $\tubes'\subset\tubes$ with $\#\tubes'\geq\delta^{\eta_1}(\#\tubes)$, and $\tubes'$ satisfies the self-similar Convex Wolff Axioms with error $\delta^{-\eps}$.
\item[(C)] There exists $\rho\in[\delta,\delta^\eps]$ and a set $\tilde\tubes$ of $\rho$-tubes, with the following properties:
	\begin{itemize}
	\item[(C.i)] The tubes in $\tilde\tubes$ are essentially distinct, and $\#\tilde\tubes\geq\rho^{-\beta-\tau}$. 
	\item[(C.ii)] $\tilde\tubes$ satisfies the Convex Wolff Axioms with error $\rho^{-\eta_1}$. 
	\item[(C.iii)] For each $T_\rho\in\tilde\tubes$, there exists $T\in\tubes$ with $T\subset T_\rho$.
	\end{itemize}

\item[(D)] There exists $\rho\in[\delta^{1-\eps},1]$, a $\rho$-tube $T_\rho$, and a set $\tilde\tubes\subset \tubes[T_\rho]$, so that the following holds: 
	\begin{itemize}
	\item[(D.i)]  $\# \tilde\tubes \geq(\delta/\rho)^{-\beta-\tau}.$ 
	\item[(D.ii)] The unit rescaling of $\tilde\tubes$ relative to $T_\rho$ satisfies the Convex Wolff Axioms with error $(\delta/\rho)^{-\eta_1}$. 
	\end{itemize}
\end{itemize}
\end{4wayDicotPropEnv}
\begin{proof}
{\bf Step 1.}\\
Decreasing $\eta_1$ if necessary, we may may suppose that $\eta_1\leq 1$. Fix a number $\eta_2\leq \eta_1\eps/40$ of the form $\eta_2=1/N>0$. Apply Lemma \ref{correctNumberTubesInRectLem} with $\eta_2$ in place of $\eps$. If Conclusion (A) of  Lemma \ref{correctNumberTubesInRectLem} holds, then Conclusion (A) of Proposition \ref{4WayDichotProp} holds (since $\eta_2\leq\eps$) and we are done. If instead Conclusion (B) of  Lemma \ref{correctNumberTubesInRectLem} holds, then there exists:
\begin{itemize}
\item A set $\tubes'\subset\tubes$,
\item For each number $\rho$ of the form $\delta^{i/N}$, a uniform partitioning cover $\tubes_\rho$ of $\tubes'$.
\end{itemize}
These objects have the properties described in Conclusion (B) of Lemma \ref{correctNumberTubesInRectLem}.

Suppose there exists a scale $\rho\in[\delta,\delta^\eps]$ of the form $\rho=\delta^{i/N}$ so that $\#\tubes_\rho\geq\delta^{-\beta-\tau}$. Then by Remark \ref{consequenceOFUnifCover}, $\tubes_\rho$ satisfies the Convex Wolff Axioms with error $\delta^{-2\eta}\leq\rho^{-2\eta/\eps}\leq \rho^{-\eta_1}$. Thus Conclusion (C) holds, with $\tilde\tubes=\tubes_\rho$. 

Henceforth we shall suppose that for each scale $\rho\in[\delta,\delta^\eps]$ of the form $\rho=\delta^{i/N}$, we have
\begin{equation}\label{smallCardTubesRho}
\#\tubes_\rho\leq\delta^{-\beta-\tau}.
\end{equation}

\medskip

\noindent {\bf Step 2.}
Our goal in this step is to show that either Conclusion (D) of Proposition \ref{4WayDichotProp} holds, or for every $\rho\in[\delta,\delta^\eps]$ of the form $\rho=\delta^{i/N}$, we have
\begin{equation}\label{boundOnFineTubesInsideRhoTube}
  \#\tubes_\rho \geq \rho^{-\beta+\eps/3},\quad\rho\in[\delta,\delta^{\eps}].
\end{equation}
First, we consider the case where $\beta\leq 2+\eps/4$; in this case \eqref{boundOnFineTubesInsideRhoTube} will always hold. Indeed, let $\rho\in[\delta,\delta^\eps]$ be of the form $\rho=\delta^{i/N}$. Then for each $T_\rho\in\tubes_\rho$, we have
\[
\frac{\#\tubes'}{\#\tubes_\rho} =\#\tubes'[T_\rho]\leq \delta^{-2\eta}\rho^2\#\tubes',
\]
where the second inequality follows from the fact that $\tubes'$ satisfies the Convex Wolff Axioms with error $\delta^{-2\eta}$. We conclude that $\#\tubes_\rho \geq \delta^{2\eta}\rho^{-2}$. In particular, if $\beta\leq 2+\eps/4$, and if we select $\eta$ sufficiently small, then 
\begin{equation}\label{TRhoBigIfBetaClose2}
\#\tubes_\rho\geq \delta^{2\eta}\rho^{-2}\geq\rho^{2\eta/\eps +\eps/4}\rho^{-\beta}\geq \rho^{-\beta + \eps/3},
\end{equation}
and hence \eqref{boundOnFineTubesInsideRhoTube} holds. 

Next we consider the case where $\beta> 2+\eps/4$. Suppose there exists some $\rho\in[\delta,\delta^\eps]$ of the form $\rho=\delta^{i/N}$ for which $\#\tubes_\rho<\rho^{-\beta+\eps/4}.$ Let $\rho\in[\delta,\delta^\eps]$ be the smallest number of the form $\delta^{iN}$ for which this is true. Note that $\rho\geq\delta^{1-\eps/20}$; this is because the tubes in $\tubes'$ are essentially distinct, and thus $\#\tubes_\rho\geq (\delta/\rho)^4(\#\tubes')\geq(\delta/\rho)^4\delta^{-\beta+2\eta}$. 

Let $T_\rho\in\tubes_\rho$. By Conclusion (B) of Lemma \ref{correctNumberTubesInRectLem}, for each $s\in[\delta,\rho]$ of the form $\delta^{jN}$, each $t\in[s,\rho]$, and each $s\times t \times 1$ rectangular prism $R\subset T_\rho$, we have
\begin{align*}
\#\{T\in\tubes'[T_\rho]\colon T\subset R\}
&\leq \delta^{-\eta_2}\frac{t}{s}\frac{(\#\tubes')}{(\#\tubes_s)}\leq\delta^{-\eta_2}\frac{t}{s}\frac{\delta^{-\beta}}{\rho^{-\beta+\eps/4}}\Big(\frac{s}{\rho}\Big)^{\beta-\eps/4}\\
&\leq \delta^{-\eta_2}\frac{t}{s}\frac{\delta^{-\beta}}{\rho^{-\beta+\eps/4}}\Big(\frac{s}{\rho}\Big)^2
\leq \delta^{-2\eta_2}\frac{t}{s}\frac{(\#\tubes')}{(\#\tubes_\rho)}\Big(\frac{s}{\rho}\Big)^2
\leq\delta^{-2\eta_2}\frac{st}{\rho^2}(\#\tubes'[T_\rho]),
\end{align*}
where the second inequality used the fact that $\#\tubes_s\geq s^{-\beta+\eps/4}$ (this follows from the minimality of $\rho$); the third inequality used the assumption $\beta-\eps/4\geq 2$; and the fourth inequality used the assumptions $\#\tubes'\geq\delta^{-\beta+\eta_1}$ and $\#\tubes_\rho\leq\rho^{-\beta+\eps/4}$.

Define $\tilde\tubes=\tubes'[T_\rho]$, and let $W\subset T_\rho$ be a convex set that contains at least one tube from $\tilde\tubes$. Then $W$ is comparable to a $s\times t\times 1$ prism, for some $\delta\leq s\leq t\leq \rho$; increasing $s$ (and possibly $t$) by a factor of at most $\delta^{-1/N}$, we may suppose that $s$ is of the form $\delta^{jN}$. Thus we have
	\[
		\#\{T\in \tilde\tubes \colon T\subset W\} \lesssim \delta^{-4\eta_2}\frac{|W|}{|T_\rho|}(\#\tilde\tubes)\leq (\delta/\rho)^{-40\eta_2/\eps}\frac{|W|}{\rho^2}(\#\tilde\tubes)\leq (\delta/\rho)^{-\eta_1}\frac{|W|}{\rho^2}(\#\tilde\tubes).
	\]
Thus our set $\tilde\tubes$ satisfies Conclusion (D) from Proposition \ref{4WayDichotProp}, and we are done. Henceforth we shall assume that \eqref{boundOnFineTubesInsideRhoTube} holds.

\medskip

\noindent {\bf Step 3.} At this point, we have reduced to the case where \eqref{smallCardTubesRho} and \eqref{boundOnFineTubesInsideRhoTube} hold for all $\rho\in[\delta,\delta^\eps]$ of the form $\rho=\delta^{i/N}$. We will show that Conclusion (B) holds. To begin, note that Items (i) and (iii) from Definition \ref{selfSimilarWolff} hold for $\tubes'$: for each $\rho_0\in[\delta,\delta^{\eps}]$, we can select $\rho\in [\rho_0, \delta^{-\eps}\rho_0]$ of the form $\rho=\delta^{i/N}$; $\tubes_\rho$ is a partitioning cover of $\tubes'$, so Item (i) holds, and Item (iii) holds by  \eqref{boundOnFineTubesInsideRhoTube} (if $\rho_0>\delta^{\eps}$ then we can select $\rho=1$ and there is nothing to prove). 

It remains to show that Item (ii) holds for $\tubes'$. It suffices to consider $\rho$-tubes $T_\rho$ with $\rho\in[\delta,\delta^{\eps}]$ of the form $\rho=\delta^{i/N}$. Initially, we will consider convex sets $R$ that are $s\times t\times 1$ prisms, with $s\in[\delta,\rho]$ of the form $s=\delta^{j/N}$ and $t\in[s,\rho]$.  We would like to estimate the number of tubes from $\tubes'[T_\rho]$ contained in $R$. 

By Conclusion (B) of Lemma \ref{correctNumberTubesInRectLem} and \eqref{boundOnFineTubesInsideRhoTube}, we have
 \begin{align*}
\#\{T\in\tubes'[T_\rho]\colon T\subset R\} 
&\leq  \delta^{-\eta_2}\frac{t}{s}\frac{(\#\tubes')}{(\#\tubes_s)} 
\lesssim \delta^{-\eta_2-\eps/4}(ts) (s^{\beta-2}\delta^{-\beta})
\lesssim \delta^{-\eta_2-\eps/4}(ts) (\rho^{\beta-2}\delta^{-\beta})\\
& = \delta^{-\eta_2-\eps/4}\Big(\frac{ts}{\rho^2}\Big)\Big(\frac{\rho}{\delta} \Big)^\beta
 \leq \delta^{-2\eta_2-\eps/4}\Big(\frac{ts}{\rho^2}\Big)\big(\#\tubes'[T_\rho]\big)
 \leq \delta^{-\eps/2}\frac{|R|}{|T_\rho|}\big(\#\tubes'[T_\rho]\big).
\end{align*}
Finally, every convex set $W\subset T_\rho$ can be contained in a bounded number of $s\times t\times 1$ rectangular prisms $R$, where $s$ is of the form $\delta^{j/N}$, $t\in[s,\rho]$, and $|R|\leq \delta^{-\eta_2}|W|$. Thus for every convex set $W\subset\tubes_\rho$, we have
\[
\#\{T\in\tubes'[T_\rho]\colon T\subset W\} \leq \delta^{-\eps}\frac{|W|}{|T_\rho|}\big(\#\tubes'[T_\rho]\big).
\]
We conclude that $\tubes$ satisfies Condition (ii) from Definition \ref{selfSimilarWolff}, and hence Conclusion (B) of Proposition \ref{4WayDichotProp} holds. 
\end{proof}

\section{Sticky Kakeya and the self-similar Convex Wolff Axioms }\label{kakeyaSelfSimWolffSection}
Our goal in this section is to prove Theorem \ref{mainThmDiscretized}. Our first task is to analyze the structure of extremal collections of tubes. If $\omega>0$, then such collections must satisfy the self-similar Convex Wolff Axioms. 

\subsection{Extremal Kakeya sets satisfy the self-similar Convex Wolff Axioms}\label{extremalImpliesSelfSimilarWolffSection}
\begin{prop}\label{extremalImpliesSelfSimilarWolff}
Suppose $\omega$ from Definition \ref{defnOmega} is positive. Then for all $\eps>0$, there exists $\eta,\delta_0>0$ so that the following holds for all $\delta\in(0,\delta_0]$. Let $(\tubes,Y)_\delta$ be $\eta$ Assouad-extremal. Then there exists $\tubes'\subset\tubes$ so that $(\tubes',Y)_\delta$ is $\eps$ Assouad-extremal, and $\tubes'$ satisfies the self-similar Convex Wolff axioms with error $\delta^{-\eps}$. 
\end{prop}
\begin{proof}
\noindent{\bf Step 1.}
By decreasing $\eps$ if necessary, we can suppose that $\eps<\omega/2$. Let $\tau>0$ be the value from Proposition \ref{4WayDichotProp}, with this choice of $\eps$. Let $\eta_1$ be a small quantity to be chosen below ($\eta_1$ will depend on $\omega$ and $\eps$), and let $\eta',\delta_0'$ be the values from Proposition \ref{4WayDichotProp} for this choice of $\eps$ and $\eta_1$. 

Let $\eta$ and $\delta_0$ be small positive quantities to be chosen below. $\eta$ and $\delta_0$ will depend on the following quantities: $\omega,\eps,\tau,\eta_1,\eta'$. In addition, $\delta_0$ will depend on $\delta_0'$. 

Let $\delta\in(0,\delta_0]$ and let $(\tubes,Y)_\delta$ be $\eta$ Assouad-extremal. By pigeonholing, we can select a subset $\tubes_1\subset\tubes$ with $\#\tubes_1\gtrsim\delta^{\eta}(\#\tubes)$ so that $|Y(T)|\geq\frac{1}{2}\delta^{\eta}|T|$ for each $T\in\tubes_1$. Then $\#\tubes_1=\delta^{-\beta}$ for some $\beta\geq\alpha-2\eta$ (recall that $\alpha$ is defined in \eqref{defnAlphaEqn}), and $\tubes_1$ satisfies the Convex Wolff Axioms with error $\delta^{-2\eta}$. Furthermore, for all $\rho_1,r_1$ with $\rho_1\leq\delta^{\eta}r_1$, and all balls $B$ of radius $r_1$, we have 

\begin{equation}\label{smallVolumeInsideBallRho1R1}
\Big| B \cap N_{\rho_1}\Big(\bigcup_{T\in\tubes_1}Y(T)\Big)\Big|\leq(\rho_1/r_1)^{\omega-\eta}|B|.
\end{equation}

\medskip

\noindent \noindent{\bf Step 2.}
Apply Proposition \ref{4WayDichotProp} to $(\tubes_1,Y)_\delta$. We will show that Conclusion (A) from Proposition \ref{4WayDichotProp} cannot hold. Suppose to the contrary that Conclusion (A) holds, i.e.~there exists $\rho,r\in[\delta,1]$ with $\rho\leq\delta^{\eta'} r$ and a $r$-ball $B$ such that
\begin{equation}\label{ConclusionAImpossible}
\Big| B\cap N_{\rho}\Big(\bigcup_{T\in\tubes_1}Y(T)\Big)\Big|\geq(\rho/r)^\eps|B|.
\end{equation}
If we select $\eta\leq\min(\eta', \omega_2)$ (recall that $\eps\leq\omega/2$), then \eqref{ConclusionAImpossible} contradicts \eqref{smallVolumeInsideBallRho1R1}. Hence Conclusion (A) cannot hold. 

\medskip

\noindent \noindent{\bf Step 3.}
We will show that Conclusion (C) from Proposition \ref{4WayDichotProp} cannot hold. Suppose to the contrary that Conclusion (C) holds, i.e.~there exists $\rho\in[\delta,\delta^\eps]$ and a set of $\rho$-tubes $\tilde\tubes$ satisfying Items (C.i), (C.ii) and (C.iii) from Conclusion (C). By Item (C.iii), we have that the shading
\[
\tilde Y(\tilde T) = \tilde T \cap N_{\rho}\Big(\bigcup_{T\in\tubes_1}Y(T)\Big),\quad\tilde T\in\tilde\tubes
\]
satisfies $|\tilde Y(\tilde T)|\gtrsim\delta^{\eta}|\tilde T|\gtrsim \rho^{\eta/\eps}|\tilde T|$ for each $\tilde T\in\tilde\tubes$; we will select $\eta<\eta'\eps$ and $\delta_0$ sufficiently small that $|\tilde Y(\tilde T)|\leq\rho^{\eta'}|\tilde T|$ for each $\tilde T\in\tilde\tubes$. 

Define
\begin{equation}\label{defnGamma}
\gamma = \omega-\omega(\beta+\tau).
\end{equation}
(In the above equation, $\omega(\beta+\tau)$ refers to the function $\omega$ evaluated at the (positive) number $\beta+\tau$). 
Provided $\eta_1$ is selected sufficiently small (depending on $\tau$), we have $\beta+\tau\geq\alpha-2\eta_1+\tau>\alpha$, and hence by the definition of $\alpha$ from \eqref{defnAlphaEqn}, we have $\gamma>0$. Hence if $\eta_1$ is sufficiently small, then
\[
\limsup_{x\searrow 0}\omega(\eta_1,\eta_1,x,\beta+\tau)\leq \omega-\gamma/2,
\] 
i.e.~there exists $x_0>0$ so that
\[
\omega(\eta_1,\eta_1,x,\beta+\tau)\leq \omega-\gamma/4\quad\textrm{for all}\ x\in(0,x_0]. 
\]
Select $\delta_0$ sufficiently small so that $\delta_0^\eps\leq x_0$, and thus $\rho\in[\delta,\delta^\eps]$ implies $\rho\leq x_0$. In particular, $\omega(\eta_1,\eta_1,\rho,\beta+\tau)\leq \omega-\gamma/4,$ i.e.~there exists $\rho_1,r_1\in[\rho,1]$ with $\rho_1\leq \rho^{\eta_1}r_1$, and a $r_1$-ball $B$ such that
\begin{equation}\label{bigBallCaseC}
\Big| B\cap N_{\rho_1}\Big(\bigcup_{\tilde T\in\tilde\tubes}\tilde Y(\tilde T)\Big)\Big|\geq(\rho_1/r_1)^{\omega-\gamma/8}|B|.
\end{equation}
But if we select $\eta\leq \min(\eta_1\eps,\gamma/9)$, then $\rho_1\leq\delta^{\eta}r_1$, and \eqref{bigBallCaseC} contradicts \eqref{smallVolumeInsideBallRho1R1}. Hence Conclusion (C) cannot hold.

\medskip

\noindent \noindent{\bf Step 4.}
We will show that Conclusion (D) from Proposition \ref{4WayDichotProp} cannot hold. Suppose to the contrary that Conclusion (D) holds, i.e.~there exists $\rho\in[\delta^{1-\eps},1]$, a $\rho$-tube $T_{\rho}$, and a set $\tilde\tubes\subset\tubes_1[T_\rho]$ satisfying Items (D.i) and (D.ii) from Conclusion (D). For each $T\in\tilde\tubes$, let $\tilde Y(\tilde T)\subset Y(\tilde T)$ be the sub-shading produced by Lemma \ref{findRegularShading}; this shading is regular at scales $\geq\delta$ and satisfies $|\tilde Y(\tilde T)|\geq \frac{1}{2}|Y(\tilde T)|$. 

Let $\tilde\delta=\delta/\rho\in[\delta,\delta^{\eps}]$ and let $\tilde\tubes_{\tilde\delta}$ be the unit rescaling of $\tilde\tubes$ relative to $T_\rho$, i.e.~the tubes in $\tilde\tubes_{\tilde\delta}$ are the images of the tubes in $\tilde\tubes$ under the affine map $\phi_{T_\rho}$.  Let $\tilde Y_{\tilde\delta}(\tilde T_{\tilde\delta})$ be the image of $\tilde Y(\tilde T)$ under this same rescaling. We have $|\tilde Y_{\tilde\delta}(\tilde T)|\gtrsim \delta^{\eta}|\tilde T_{\tilde\delta}|\geq \tilde\delta^{\eta/\eps}|\tilde T_{\tilde\delta}|$, and we will select $\eta$ and $\delta_0$ sufficiently small so that $|\tilde Y_{\tilde\delta}(\tilde T_{\tilde\delta})|\geq\tilde\delta^{\eta_1}|\tilde T_{\tilde\delta}|$ for all $\tilde T_{\tilde\delta}\in \tilde\tubes_{\tilde\delta}$. The tubes in $\tilde\tubes_{\tilde\delta}$ are essentially distinct; by Item (D.i) we have $\#\tilde\tubes_{\tilde\delta}\geq\tilde\delta^{-\beta-\tau}$; and by Item (D.ii) we have that $\tilde\tubes_{\tilde\delta}$ satisfies the Convex Wolff Axioms with error $\lesssim\tilde\delta^{-\eta_1}$.

Arguing as in Step 3, we conclude that there exists $\rho_2,r_2\in[\tilde\delta,1]$ with $\rho_2\leq\tilde\delta^{\eta_1}r_2$, and a $r_2$-ball $B_2$ such that
\begin{equation}\label{rho2NbhdInsideB2}
\Big| B_2\cap N_{\rho_2}\Big(\bigcup_{\tilde T_{\tilde\delta}\in\tilde\tubes_{\tilde\delta}}\tilde Y_{\tilde\delta}(\tilde T_{\tilde\delta})\Big)\Big|\geq(\rho_2/r_2)^{\omega-\gamma/8}|B_2|.
\end{equation}

\noindent Let $E_2=\phi_{T_\rho}^{-1}(B_2)$; this is an ellipsoid of dimensions $\rho r_2\times \rho r_2\times \rho$. We claim that
\begin{equation}\label{claimedSizeOfE2CapNRhoRho2}
\Big| E_2\cap N_{\rho\rho_2}\Big(\bigcup_{\tilde T\in\tilde\tubes}\tilde Y(\tilde T)\Big)\Big|\gtrapprox \delta^{-\eta}(\rho_2/r_2)^{\omega-\gamma/8}|E_2|.
\end{equation}

To verify this claim, we argue as follows. Let $\tilde T_{\tilde\delta}\in\tilde\tubes_{\tilde\delta}$, let $\tilde x\in \tilde Y_{\tilde\delta}(\tilde T_{\tilde\delta})$, and let $\tilde B = B(\tilde x, \rho_2)$. Let $E=\phi_{T_\rho}^{-1}(\tilde B)$; this is an ellipsoid of dimensions $\rho_2\rho\times\rho_2\rho\times\rho_2$, and let $x=\phi_{T_\rho}^{-1}(\tilde x)$ be the center of $E$. Similarly, let $\tilde T_\delta$ and $\tilde Y(\tilde T)$ be the $\delta$ tube and its shading corresponding to $\tilde T_{\tilde\delta}$ and the shading $\tilde Y_{\tilde\delta}(\tilde T_{\tilde\delta})$. Then $E\cap\tilde T$ contains a tube-segment of length $\geq \rho_2/2$ centered at $x$, i.e. $E\cap \tilde T\supset B(x, \rho_2/2)$. Since $\tilde Y(\tilde T)$ is regular at scales $\geq\delta$ (and $\rho_2\geq \tilde\delta =\delta/\rho)$, we have that $|E\cap\tilde Y(\tilde T)|\gtrapprox \delta^\eta \rho_2\delta^2$, i.e. $E\cap\tilde Y(\tilde T)$ contains $\gtrapprox \delta^\eta (\rho_2)/(\rho\rho_2)=\delta^\eta \rho^{-1}$ points that are $10\rho\rho_2$-separated. For each such point $y$, we have that $B(y, \rho\rho_2)\subset 2E$, and these balls are disjoint. Since each of these balls has volume roughly $(\rho\rho_2)^3$, we conclude that
\begin{equation}\label{volumeInsideEllipsoid}
|2E\cap N_{\rho\rho_2}\Big(\bigcup_{\tilde T'\in\tilde\tubes}\tilde Y(\tilde T')\Big)\Big| \geq |2E\cap N_{\rho\rho_2}(\tilde Y(\tilde T))| \gtrapprox \delta^{\eta} \rho^2\rho_2^3.
\end{equation}

Next, let $\{\tilde x_i\}_{i=1}^N$ be a $2\rho_2$-separated subset of $B_2\cap \bigcup_{\tilde\tubes_{\tilde\delta}}\tilde Y_{\tilde\delta}(\tilde T_{\tilde\delta})$ of cardinality
\[
N\sim \rho_2^{-3}\Big| B_2\cap N_{\rho_2}\Big(\bigcup_{\tilde T_{\tilde\delta}\in\tilde\tubes_{\tilde\delta}}\tilde Y_{\tilde\delta}(\tilde T_{\tilde\delta})\Big)\Big|
\geq\rho_2^{-3}(\rho_2/r_2)^{\omega-\gamma/8}|B_2|,
\]
with the property that $B(\tilde x_i, 2\rho_2)\subset B_2$ for each index $i$.

For each index $i$, let $E_i=\phi_{T_\rho}^{-1}(B(\tilde x_i, \rho_2));$ this is an ellipsoid of dimensions $\rho_2\rho\times\rho_2\rho\times\rho_2$ (this introduces a slight clash of notation, since we previously introduced the ellipsoid $E_2$ of dimensions $\rho r_2\times \rho r_2\times \rho$). The ellipsoids $\{2E_i\}_{i=1}^N$ are disjoint and contained in $E_2$. Applying \eqref{volumeInsideEllipsoid} to each such ellipsoid, we have
\begin{equation}
\begin{split}
\Big| E_2\cap N_{\rho\rho_2}\Big(\bigcup_{\tilde T\in\tilde\tubes}\tilde Y(\tilde T)\Big)\Big| 
& \geq \sum_{i=1}^N \Big|2E_i\cap N_{\rho\rho_2}\Big(\bigcup_{\tilde T\in\tilde\tubes}\tilde Y(\tilde T)\Big)\Big|\\
& \gtrapprox N \delta^{\eta}\rho^2\rho_2^3\\
&\gtrapprox \delta^{\eta}\rho^2(\rho_2/r_2)^{\omega-\gamma/8}|B_2|\\
& =\delta^{\eta}(\rho_2/r_2)^{\omega-\gamma/8}|E_2|.
\end{split}
\end{equation}
This establishes \eqref{claimedSizeOfE2CapNRhoRho2}, as claimed. 

Next, by \eqref{claimedSizeOfE2CapNRhoRho2} and pigeonholing we can select a ball $B\subset E_2$ of radius $r_2\rho$ so that
\[
\Big| B\cap N_{\rho\rho_2}\Big(\bigcup_{\tilde T\in\tilde\tubes}\tilde Y(\tilde T)\Big)\Big|\gtrapprox \delta^{\eta}(\rho_2/r_2)^{\omega-\gamma/8}|B|\geq (\rho_2/r_2)^{\omega-\gamma/8+\eta/(\eps\eta_1)}|B|.
\]
If we select $\delta_0$ and $\eta$ sufficiently small depending on $\gamma,\eps,$ and $\eta_1$, then

\begin{equation}\label{bigBallCaseD}
\Big| B\cap N_{\rho\rho_2}\Big(\bigcup_{\tilde T\in\tilde\tubes}\tilde Y(\tilde T)\Big)\Big|\geq 2 (\rho_2/r_2)^{\omega-\eta}|B|.
\end{equation}
Finally, observe that the scale separation $\frac{\rho\rho_2}{\rho r_2}$ satisfies $\frac{\rho\rho_2}{\rho r_2}=\frac{\rho_2}{r_2}\leq\rho^{\eta_1}\leq\delta^{\eta_1\eps}\leq\delta^\eta$, provided we select $\eta$ sufficiently small. But then \eqref{bigBallCaseD} contradicts \eqref{smallVolumeInsideBallRho1R1}. Hence Conclusion (D) cannot hold. 

\medskip

\noindent \noindent{\bf Step 5.}
At this point, we shown that Conclusions (A), (C), and (D) from Proposition \ref{4WayDichotProp} cannot hold. Thus Conclusion (B) must hold. This completes the proof of Proposition \ref{extremalImpliesSelfSimilarWolff}.
\end{proof}

\subsection{Proof of Theorem \ref{mainThmDiscretized}}
Our proof of Theorem \ref{mainThmDiscretized} follows by combining two main ingredients. The first is Proposition \ref{extremalImpliesSelfSimilarWolff}. The second is a variant of the Sticky Kakeya Theorem, which was proved by the authors in \cite{WZ}. The precise statement we need is as follows.

\begin{thm}\label{stickySelfSimThm}
For all $\eps>0$, there exists $\eta,\delta_0>0$ so that the following holds for all $\delta\in(0,\delta_0]$. Let $\tubes$ be a set of $\delta$-tubes that satisfy the Convex Wolff Axioms at every scale with error $\delta^{-\eta}$, and let $Y(T)$ be a $\delta^{\eta}$ dense shading. Then
\begin{equation}\label{conclusionStickySelfSimThm}
\Big|\bigcup_{T\in\tubes}Y(T)\Big|\geq\delta^\eps.
\end{equation}
\end{thm}

As noted above, Theorem \ref{stickySelfSimThm} is similar to the Sticky Kakeya Theorem proved by the authors in \cite{WZ}. In Section \ref{stickySelfSimThmSec} we will discuss the proof of Theorem \ref{stickySelfSimThm}, and how it differs from the arguments in \cite{WZ}. Assuming this result for now, we can combine these two ingredients to prove Theorem \ref{mainThmDiscretized} as follows

\begin{proof}[Proof of Theorem \ref{mainThmDiscretized}]
It suffices to show that $\omega=0$. Suppose instead that $\omega>0$. Let $\eta_0,\delta_0$ be the output from Theorem \ref{stickySelfSimThm}, with $\eps=\omega/3$. By Proposition \ref{extremalImpliesSelfSimilarWolff}, there exists $\delta\in(0,\delta_0]$ and a set $(\tubes',Y)_\delta$ with discretized Assouad dimension $\leq 3-\omega/3$ that satisfies the self-similar Convex Wolff Axioms (and hence the Convex Wolff Axioms at every scale) with error $\delta^{-\eta_0}$, and for which $\{Y(T)\}$ is a $\delta^{\eta_0}$ dense shading. But then \eqref{conclusionStickySelfSimThm} says that $\bigcup_\tubes Y(T)$ has discretized Assouad dimension at least $3-\omega/3$, which is a contradiction. 
\end{proof}
Thus we have completed the proof of Theorem \ref{mainThmDiscretized}, except that it remains to prove Theorem \ref{stickySelfSimThm}; we will do so in the next section.


 \section{Kakeya estimates for tubes satisfying the Convex Wolff Axioms at every scale}\label{stickySelfSimThmSec}
Our goal in this section is to prove Theorem \ref{stickySelfSimThm}. Theorem \ref{stickySelfSimThm} is closely related to the discretized analogue of \cite[Theorem 1.1]{WZ}, and specifically the estimate $\sigma_3=0$ from that paper; the latter result says that if $\tubes$ is a collection of $\delta$-tubes in $\RR^3$ of cardinality roughly $\delta^{-2}$ that point in $\delta$-separated directions, and if $\tubes$ satisfies a property called ``stickiness,'' then the union of the tubes must have volume $\gtrapprox 1$. Stickiness is a technical property, which roughly speaking says that for each $\rho\in[\delta,1]$, the tubes in $\tubes$ can be covered by a collection of $\rho$-tubes that point in $\rho$-separated directions (more accurately, the tubes point in $\rho$-separated directions up to multiplicity $\lessapprox 1$). 

Stickiness is a multi-scale property, in the following sense: if $\tubes$ is sticky, then for each $\rho\geq\delta$ we can cover $\tubes$ by a sticky collection of $\rho$-tubes. This multi-scale property was exploited throughout the arguments in \cite{WZ}. Theorem \ref{stickySelfSimThm} differs from \cite[Theorem 1.1]{WZ}, because the stickiness hypothesis has been replaced by the hypothesis that the tubes satisfy the Convex Wolff axioms at every scale. As we will see below, however, this latter hypothesis is sufficient in order to repeat the arguments from  \cite{WZ}. Indeed, with two exceptions that will be explained below, the arguments from  \cite{WZ} can be repeated mutatis mutandis to obtain Theorem \ref{stickySelfSimThm}.

The proof of \cite[Theorem 1.1]{WZ} is divided into several sections. The paper begins by supposing that the result is false, and studying the structure of a (hypothetical) counter-example. In Section 2, the authors define what it means for a pair $(\tubes,Y)$ of tubes and their associated shading to be a worst possible counter-example to \cite[Theorem 1.1]{WZ}; this is called an $\eps$-extremal pair, and we will adopt this terminology in our discussion below. In \cite{WZ}, Section 3, it is shown that an $\eps$-extremal pair $(\tubes,Y)$ must look coarsely self similar at every intermediate scale $\rho \in [\delta,1]$. More precisely, for every such $\rho$ it is possible to cover the tubes in $\tubes$ by a collection of $\rho$-tubes, so that the corresponding collection of $\rho$-tubes is a worst possible counter-example to \cite[Theorem 1.1]{WZ} at scale $\rho$, and the $\delta$-tubes inside each $\rho$-tube also form a (rescaled) worst-possible counter-example at scale $\delta/\rho$. The analogous statement and proof in our setting are nearly identical; we will state the former and briefly sketch the latter.

In Sections 3, 4, and 5 of \cite{WZ}, the authors use the multi-scale self-similarity established in Section 2 of \cite{WZ} to show that an extremal pair $(\tubes,Y)$ must have certain structural properties called planiness and graininess, and in particular there must exist local grains that are described by a Lipschitz plane map, and global grains that are described by a $C^2$ slope function $f$; this will be described in greater detail below. The arguments from \cite{WZ} can be repeated without modification to establish the same conclusions in the present setting. In particular, the hypothesis from \cite{WZ} that the tubes point in $\delta$-separated directions has not been used up until this step. 

In Section 6 of \cite{WZ}, the authors show that the derivative $f'$ of the slope function has magnitude roughly 1 (and in particular, this quantity is bounded away from 0). It is during this step that the authors use the hypothesis that the tubes point in $\delta$-separated directions. With some small modifications, we can instead use the hypothesis that the tubes satisfy the Convex Wolff axioms at every scale; we will elaborate on this step in Section \ref{largeSlopeSection}.

Finally, in Section 7 of \cite{WZ}, the authors recall an analogue of Wolff's circular maximal function estimate \cite{Wolff97, PYZ}, and they show how this estimate contradicts the assumption that there exists a collection of sticky tubes whose union has volume much smaller than 1, thereby completing the proof of the theorem. This section of \cite{WZ} uses the assumption that the tubes point in different directions in order to obtain a rich planar arrangement of curves to which Wolff's circular maximal theorem can be applied. With some modifications, the direction-separated assumption can be replaced with a weaker ball condition. We will explain this in detail in Sections \ref{twistedProjectionsSection} and \ref{twistedProjectionsTubeSec}.

\subsection{Relevant definitions from \cite{WZ}}
Before describing the arguments in \cite{WZ}, it will be helpful to recall a few key definitions. First, a minor technical annoyance: tubes and cubes in \cite{WZ} are defined slightly differently than in the present paper. In \cite{WZ}, a $\delta$-tube in $\RR^3$ is defined to be a set of the form $N_{6\delta}(\ell)\cap [-1,1]^3$, where $\ell$ is a line in $\RR^3$. This definition is adopted for technical reasons, and for consistency with the results in \cite{WZ} we will use this definition throughout Section \ref{stickySelfSimThmSec}; the distinction between this definition and our previous definition in Section \ref{introSection} is not important in the arguments that follow. Second, in \cite{WZ}, a shading $Y(T)\subset T$ is defined to be a union of axis-aligned $\delta$-cubes contained in $T$. Again, the distinction between this definition and our previous one from Section \ref{tubesAndShadingsSection} is harmless: replace each tube $T$ (in the sense of Section \ref{introSection}) and a each shading $Y(T)$ (in the sense of Section \ref{tubesAndShadingsSection}) with the corresponding tube in the sense of \cite{WZ} and the shadings consisting of the union of axis-aligned $\delta$-cubes $Q$ that satisfy $|Q\cap Y(T)|\geq\frac{1}{100}|Y(T)|/|T|$. This gives us a collection of tubes and shadings in the sense of \cite{WZ}, and if the conclusion of Theorem \ref{mainThmDiscretized} holds for this latter collection of tubes and shadings, then it also holds for the original collection of tubes and shadings, except Inequality \eqref{rhoRNbhdFullVol} has been weakened by an additional (harmless) factor of $1/1000$.

\paragraph{Tubes, covers, rescaling.}
We will often refer to a collection of $\delta$-tubes $\tubes$, and a shading $Y(T)$ of these tubes. We will refer to this pair as $(\tubes,Y)_\delta$. We use $E_{\tubes}$ to denote the set $\bigcup_{\tubes}Y(T)$; the shading $Y$ will always be apparent from context. For $p\in\RR^3$, we define
\begin{equation}\label{eq: TpDefn} 
\tubes(p)=\{T\in\tubes\colon p\in Y(T)\}.
\end{equation}
We say a pair $(\tubes',Y')_\delta$ is a \emph{refinement} of $(\tubes,Y)_\delta$ if $\tubes'\subset\tubes$; $Y'(T')\subset Y(T')$ for each $T\in\tubes'$; and $\sum_{\tubes'}|Y'(T')|\geq C^{-1}|\log\delta|^{-C}\sum_{\tubes}|Y(T)|$ for some absolute constant $C$ (in practice, $C=100$ will always suffice).

Let $0<\delta\leq\rho \leq 1$. In \cite{WZ}, the authors say that a $\rho$-tube $\tilde T$ \emph{covers} a $\delta$-tube $T$ if their respective coaxial lines satisfy $d(\ell,\tilde\ell)\leq\rho/2$, where $d(\cdot,\cdot)$ is an appropriately chosen metric on the affine Grassmannian of lines in $\RR^3$; see Section 2 of \cite{WZ} for details. Note that if $\tilde T$ covers $T$, then $T\subset \tilde T$, and hence this definition is consistent with ours from Definition \ref{coversDefn} (the definition in \cite{WZ} is slightly stronger than Definition \ref{coversDefn}, but the distinction is harmless).

If $(\tubes,Y)_\delta$ and  $(\tilde\tubes,\tilde Y)_\rho$ are pairs of $\delta$ (resp.~$\rho$) tubes and their associated shadings, then we say $(\tilde\tubes,\tilde Y)_\rho$ \emph{covers} $(\tubes,Y)_\delta$ if each $T\in\tubes$ is covered by at least one $\tilde T\in\tilde\tubes$, and for each $\tilde T \in\tilde\tubes$ covering $T$, we have $Y(T)\subset\tilde Y(\tilde T)$. If this is the case, we write $\tubes[\tilde T]$ to denote the set of tubes in $\tubes$ covered by $\tilde T$. We say that $(\tilde\tubes,\tilde Y)_\rho$ is a \emph{balanced cover} of $(\tubes,Y)_\delta$ if it is a cover, and in addition $|E_{\tubes}\cap Q|$ is the same for each $\rho$-cube $Q\subset E_{\tilde\tubes}$.

Let $(\tubes,Y)_\delta$ be a set of $\delta$-tubes and their associated shading. Let $\tilde T$ be a $\rho$-tube, and suppose each tube in $\tubes$ is covered by $\tilde T$. We define a new pair $(\hat\tubes, \hat Y)_{\delta/\rho}$, which we call the \emph{unit rescaling of $(\tubes,Y)_\delta$ relative to $\tilde T$}. Informally, our definition is as follows: consider the linear transformation that sends the coaxial line of $\tilde T$ to the $z$ axis, and dilates the $x$ and $y$ directions by $\rho^{-1}$; the image of $\tilde T$ under this transformation is comparable to the unit ball. The tubes in $\hat\tubes$ are those $(\delta/\rho)$-tubes whose coaxial lines are the images of the coaxial lines of the tubes in $\tubes$ under this transformation. The shadings $\hat Y(\hat T)$ consist of the union of axis-aligned $\delta/\rho$-cubes that intersect the image of $Y(T)$ under the above transformation. See Section 3 of \cite{WZ} for details. As above, the definition of unit rescaling in \cite{WZ} is slightly different than Definition \ref{unitRescalingDefn}, but the distinction is harmless.

\paragraph{The plane map.}
In the arguments in \cite{WZ}, the authors show that certain collections of tubes possess an important structural property called planiness. Let $(\tubes,Y)_\delta$ be a set of tubes and their associated shading. We say a function $V\colon E_{\tubes}\to S^2$ is a \emph{plane map} for $(\tubes,Y)_\delta$ if
\[
|\operatorname{dir}(T)\cdot V(p)|\leq\delta\quad \textrm{for all}\ (T,p)\in \tubes\times E_{\tubes}\ \textrm{with}\ p\in Y(T),
\]
where $\operatorname{dir}(T)\in S^2$ is a unit vector parallel to the line coaxial with $T$.

\paragraph{Discretized Ahlfors-David regular sets.}
We say a set $E\subset\RR$ is a $(\delta,\alpha,C)$-ADset if for all $\rho\geq\delta$, all $r\geq\rho$, and all intervals $I$ of length $r$, we have $\mathcal{E}_\rho(E\cap I)\leq C(r/\rho)^\alpha$, where $\cE_{\rho}(E\cap I)$ means the $\rho$-covering number of $E\cap I$.  In practice we will have $\alpha\in (0,1]$, and $C$ will be a small power of $1/\delta$. See Section 4 of \cite{WZ} for further discussion and motivation for this definition.

\subsection{Extremal collection of tubes, and multi-scale structure}
We are ready to begin the proof of Theorem \ref{stickySelfSimThm}. Suppose to the contrary that the result was false. Then there exists $\sigma'>0$ so that for all $\eta,\delta_0>0$, there exists $\delta\in(0,\delta_0]$; a set $\tubes$ of $\delta$-tubes that satisfy the Convex Wolff Axioms at every scale with error $\delta^{-\eta}$; and a $\delta^\eta$-dense shading $Y(T)$, so that 
\[
\Big|\bigcup_{T\in\tubes}Y(T)\Big|\leq\delta^{\sigma'}.
\]
Let $\sigma>0$ be the supremum of all numbers $\sigma'$ with the above property; this supremum exists, since by hypothesis the set of admissible $\sigma'$ is non-empty, and it is bounded above by 3. Our goal is to show that we must in fact have $\sigma=0$, and thereby obtain a contradiction. 

The following definition is the analogue of \cite[Definition 3.1]{WZ}. 
\begin{defn}\label{def: extremal}
    Let $\eps, \delta>0$.  We say a pair $(\tubes, Y)_{\delta}$ is $\eps$-\emph{extremal} if the following holds:
    \begin{enumerate}
        \item $\tubes$ satisfies the Convex Wolff Axioms at all scales with error $\delta^{-\eps}$.
        \item $Y$ is $\delta^{\eps}$-dense.
        \item $\Big|\bigcup_{\tubes}Y(T)\Big|\leq\delta^{\sigma-\eps}.$
    \end{enumerate}
\end{defn}
It is immediate from the above definitions that for every $\eps,\delta_0>0$, there exists an $\eps$-extremal collection of tubes $(\tubes,Y)_\delta$ for some $\delta\in(0,\delta_0]$; cf.~\cite[Lemma 3.1]{WZ}. 

Our next task is to prove an analogue of \cite[Proposition 3.2]{WZ}. The precise statement is as follows. 

\begin{prop}\label{prop: sticky}
    For all $\eps>0$, there exists $\eta>0$ and $\delta_0>0$ so that the following holds for all $\delta\in (0,\delta_0]$.  Let $(\TT, Y)_{\delta}$ be an $\eta$-extremal collection of tubes, and let $\rho\in [\delta^{1-\eps}, \delta^{\eps}]$. Then there is a refinement $(\TT', Y')_{\delta}$ of $(\TT, Y)_{\delta}$ and a balanced cover $(\tilde{\TT}, \tilde{Y})_{\rho}$ of $(\TT', Y')_{\delta}$ with the following properties: 
    \begin{enumerate}[(i)]
        \item  \label{it: i} $(\tilde{\TT}, \tilde{Y})_{\rho}$ is $\eps$-extremal.
        \item  \label{it: ii} For each $\tilde{T}\in \tilde{\TT}$, the unit rescaling  of $(\TT'[\tilde{T}], Y')_{\delta}$ relative  to $\tilde{T}$ is $\eps$-extremal.
        \item  \label{it: iii} For each $p\in \RR^3$, $\# \tilde{\TT}(p) \leq \rho^{2-\sigma-\eps} (\#\tilde{\TT}).$
        \item  \label{it: iv} For each $p\in \RR^3$ and each $\tilde{T}\in \tilde{\TT}$, $\# \TT'[\tilde{T}](p) \leq (\delta/\rho)^{2-\sigma-\eps} (\# \TT'[\tilde{T}]).$
    \end{enumerate}
\end{prop}

\begin{proof}
The proof is nearly identical to the proof of the corresponding statement in \cite[Proposition 3.2]{WZ}. In \cite{WZ}, the first major step was to establish the existence of a cover $\tubes_{\rho}$ of $\tubes$ (or more accurately, a large subset $\tubes'\subset \tubes$), with the property that both $\tubes_\rho$ and the rescaled sets $\tubes'[T_\rho]$ are sticky. In our case, this step was handled by Lemma \ref{multiScaleWolffLem}. The second major step was to establish a shading on $\tubes_\rho$ and a refinement of the shadings on the sets $\tubes'[T_\rho]$ so that not too many tubes pass through each point (and as a consequence, the corresponding collections are $\eps$-extremal). 

Next we will give a slightly more detailed sketch, though we refer the reader to the proof of \cite[Proposition 3.2]{WZ} for full details. We begin by reducing to the case where each $T\in\tubes$ satisfies $|Y(T)|\geq\frac{1}{2}\delta^{\eta}|T|$. This can be done by discarding those tubes with $|Y(T)|<\frac{1}{2}\delta^\eta|T|$ and then repeatedly pigeonholing to re-establish a collection of tubes that satisfies the Convex Wolff Axioms at every scale (with a slightly worse constant, say $\delta^{-2\eta}$); after this process, the cardinality of our collection of tubes has been reduced by at most a factor of $\delta^\eta$.

Next, we apply Lemma \ref{multiScaleWolffLem} with $\rho_0=\delta^{2\eta}\rho$. We get a number $\rho'\in [\delta^{2\eta}\rho, \rho]$ and a collection $\tubes_{\rho'}$ of $\rho'$-tubes, as described by that Lemma. We may replace each $\rho'$ tube by its corresponding $\rho$-tube, and denote the resulting collection by $\tilde \tubes$. After pigeonholing to re-establish the property that the tubes are essentially distinct, it is still the case that $\tilde\tubes$ satisfies the Convex Wolff Axioms at every scale with error $\delta^{-O(\eta)}$, and each of the rescaled collections $\tubes[\tilde T]$ satisfies the Convex Wolff Axioms at every scale with error $\delta^{-O(\eta)}$. 

After refining the shadings $Y(T)$, we may suppose that for each collection $\tubes[\tilde T]$, the same number of tubes pass through each $\delta$-cube in $E_{\tubes[\tilde T]}$ in the sense of \eqref{eq: TpDefn}; after further pigeonholing we may suppose that this number is the same (up to a multiplicative factor of 2) for each set $\tubes[\tilde T]$; call this number $\mu_{\operatorname{fine}}$. Since $\tubes[\tilde T]$ is a (rescaled) collection of tubes satisfying the Convex Wolff Axioms at every scale, if $\eta$ and $\delta_0$ are chosen sufficiently small then we must have $\mu_{\operatorname{fine}}\leq (\delta/\rho)^{2-\sigma-\eps/10} (\# \TT'[\tilde{T}])$; this is Conclusion (iv) above.

Our next task is to define a shading $\tilde Y$ on the tubes of $\tilde\tubes$. To begin, for each $\tilde T$ we consider the union of all $\rho$-cubes $\tilde Q\subset \tilde T$ that intersect $E_{\tubes[\tilde T]}$. After pigeonholing, we can refine this shading so that each $\rho$-cube intersects $E_{\tubes[\tilde T]}$ in a set of roughly the same size, and the same number of tubes from $\tilde \tubes$ pass through each cube of $E_{\tilde \tubes}$; call this number $\mu_{\operatorname{coarse}}$. We refine the corresponding shadings $Y(T),\ T\in\tubes$ so that $Y(T)\subset\tilde Y(\tilde T)$ whenever $T\in\tubes[\tilde T]$; in particular $(\tilde{\TT}, \tilde{Y})_{\rho}$  is now a balanced cover of $(\tubes, Y)_\delta$. Note that Conclusion (iv) remains true after this procedure. Since $\tilde\tubes$ satisfies the Convex Wolff Axioms with error $\delta^{-O(\eta)}$, if $\eta$ and $\delta_0$ are chosen sufficiently small then we must have $\mu_{\operatorname{coarse}}\leq \rho^{2-\sigma-\eps/10} (\tilde\tubes)$; this is Conclusion (iii) above. 

Finally, observe that $\mu_{\operatorname{fine}}\mu_{\operatorname{coarse}}\leq \delta^{2-\sigma-\eps/5}(\#\tubes)$, but on the other hand since $\tubes$ was $\eta$-extremal, we must have $\mu_{\operatorname{fine}}\mu_{\operatorname{coarse}}\gtrsim \delta^{2-\sigma+\eta}(\#\tubes)$. We conclude that 
\begin{align*}
(\delta/\rho)^{2-\sigma+\eps/10} (\# \TT'[\tilde{T}]) \leq & \mu_{\operatorname{fine}}\leq (\delta/\rho)^{2-\sigma-\eps/10} (\# \TT'[\tilde{T}]),\quad \tilde T\in\tilde\tubes,\\
\rho^{2-\sigma+\eps/10} (\# \tilde\tubes) \leq & \mu_{\operatorname{coarse}}\leq \rho^{2-\sigma-\eps/10} (\# \tilde\tubes),
\end{align*}
and in particular Conclusions (i) and (ii) are satisfied. We refer the reader to \cite{WZ} for complete details. 
\end{proof}

\subsection{Local and global grains}
Using Proposition~\ref{prop: sticky} in place of \cite[Proposition 3.2]{WZ}, we can follow an identical argument as in \cite{WZ} to obtain the analogue of \cite[Proposition 4.1]{WZ}, which we recall here. 

\begin{prop}\label{prop: grain}

For all $\eps, \delta_0>0$, there exists $\delta\in (0,\delta_0]$ and an $\eps$-extremal set of tubes $(\TT, Y)_{\delta}$ with the following two properties:

\begin{enumerate}
    \item \textbf{ $E_{\TT}$ is a union of global grains with Lipschitz slope function. }

    There is a 1-Lipschitz function $f: [-1,1]\rightarrow \mathbb{R}$ so that $(E_{\TT}\cap \{z=z_0\})\cdot (1, f(z_0), 0)$ is a $(\delta, 1-\sigma, \delta^{-\eps})_1$-ADset for each $z_0\in [-1,1]$. 

    \item \textbf{ $E_{\TT}$ is a union of local grains with Lipschitz plane map. }

    $(\TT, Y)_{\delta}$ has a $1$-Lipschitz plane map $V$. For all $\rho\in [\delta, 1]$ and all $p\in \mathbb{R}^3$, $V(p)\cdot (B(p, \rho^{1/2})\cap E_{\TT})$ is a $(\rho, 1-\sigma, \delta^{-\eps})_1$-ADset. 
\end{enumerate}
\end{prop}

 Following the same proof as \cite[Section 5]{WZ} using radial projections, we can upgrade the global grains slope function $f$ in Proposition~\ref{prop: grain} from Lipschitz to $C^2$. This is the analogue of \cite[Proposition 5.1]{WZ}. The precise statement is as follows.

 \begin{prop}\label{prop: C2}
     For all $\eps, \delta_0>0$, there exists $\delta\in (0,\delta_0]$ and an $\eps$-extremal set of tubes $(\TT, Y)_{\delta}$  with the following properties:
     \begin{enumerate}
         \item[1$'$] \textbf{ Global grains with $C^2$-slope function. }

         There is a function $f: [-1,1]\rightarrow \mathbb{R}$ with $\|f\|_{C^2}\leq 1$ so that $(E_{\TT} \cap \{z=z_0\})\cdot (1, f(z_0), 0)$ is a $(\delta, 1-\sigma, \delta^{-\eps})_1$-ADset for each $z_0 \in [-1,1]$. 
         
         \item[2] \textbf{ Local grains with Lipschitz plane map. }

    $(\TT, Y)_{\delta}$ has a $1$-Lipschitz plane map $V$. For all $\rho\in [\delta, 1]$ and all $p\in \mathbb{R}^3$, $V(p)\cdot (B(p, \rho^{1/2})\cap E_{\TT})$ is a $(\rho, 1-\sigma, \delta^{-\eps})_1$-ADset. 
     \end{enumerate}
 \end{prop}

\subsection{The slope function has large slope}\label{largeSlopeSection}
Our next task is to show that the global grains slope function $f$ in Proposition~\ref{prop: grain} has derivative $f'$ with magnitude bounded away from 0. As discussed above, the proof of the analogous statement \cite[Proposition 6.1]{WZ} used the property that $\tubes$ contains tubes pointing in many different directions. Here, we will use the property that $\tubes$ satisfies the Convex Wolff Axioms at every scale. The precise statement we need is as follows. 

\begin{prop}\label{prop: large slope}
    For all $\eps, \delta_0>0$, there exists $\delta\in (0, \delta_0]$ and an $\eps$-extremal set of tubes $(\TT, Y)_{\delta}$ with the following properties: 
    \begin{enumerate}
         \item[1$''$]
         \textbf{Global grains with nonsingular $C^2$-slope function.}
         There is a function $f: [-1,1]\rightarrow \mathbb{R}$ with $1\leq |f'(z)|\leq 2$ and $|f''(z)|\leq 1/100$ for all $z\in [-1,1]$,  so that $(E_{\TT} \cap \{z=z_0\})\cdot (1, f(z_0), 0)$ is a $(\delta, 1-\sigma, \delta^{-\eps})_1$-ADset for each $z_0 \in [-1,1]$.

         \item[2] \textbf{ Local grains with Lipschitz plane map. }

    $(\TT, Y)_{\delta}$ has a $1$-Lipschitz plane map $V$. For all $\rho\in [\delta, 1]$ and all $p\in \mathbb{R}^3$, $V(p)\cdot (B(p, \rho^{1/2})\cap E_{\TT})$ is a $(\rho, 1-\sigma, \delta^{-\eps})_1$-ADset. 
    \end{enumerate}
\end{prop}

Before proving \ref{prop: large slope}, we will establish the following technical result, which is the analogue of \cite[Lemma 6.3]{WZ}. 
\begin{lem} \label{lem: large slope}
    For all $\eps>0$, there exists $\eta>0, \delta_0>0$ so that the following holds for all $\delta\in (0, \delta_0]$.  Let $(\TT, Y)_{\delta}$ be a set of tubes that satisfies the conclusions of Proposition~\ref{prop: C2}, with $\eta$ in place of $\eps$. Let $J\subset[-1,1]$ be an interval, with $\delta^{1/2}\leq |J|\leq \delta^{\eps}$. Suppose that 
    \begin{equation}\label{eq: dense}
        \sum_{T\in \TT} |Y(T)\cap (\RR^2\times J)|\geq \delta^{\eta} |J|,
    \end{equation}
    and that $(\RR^2\times J)\cap \bigcup_{\TT} Y(T)$ can be covered by a union of $\leq \delta^{-\eta}|J|^{-2+\sigma}$ cubes of side-length $J$. 

    Then the slope function $f: [-1,1]\rightarrow \RR$ from Proposition~\ref{prop: C2} satisfies 
    \[
    |f'(z)|\geq |J| \text{ for all } z\in J. 
    \]
\end{lem}
\begin{proof}
    The proof is the same as \cite[Lemma 6.3]{WZ} until {\bf Step 4: $\TT$ contains parallel tubes. }

Here is a summary of Steps 1-3. After a refinement, we may suppose that for each point $p\in F:= (\RR^2 \times J) \cap \bigcup_{\TT} Y(T)$, 
\begin{equation} \label{eq: Tp}
\delta^{2-\sigma+2\eta} (\#\TT)  \leq \#\TT(p) \leq \delta^{2-\sigma-2\eta} (\#\TT) . \end{equation}  
After a translation followed by a rotation that fixes the $z$ axis (i.e.~a rotation of the form $(x,y,z)\mapsto(x\cos\theta-y\sin\theta, x\sin\theta+y\cos\theta,z)$), we may suppose that $J=[0, \rho^{1/2}]$ and $f(0)=0$.  

Suppose for contradiction that $|f'(z)|\leq |J|$ for some $z\in J$. Since $|f''(z)|\lesssim 1$, we have $|f'(z)|\lesssim |J|$ for all $z\in J$, and thus $|f(z)|\lesssim |J|^2$.  We find a subset $F_2\subset F$ with $|F_2|\geq \delta^{4\eta}|F|$ so that the follows holds:
\begin{enumerate}

    \item For each $z_1\in J$ and each $r\in (F\cap \{z=z_1\})\cdot (1, f(z_1), 0)$, the $1\times \delta \times \delta$-rectangular prism 
    \begin{equation}\label{eq: Grz}
    G_{r, z_1} =\{ |z-z_1|\leq \delta\}\cap \{|(x,y)\cdot (1, f(z_1))-r|\leq \delta\}
    \end{equation}
    has a large intersection with $F_2$, i.e.
    \begin{equation}
        |G_{r, z_1}\cap F_2|\geq \delta^{4\eta}|G_{r,z_1}|\geq \delta^{2+4\eta}.
    \end{equation}

    \item Let $\Pi_{x,z}: \RR^3\rightarrow \RR^2$ be the projection to the first and third coordinates, then \[|N_{\rho^{1/2}}\Pi_{x,z}(F_2)|\leq \delta^{-7\eta} \rho^{1/2+\sigma/2}.\]

    \item 
    There is a $y_0\in [-1,1]$ such that 
    \[
    |\{y=y_0\}\cap F_2|\geq |F_2|/2,
    \]
    where the $|\cdot|$ on the LHS denote $2$-dimensional measure and $|\cdot|$ on the RHS denote $3$-dimensional measure, 
    and each $G_{r, z_1}$ defined in \eqref{eq: Grz} satisfies
    \[
    G_{r, z_1}\cap \{y=y_0\}\cap F_2\neq \emptyset.
    \]
\end{enumerate}

Cover $[-1,1]\times \{y=y_0\}\times J$ by a set $\mathcal{S}$ of interior-disjoint squares of side-length $\rho^{1/2}=|J|$. 
Let $S\in \cS$ be a square with center $p_S=(x_S, y_0, \rho^{1/2}/2)$. By Item 2 from Proposition~\ref{prop: C2}, 
$S\cap F_2 \cdot V(p_S)$ is a $(\rho, 1-\sigma, \delta^{-\eta})_1$-ADset. Write $V(p_S)=(V_x(p_S), V_y(p_S), V_z(p_S))$ and define  $\tilde{V}(p_S)=(V_x(p_S), 0, V_z(p_S))$.   Since $S\cap F_2\subset \{ y= y_0\}, $ 
\begin{equation}
    (S\cap F_2) \cdot \tilde{V}(p_S) \text{ is a } (\rho, 1-\sigma, \delta^{-\eta})_1\text{-ADset}.
\end{equation}
Define $P_S =\{(x,y,z): (x,z)\in S\}$, then 
\[
(F_2\cap P_S)\cdot \tilde{V}(p_S) \text{ is a } (\rho, 1-\sigma, 2\delta^{-\eta})_1\text{-ADset}. 
\]
This finishes the summary of Steps 1-3 and we have introduced the necessary geometric objects and notation. 

In Step 4, we will show that the constraints from the above steps will force many tubes from $\TT$ to be contained in a $\delta^{-2C_0\eta} \rho^{1/2}\times 1\times 1 $ rectangular prism (henceforth called a \emph{slab}, to match the terminology in \cite{WZ}) for some constant $C_0 = C_0(\sigma)$, which contradicts the assumption that $\TT$ satisfies the Convex Wolff Axioms (see Definition~\ref{def: extremal} and Definition \ref{convexWolffAxiomsDefn}). 

For each $T\in \TT$, define $F_2(T)=Y(T)\cap F_2$. Recalling \eqref{eq: dense} and \eqref{eq: Tp}, we have 
\begin{equation}
    \sum_{T\in \TT} |F_2(T)|\geq \delta^{10\eta} |J|=\delta^{10\eta} \rho^{1/2}. 
\end{equation}
Recall that each $T\in \TT$ intersects at most 3 prisms $P_S$, $S\in \cS$ and for each such prism we have $|F_2(T)\cap P_S|\leq \rho^{1/2}\delta^2$.  By pigeonholing, there exists a square $S\in \cS$ and a set $\TT_S =\{ T\in \TT: T\cap P_S\neq \emptyset\}$ with 
\begin{equation}\label{eq: TSlower}
\#\TT_S \gtrsim (\#\TT) (\#\cS)^{-1}\gtrsim \delta^{\eta}\rho^{-\sigma/2+1/2}(\#\TT).
\end{equation}

But for each $T\in \TT_S$, we have 
\begin{equation}\label{FTdotVADReg}
F_2(T)\cdot \tilde{V}(p_S) \subset (F_2\cap P_S)\cdot \tilde{V}(p_S) \text{ is a } (\rho, 1-\sigma, 2\delta^{-\eta})_1\text{-ADset}. 
\end{equation}

Recall that for each $T\in\TT_S$, the (shaded) tube segment $F_2(T)$ is almost full. If $\tilde{V}(p_S)$ was parallel to $\operatorname{dir}(T)$, then $F_2(T)\cdot \tilde{V}(p_S)$ would fill out most of an interval of length $\rho^{1/2}$. However, \eqref{FTdotVADReg} says that $F_2(T)\cdot \tilde{V}(p_S)$ is an ADset with dimension $1-\sigma$. Since $\sigma>0$, we conclude that $\tilde{V}(p_S)$ cannot be parallel to $\operatorname{dir}(T)$; in fact, the two vectors must be almost orthogonal. \cite[Lemma 6.2]{WZ} makes this precise; it says that if \eqref{FTdotVADReg} holds, then
\begin{equation}
    |\text{dir}(T)\cdot \tilde{V}(p_S)|\leq \delta^{-C_0\eta} \rho^{1/2}.
\end{equation}
It follows that $\TT_S$ is contained in  the slab 
\[
\Sigma_S=\{ (x, y, z)\in [0,1]^3: |(x-x_S, 0, z)\cdot \tilde{V}(p_S)|\leq \delta^{-2C_0\eta} \rho^{1/2} \}.
\]
By the Convex Wolff Axioms,
\[
\#\TT_S \leq |\Sigma_S| \delta^{-\eta} (\#\TT) \leq \delta^{-(2C+1)\eta} \rho^{1/2} (\#\TT),
\]
which yields a contradiction to \eqref{eq: TSlower}.
\end{proof}

The proof of Proposition~\ref{prop: large slope} is the same as the proof of \cite[Proposition 6.1]{WZ}, except that we use Lemma \ref{lem: large slope} in place of \cite[Lemma 6.3]{WZ}.

\subsection{Twisted projections, and the completion of the proof}\label{twistedProjectionsSection}
Recalling Definition 6.1 from \cite{WZ}, for $f\colon [-1,1]\to\RR$ we define the twisted projection $\pi_f(x,y,z) = (x + f(z)y, z)$. The consequence of Proposition~\ref{prop: large slope} is the following analogue of \cite[Corollary 6.5]{WZ}:
\begin{lem}\label{L32EstimateLowerBd}
For all $\eps, \delta_0>0$, there exists $\delta\in (0, \delta_0]$; a pair $(\tubes,Y)$ of $\delta$-tubes satisfying the Convex Wolff Axioms with error $\delta^{-\eps}$ with a $\delta^{\eps}$-dense shading; and a function $f\colon [-1,1]\to\RR$ with $1\leq |f'(z)|\leq 2$ and $|f''(z)|\leq 1/100$ for all $z\in [-1,1]$, so that
\begin{equation}
\Big| \bigcup_{T\in\tubes}\pi_f (Y(T))\Big|\leq\delta^{\sigma+\eps}.
\end{equation}
\end{lem}

To complete the proof of Theorem \ref{stickySelfSimThm}, it suffices to establish the following estimate, which is the analogue of \cite[Proposition 7.1]{WZ}:
\begin{prop}\label{L32EstimateUpperBd}
For all $\eps>0,$ there exists $\eta,\delta_0>0$ so that the following holds for all $\delta\in(0,\delta_0]$. Let $(\tubes,Y)_\delta$ be a set of $\delta$-tubes satisfying the Convex Wolff Axioms with error $\delta^{-\eta}$, and let $Y$ be a $\delta^\eta$-dense shading. Let $f\colon[-1,1]\to\RR$ with $1\leq |f'(z)|\leq 2$ and $|f''(z)|\leq 1/100$ for all $z\in[-1,1]$. Then 
\[
\Big| \bigcup_{T\in\tubes}\pi_f (Y(T))\Big| \geq \delta^{\eps}.
\]
\end{prop}
Comparing Lemma \ref{L32EstimateLowerBd} and Proposition \ref{L32EstimateUpperBd}, we conclude that $\sigma=0$, which completes the proof of Theorem \ref{stickySelfSimThm}. It remains to prove Proposition \ref{L32EstimateUpperBd}. We defer this to the next section.


\section{Twisted projections of tubes satisfying the Convex Wolff Axioms}\label{twistedProjectionsTubeSec}
Our goal in this section is to prove Proposition \ref{L32EstimateUpperBd}. In fact, we will prove a slightly stronger statement. We say a set of tubes $\tubes$ satisfies the \emph{Tube Wolff Axioms} with error $C$, if for every $\rho$-tube $T_\rho$, at most $C|T_\rho|(\#\tubes)$ tubes from $\tubes$ are contained in $T_\rho$; since $\rho$-tubes are convex, this is a special case of the Convex Wolff Axioms. We will prove the following mild generalization of Proposition \ref{L32EstimateUpperBd}:

\begin{prop}\label{strengthenedL32EstimateUpperBd}
For all $\eps>0,$ there exists $\eta,\delta_0>0$ so that the following holds for all $\delta\in(0,\delta_0]$. Let $(\tubes,Y)_\delta$ be a set of $\delta$-tubes satisfying the Tube Wolff Axioms with error $\delta^{-\eta}$, and let $Y$ be a $\delta^\eta$-dense shading. Let $f\colon[-1,1]\to\RR$ with $1\leq |f'(z)|\leq 2$ and $|f''(z)|\leq 1/100$ for all $z\in[-1,1]$. Then 
\[
\Big| \bigcup_{T\in\tubes}\pi_f (Y(T))\Big| \geq \delta^{\eps}.
\]
\end{prop}

Before proving Proposition \ref{strengthenedL32EstimateUpperBd}, we will need to introduce two closely related non-concentration conditions for subsets of $\RR^n$, which arise from discretizing fractal sets. The first was proposed by Katz and Tao \cite{KT01}, while the second was explored by Orponen and Shmerkin \cite{OrponenShmerkin2}.

\begin{defn}
        Let $\delta>0$, $s\in (0, n]$, and $C>0$. A non-empty bounded set $A\subset \mathbb{R}^n$ is called a $(\delta, s, C)$-Katz-Tao set if 
    \[
    \mathcal{E}_{\delta}\big(A\cap B(x,r)\big)\leq C(\delta/r)^s, \quad \forall x\in \mathbb{R}^n, r\in [\delta, 1].
    \]
\end{defn}

\begin{defn}
    Let $\delta>0$, $s\in (0, n]$, and $C>0$. A non-empty bounded set $A\subset \mathbb{R}^n$ is called a $(\delta, s,C)$-Frostman set if 
    \[
    \mathcal{E}_\delta\big(A\cap B(x,r)\big)\leq C r^s \mathcal{E}_{\delta}(A), \quad \forall x\in \mathbb{R}^n, r\in [\delta, 1].
    \]
\end{defn}

\begin{rem}\label{tubeWolffVsFrostmanSet}
If $\tubes$ is a set of $\delta$-tubes contained in a bounded set such as $B(0,1)$, then we can identify the coaxial line of each tube with a point in $\RR^4$. If $\tubes$ satisfies the Tube Wolff Axioms, then the corresponding set of points in $\RR^4$ forms a $(\delta, 2, O(1))$-Frostman set.
\end{rem}

\begin{rem}\label{rem: delta thick}
If $A\subset\RR^n$ is a $(\delta, s, C)$-Katz-Tao set (resp.~$(\delta,s, C)$-Frostman set), then $N_{\delta}(A)$ is a $(\delta, s, C')$-Katz-Tao set (resp.~$(\delta,s, C')$-Frostman set), where $C'=O_n(C)$, and conversely.
\end{rem}
\begin{lem}\label{lem: Katz-Tao}
Let $A\subset \RR^n$ be a $(\delta, s,C)$-Frostman set. Then there exists a $(\delta, s, 100)$-Katz-Tao set $A'\subset A$, with  $\mathcal{E}_{\delta}(A') \gtrsim_{n,s} \delta^{-s}  |\log \delta|^{-1} C^{-1}$. 
\end{lem}
\begin{proof}
Recall that the $s$-dimensional Hausdorff content  is defined as 
\[
\mathcal{H}^s_{\infty}(A):=\inf\Big\{\sum_i r(B_i)^s: A\subset \cup_i B_i\Big\},
\]
where each $B_i$ in the covering is a cube and $r(B_i)$ is the side-length of $B_i$.  Next we recall Lemma 3.13 from \cite{FasslerOrponen}
\begin{lem}\label{FOLemma}
Let $\delta,s>0$ and suppose that $\mathcal{H}^s_{\infty}(A)=\kappa>0$. Then there exists a $(\delta,s,100)$-Katz-Tao set $A'\subset A$ with $\mathcal{E}_{\delta}(A')\gtrsim \kappa \delta^{-s}$.
\end{lem}
(In \cite{FasslerOrponen}, the authors do not explicitly state the constant $C=100$, but this follows from their proof; any fixed constant would work equally well for our argument). 

In light of Lemma \ref{FOLemma}, in order to prove Lemma \ref{lem: Katz-Tao}, it suffices to obtain the bound
\begin{equation}\label{neededHsBd}
\mathcal{H}^s_{\infty}(A) \gtrsim_{n,s} |\log \delta|^{-1} C^{-1}.
\end{equation}
The remainder of the proof is devoted to establishing \eqref{neededHsBd}. By Remark~\ref{rem: delta thick}, we can assume $A$ is a union of $\delta$-balls (this introduces a harmless $O_n(1)$ loss). Let $\{B_i\}$ be a cover of $A$ by balls. Sort the balls according to their side-length; it suffices to consider balls of side-length $\geq \delta$ since $A$ is a union of $\delta$-balls. By dyadic pigeonholing, there exists a  dyadic number $\lambda\in [\delta, 1]$ such that
\[
\mathcal{E}_{\delta} \Big(\bigcup_{r(B_i)\sim \lambda}  B_i \cap A\Big) \gtrsim |\log \delta|^{-1} \mathcal{E}_{\delta}(A).
\]

 Since $A$ is a $(\delta,s, C)$-Frostman set, $\mathcal{E}_{\delta}(A\cap B_i) \leq C \lambda^{s} \mathcal{E}_{\delta}(A)$. The number of balls $B_i$ in the cover with $r(B_i) \sim \lambda $ is $\gtrsim_{n,s} \lambda^{-s} |\log \delta|^{-1}C^{-1}$. This establishes \eqref{neededHsBd} and concludes the proof. 
\end{proof}

The following definition describes when a discretized set is homogeneous at many different scales.

\begin{defn}[$\eta$-uniform set down to scale $\delta$]
     Let $\eta>0$ be a small parameter. A set $A\subset[0,1]^n$ is called an $\eta$-uniform set up to  scale $\delta$ if for any $\rho\in 2^{-\eta^{-1}\mathbb{N}}\cap [\delta,1]$, if two $\rho$-cubes $Q_1, Q_2\subset N_{\rho}(A)$, then the $\delta$-covering numbers are comparable, in the sense that 
     \[
     \frac{1}{100}\cE_{\delta}(Q_1\cap A) \leq \cE_{\delta}(Q_2\cap A)  \leq 100 \cE_{\delta}(Q_1\cap A) . 
     \]
\end{defn}

Let $\eta>0$ be a small parameter and $\delta \in  2^{-\eta^{-1}\mathbb{N}}$. Given an arbitrary set $ A \subset[0,1]^n$, \cite[Lemma 2.15]{OrponenShmerkin}  says there exists an $\eta$ uniform set  $A'\subset A$ down to scale $\delta$ and $\cE_{\delta}(A') \gtrsim |\log \delta|^{-\eta^{-1}} \cE_{\delta}(A)$. We call $A'$ an $\eta$-uniform refinement (usually $\delta$ is clear from context and we omit it from notation). 

Even though \cite[Lemma 2.15]{OrponenShmerkin} is for dyadic cubes, the same argument works for any set $A$ and the relevant $\rho$-balls $Q\subset N_{\rho}(A)$, which we restate  in our language below.

\begin{lem}[Lemma 2.15 from \cite{OrponenShmerkin} in our notation] \label{lem: uniform}
    Let $A\subset [0,1]^n$ and let $\eta>0, T\in \mathbb{N}$, $\delta= 2^{-\eta^{-1}T}$.  Then there exists an $\eta$-uniform refinement $\cP'\subset \cP$ such that 
    \[
    \cE_{\delta}(\cP')\geq (2T)^{-\eta^{-1}} \cE_{\delta}(\cP).
    \]
    In particular, if $\eps>0$ and $ T^{-1} \log(2T) \leq\eps$, then $ \cE_{\delta}(\cP')\geq \delta^{\eps} \cE_{\delta}(\cP). $
\end{lem}

Given an $\eta$-uniform set $\cP$ with $\delta$-covering number $\delta^{-s}$, the next lemma finds a Frostman-set if one zooms in at sufficiently small scale. 

\begin{lem}\label{lem: spacing}
    For any $\eps>0$,  there exists $\eta>0$ such that for any $\delta\in 2^{-\eta^{-1}\mathbb{N}}$, the following holds. Let   $\cP\subset [0,1)^n$ be an $\eta$-uniform set with $\cE_{\delta}(\cP)\geq \delta^{-s}$. Then  there exists $\rho\in (\delta^{1-\eps}, 1]$ such for any  $\rho$-cube $Q\in N_{\rho}(\cP)$, 
    \[
     \cP_Q:= S_Q(\cP\cap Q) \text{ is a } (\delta/\rho, s, (\delta/\rho)^{-4d\eps})\text{-Frostman set},
    \]
    where $S_Q: Q\rightarrow [0,1]^n$ is an affine transformation that maps $Q$ to the unit cube.  
\end{lem}
\begin{proof}
Let $\eta\ll \eps$ to be determined later. 
We shall apply the  multi-scale decomposition techniques developed by Keleti and Shmerkin \cite{KS}. By \cite[Lemma 5.21]{ShmerkinWang}, there exists a decomposition of scales such that 
\begin{itemize}
    \item $a_1=\eta^{1/2}, a_{J+1}=1$, $a_{j+1}-a_j\geq \tau$ where $\tau=\tau(\eps)>0$ depends only on $\eps$. 
    \item For each $\delta^{a_j}$-cube $Q\subset N_{\delta^{a_j}}(\cP)$, $\cP_Q:= S_Q(\cP\cap Q)$ is a $(\delta^{a_{j+1}-a_j}, t_j, \delta^{-\eta})$-set, where $S_Q:Q\rightarrow [0,1]^3$ is an affine transformation that maps $Q$ to the unit cube. 
    \item $\sum t_j(a_{j+1}-a_j) \geq s-\eps^2$. 
\end{itemize}
In principle there is no control on the values of $t_j$. However, we can fix this as follows. The ``Merging Lemma,'' \cite[Lemma 5.20]{ShmerkinWang} says that if $t_j\geq t_{j+1}$, then for  each $\delta^{a_j}$-cube $Q\subset N_{\delta^{a_j}}(\cP)$, $\cP_Q:= S_Q(\cP\cap Q)$ is a $(\delta^{a_{j+2}-a_j}, t_j', \delta^{-\eta})$-set, where $t_j'= \frac{(a_{j+1}-a_j)t_j+ (a_{j+2}-a_{j+1})t_{j+1}}{a_{j+2}-a_j}.$  

Start with the output of \cite[Lemma 5.21]{ShmerkinWang} applied to $\cP$. If there exists $t_j\geq t_{j+1}$, apply \cite[Lemma 5.20]{ShmerkinWang} to merge the two intervals $[a_j, a_{j+1}]$ and $[a_{j+1}, a_{j+2}]$ to $[a_j, a_{j+2}]$ and the two numbers $t_j, t_{j+1}$ becomes $t_j'$.  Rename the set of intervals and $t_j$s and replace $J$ by $J-1$. Iterate the process until $t_j<t_{j+1}$  for all $j$. Now we have obtained a  decomposition of scales $\{[a_j, a_{j+1}]\}_{j=1}^J$ such that 
\begin{itemize}
    \item $a_1=\eta^{1/2}, a_{J+1}=1$, $a_{j+1}-a_j\geq \tau$ where $\tau=\tau(\eps)>0$ depends only on $\eps$. 
    \item For each $\delta^{a_j}$-cube $Q\subset N_{\delta^{a_j}}(\cP)$, $\cP_Q:= S_Q(\cP\cap Q)$ is a $(\delta^{a_{j+1}-a_j}, t_j, \delta^{-\eta})$-set, where $S_Q:Q\rightarrow [0,1]^3$ is an affine transformation that maps $Q$ to the unit cube. 
    \item $t_j<t_{j+1}$ and $\sum t_j(a_{j+1}-a_j) \geq s-\eps^2$. 
\end{itemize}

Let $J_0$ be the smallest $j$ such that $t_j\geq s-\eps$. It follows from 
\[
s-\eps^2 \leq  \sum_j t_j(a_{j+1}-a_j) \leq (s-\eps)a_{J_0}  + d(1-a_{J_0}) 
\]
 that $a_{J_0} \leq 1- \frac{\eps}{2d}.$

 Define  $\rho=\delta^{a_{J_0}}$. For any $\rho$-cube $Q\subset N_{\rho}(\cP)$ and $\cP_Q=S_Q(\cP\cap Q)$. Since $s-\eps \leq t_{J_0}<t_{J_0+1}<\cdots <t_J$, $\cP_Q$ is a union of $\delta/\rho$-balls that is a $(\delta/\rho, s, (\delta/\rho)^{-2\eps})$-set if $\eta\ll \eps^2$.  The conclusion holds with $\eps$ replaced by $2d\eps$. 
\end{proof}

\begin{lem}\label{lem: construct}
    Let $s\in (0, n]$. Suppose $Q_0\subset[0,1]^n$ is a cube of side-length $\rho$ and $\cP_0\subset Q_0$ is a $(\delta, s, C)$-Katz-Tao set with $\mathcal{E}_{\delta}(\cP_0)\gtrsim (\rho/\delta)^s.$ Then there exists a $(\delta, s, O(C))$-Katz-Tao set $\cP$ with $\mathcal{E}_{\delta}(\cP)\gtrsim \delta^{-s}$ so that for each axis-aligned $\rho$-cube $Q$ that intersects $\cP$, we have that $Q\cap \cP$ is a translation of $\cP_0$. 
\end{lem}
\begin{proof}
    Divide the unit cube $[0,1]^n$ into a union of cubes $\tilde{Q}$ of side length $\rho^{s/d} \leq \rho$.  For each $\tilde{Q}$, let $A_{\tilde{Q}}$ denote the translation that maps the center of $Q_0$ to the center of $\tilde{Q}$. Define $\cP=\bigsqcup A_{\tilde{Q}}(\cP_0)$. Since the side-length of $\tilde{Q}$ is greater than the side-length of $Q_0$, $A_{\tilde{Q}}(\cP_0)$ are disjoint, and so $\mathcal{E}_{\delta}(\cP)\gtrsim \delta^{-s}$. 

    It remains to verify that $\cP$ is a $(\delta, s, C)$-Katz-Tao set.  By the construction of $\cP$, it suffices to check that for each $r\geq \rho$, 
    \[
    \mathcal{E}_{\delta}(\cP\cap B_r) \lesssim C (r/\delta)^s. 
    \]
    To see this, when $r\in (\rho, \rho^{s/n})$, we have
    $\mathcal{E}_{\delta}(\cP\cap B_r)\sim \mathcal{E}_{\delta}(\cP_0) \lesssim C (\rho/\delta)^s.$ On the other hand, when $r\geq \rho^{s/n}$,  we have
    \[
    \mathcal{E}_{\delta}(\cP\cap B_r)\lesssim C \big(\frac{r}{\rho^{s/n}}\big)^n \cdot  \big(\frac{\rho}{\delta}\big)^s \lesssim C \big(\frac{r}{\delta}\big)^s.\qedhere
    \]
\end{proof}

Next we will prove a weaker variant of Proposition \ref{strengthenedL32EstimateUpperBd}, which relates the area of the twisted projection $\bigcup_{\tubes}\pi_f (Y'(T))$ to the area of this projection at a coarser resolution.

\begin{lem}\label{lem: strengthenedL32EstimateUpperBd}
For all $\eps>0,$ there exists $\eta,\delta_0>0$ so that the following holds for all $\delta\in(0,\delta_0]$. Let $(\tubes,Y)_\delta$ be a set of $\delta$-tubes satisfying the Tube Wolff Axioms with error $\delta^{-\eta}$, and let $Y$ be a $\delta^\eta$-dense shading. Let $f\colon[-1,1]\to\RR$ with $1\leq |f'(z)|\leq 2$ and $|f''(z)|\leq 1/100$ for all $z\in[-1,1]$. Then there exists $\rho \in (\delta^{1-\eps^2}, 1)$  and $\delta^{4\eta}$-shading $Y'\subset Y$ such that 
\begin{equation}\label{eq: iteration}
\Big| \bigcup_{T\in\tubes}\pi_f (Y'(T))\Big| \geq (\delta/\rho)^{\eps} \Big| N_{\rho}\big( \bigcup_{T\in\tubes}\pi_f (Y'(T)) \big)\Big|.
\end{equation}
\end{lem}

The quantity on the RHS of \eqref{eq: iteration} will be closely related to the graph of cinematic functions (see \cite[Definition 7.1]{WZ}). To estimate it, we recall \cite[Theorem 7.2]{WZ}:

 \begin{thm}\label{thm: cinematic}
     Let $I\subset \mathbb{R}$ be a compact interval and let $M\subset C^2(I)$ be a family of cinematic functions (see \cite[Definition 1.6]{WZ}). Then for all $\eps>0$, there exists $\delta_0>0$ so that the following holds for all $\delta\in (0, \delta_0]$. Let $G\subset M$ be a $\delta$-separated set that satisfies the following Frostman non-concentration condition:
     \begin{equation}\label{nonConcentrationConditionInThmCinematic}
         \#(G\cap B) \leq \delta^{-\eps} (r/\delta) \text{ for all balls } B\subset C^2(I) \text{ of diameter } r. 
     \end{equation}
     Then
     \[
     \Big\Vert\sum_{g\in G}\chi_{g^{\delta}}\Big\Vert_{3/2}\leq \delta^{-\eps},
     \]
     where $g^{\delta}$ is the $\delta$-neighborhood of \text{graph}(g).
 \end{thm}

\begin{proof}[Proof of Lemma \ref{lem: strengthenedL32EstimateUpperBd}]
Let $X_0=\bigcup_\tubes \pi_f(Y(T))$, and let $0<\eta\ll \eps$ be a small parameter to be chosen later. For each dyadic number $1<\lambda < \delta^{-4}$, let $X_{\lambda}$ be the set of $x\in X_0$ such that $
\sum_{p\in \pi_f^{-1}(x)} \#\mathbb{T}(p)\sim \lambda$,  where the sum is over a $\delta$-separated set of $p\in \pi_f^{-1}(x)$ and   $\#\mathbb{T}(p)$ was defined in \eqref{eq: TpDefn}.  By pigeonholing, there exists $\lambda$ such that $ Y_{\lambda}(T):= Y(T)\cap \pi_f^{-1}(X_{\lambda})$ is a $|\log\delta|^{-2} \delta^{\eta}$-dense shading. 

Apply Lemma~\ref{lem: uniform} to find an $\eta$-uniform  subset $X\subset X_\lambda$ such that $|X| \geq \delta^{\eta/2} |X_\lambda|$. By double counting, this implies that for a $\geq \delta^{2\eta}$-fraction of $\mathbb{T}$, $|Y(T)\cap \pi_f^{-1}(X)|\geq \delta^{2\eta} |T|$.  Define $Y'(T)= Y_{\lambda}(T)\cap \pi_f^{-1}(X)$.


Observe that the image of the line $l_{a,b,c,d} := \{(a+ct, b+dt,t ): t\in\mathbb{R}\}$ under the twisted projection $\pi_f$ is the curve 
\[
\pi_f (l_{a,b,c,d})=\{ (a+ct+f(t)(b+dt),t)\}.
\]  
After possibly refinement and rotation, each tube $T\in \TT$ is $\delta$-neighborhood of a line of the form $l_{a,b,c,d}$ intersecting the unit cube. Let $\cP(\TT)$ denote the set of tuples $(a,b,c,d)$ corresponding to these lines, i.e. 
\[ 
\cP(\TT):=\{ (a,b,c,d): T_{a,b,c,d}:=N_{\delta}(l_{a,b,c,d})\cap [0,1]^3\in \TT\}.
\] 
Since $\TT$ satisfies the Tube Wolff axioms, by Remark \ref{tubeWolffVsFrostmanSet} we have that $\cP(\TT)\subset \RR^4$ is a $(\delta, 2,\delta^{-\eta})$-Frostman set. 

Let $\cP_c = \{ (a, b, d): (a,b,c,d)\in \cP(\TT)\}$. By Fubini,  there exists a $c\in [0,1]$ such that $\cE_{\delta}(\cP_c):=\delta^{-s} \geq \delta^{-1+\eta}\geq \delta^{-1+\eps^2}$. 
Apply Lemma~\ref{lem: uniform} to replace $\cP_c$ with an $\eta$-uniform subset, which we will continue to denote by $\cP_c$.

At this point we have a bound on the $\delta$-covering number of $\cP_c$, but we do not know if this set satisfies our desired non-concentration condition. To apply Theorem \ref{thm: cinematic}, we need to know that $\cP_c$ satisfies \eqref{nonConcentrationConditionInThmCinematic}. By Lemma~\ref{lem: spacing}, there exists $ \rho' \in (\delta^{1-\eps^2}, 1]$ such that for any $\rho$-cube $Q\subset N_{\rho'}(\cP_c)$, 
$\cP_{Q}:= S_Q(\cP_c\cap Q)$ is a $(\rho, 1-\eps^2, \rho^{-12\eps^2})$-set with $\rho=\delta/\rho'$. 

By Lemma~\ref{lem: Katz-Tao}, since $\cP_Q$ is a $(\rho, 1-\eps^2, \rho^{-12\eps^2})$-set,  there exists 
a  $(\rho, 1-\eps^2, 100)$-Katz-Tao set $\cP_Q'\subset \cP_Q$ with $\cE_{\rho}(\cP_Q') \gtrsim \rho^{14\eps^2-1}.$  Then $S_Q^{-1} (\cP_Q')$ is a  $(\delta, 1-\eps^2, 100)$-Katz-Tao set contained  in $Q$, a cube of side-length $\rho'=\delta/\rho$. By Lemma~\ref{lem: construct}, there
exist $\sim \rho'^{-1}$ translations $A_i$ such that the set $\cP_c':= \bigsqcup_{i} A_i (S_Q^{-1}(\cP_Q'))$ is a  $(\rho, 1-\eps^2, 100)$-Katz-Tao set with $\cE_{\delta}(\cP_c')\gtrsim  \rho^{14\eps^2} \delta^{-1}.$ For each $A_i$ and $(a,b,d)\in S_Q^{-1}(\cP_Q')$, define the  shading on the tube $T'=T_{a',b',c,d'}$, where $(a',b',d')=A_i(a,b,d)$, as
\[ (a'+ct, b'+d't,t)\in Y'(T') \text{ if and only if  } (a+ct, b+dt, t)\in Y'(T).\]  
For each $t$, define 
$\pi_{f,t}= (x+f(t)y)$. Then for each fixed pair $(A_i, A_j)$, 
\begin{equation}\label{equivalenceOfUnions}
\Big|\bigcup_{A_i} \pi_{f,t} (Y'(T)) \Big|\sim \Big|\bigcup_{A_j} \pi_{f,t}(Y'(T))\Big|,
\end{equation}
where the union $\bigcup_{A_i}$ is taken over all $(a,b,d)\in A_i(S_Q^{-1}(\cP_Q'))$ and the corresponding $T=T_{a,b,c,d}$.  To see that \eqref{equivalenceOfUnions} holds, let $A_j\circ A_i^{-1}(x)= x+ (a_0,b_0, d_0)$. 
Then $a+ct+ f(t)(b+dt)\in \Omega_{i,t}:= \bigcup_{A_i} \pi_{f,t}(Y'(T))$ if and only if  \[ a+a_0 + ct+  f(t)( b+b_0 + (d+d_0)t)\in \Omega_{j,t}:=\bigcup_{A_j}\pi_{f,t}(Y'(T)).\]
 Therefore, \[
 \Omega_{i,t}= \Omega_{j,t}+ a_0 +b_0f(t) + d_0f(t)t,\]
which establishes \eqref{equivalenceOfUnions}. Finally, integrating \eqref{equivalenceOfUnions} in $t$ yields
 \begin{equation}\label{eq: translate}
 \Big|\bigcup_{A_i} \pi_f (Y'(T))\Big| \sim \Big|\bigcup_{A_j} \pi_{f} (Y'(T))\Big|.
 \end{equation}

By \cite[Lemma 7.3]{WZ}, the set of functions
\[ 
M_f: =\big\{ g_{a,b,d}(t)=a+bf(t)+ dtf(t): a,b,d\in [-1,1]\big\} 
\]
is a family of cinematic functions (see \cite[Definition 7.1]{WZ}). 
Apply  Theorem~\ref{thm: cinematic} with $I=[-1,1]$, $M=M_f$ and $G= \{g_{a,b,d}: (a,b,d)\in \cP_c'\}$. We conclude that
\[
\Big\Vert \sum_{g\in G} \chi_{g^{2\delta}} \Big\Vert_{3/2} \lesssim \delta^{-\eps^2}.
\]
If $g=g_{a,b,d}$ then $\pi_f(T)\subset g^{2\delta}$ with $T=T_{a,b,c,d}$. Let $\mathbb{T}_0=\{T_{a,b,c,d}: (a,b,d)\in \cP_c'\}$.
Therefore, 
\begin{align*}
    \rho^{15\eps^2} \leq \sum_{T\in \mathbb{T}_0} |Y'(T)|& \leq \Big|\bigcup \pi_f(Y'(T))\Big|^{1/3} \Big\Vert\sum_{T\in \mathbb{T}_0} \chi_{\pi_f(Y'(T))}\Big\Vert_{3/2} \\
    & \leq \Big|\bigcup \pi_f(Y'(T))\Big|^{1/3} \Big\Vert\sum_{T\in \mathbb{T}_0} \chi_{\pi_f(T)}\Big\Vert_{3/2} \\
    & \leq \Big|\bigcup \pi_f(Y'(T))\Big|^{1/3} \Big\Vert\sum_{g\in G} \chi_{g^{2\delta}} \Big\Vert_{3/2}.
\end{align*}
Therefore, 
\[
\Big|\bigcup_{T\in \mathbb{T}_0} \pi_f(Y'(T))\Big|\gtrsim  \rho^{45\eps^2}.
\]
 Since $\cP_c'$ consists of $\lesssim \rho/\delta$ translated copies of $ S_Q^{-1}(\cP_Q')$, and the latter is a subset of $\cP_c\cap Q$, \eqref{eq: translate} implies 
\begin{equation}\label{eq: fatcurve}
\Big|\bigcup_{(a,b,c,d)\in \cP_c\cap Q} \pi_f(Y'(T_{a,b,c,d}))\Big|\gtrsim \rho^{45\eps^2} \delta/\rho. 
\end{equation}

Note that the LHS of \eqref{eq: fatcurve} is contained in the $10\delta/\rho$-neighborhood of a curve $\pi_f(l_{a,b,c,d})$,  $(a,b,d)\in Q$, whose volume is comparable to the RHS of \eqref{eq: fatcurve}, up to a factor of $\rho^{-45\eps^2}$.  We conclude that the $\eta$-uniform set $X=\bigcup_{\mathbb{T}}\pi_f(Y'(T))$ satisfies 
\[
|N_{\delta/\rho}X|\leq \rho^{-50\eps^2} |X|.\qedhere
\]
\end{proof}

We prove Proposition~\ref{strengthenedL32EstimateUpperBd} by iterating Lemma~\ref{lem: strengthenedL32EstimateUpperBd}. The details are as follows. 
\begin{proof}[Proof of Proposition~\ref{strengthenedL32EstimateUpperBd}]
Let $0<\eta\ll \eps$ be a small number to be determined below. Let $\mathbb{T}'\subset \mathbb{T}$ be the set of tubes with $|Y(T)|\geq \delta^{2\eta}|T|$; then $\#\mathbb{T}'\geq \delta^{\eta} (\#\mathbb{T}).$ Apply Lemma~\ref{lem: uniform} to find  an $\eta$-uniform subset  $\mathbb{T}''\subset \mathbb{T}'$ with $\#\mathbb{T}''\geq \delta^{\eta} (\#\mathbb{T}')$.  Then $(\mathbb{T}'', Y)_{\delta}$ becomes a set of $\delta$-tubes satisfying the Tube Wolff Axioms with error $\delta^{-3\eta}$, and $Y$ is a $\delta^{2\eta}$-dense shading.  To ease notation, we will rename $\tubes''$ to be $\tubes$.

    By Lemma~\ref{lem: strengthenedL32EstimateUpperBd}, we can find a scale $\rho_1\in (\delta^{1-\eps^2}, 1)$ and a $\delta^{4\eta}$-shading $Y_1\subset Y$ such that
    \begin{equation}\label{projectionVsScaleRho1}
\Big| \bigcup_{T\in\tubes}\pi_f (Y(T))\Big| \geq \Big| \bigcup_{T\in\tubes}\pi_f (Y_1(T))\Big| \geq (\delta/\rho_1)^{\eps} \Big| N_{\rho_1}\big( \bigcup_{T\in\tubes}\pi_f (Y_1(T)) \big)\Big|.
\end{equation}
Consider $\mathbb{T}_{\rho_1}$ as the set of distinct $\rho_1$-tubes containing at least one tube in $\mathbb{T}$. Since $\mathbb{T}$ is $\eta$-uniform, $\mathbb{T}_{\rho_1}$ satisfies the Tube Wolff Axioms with error $\delta^{-3\eta}$. Define $Y_{\rho_1}$ to be the shading on $\tilde{T}\in \mathbb{T}_{\rho_1}$ given by $Y_{\rho_1} = N_{\rho_1}\big(\bigcup_{\mathbb{T}[\tilde{T}]} Y_1(T)\big)$. Therefore, $Y_{\rho_1}$ is a $\delta^{2\eta}$-dense shading. Apply Lemma~\ref{lem: strengthenedL32EstimateUpperBd} to $(\mathbb{T}_{\rho_1}, Y_{\rho_1})_{\rho_1}$ and find $\rho_2\in (\rho_1^{1-\eps^2}, 1)$ and corresponding $Y_2\subset Y_{\rho_1}$ such that 
    \[
\Big| \bigcup_{\tilde{T}\in\mathbb{T}_{\rho_1}}\pi_f (Y_2(\tilde{T}))\Big| \geq (\rho_1/\rho_2)^{\eps} \Big| N_{\rho_2}\big( \bigcup_{\tilde{T}\in\mathbb{T}_{\rho_1}}\pi_f (Y_2(\tilde{T})) \big)\Big|.
\]
Since for any $\tilde{T}\in \mathbb{T}_{\rho_1}$,  we have
\[
Y_2(\tilde T)\subset Y_{\rho_1}(\tilde T) =N_{\rho_1}\Big(\bigcup_{T\in\mathbb{T}[\tilde{T}]} Y_1(T)\Big),
\]
we have 
\[
\Big| \bigcup_{T\in\tubes}\pi_f (Y(T))\Big| \geq (\delta/\rho_2)^{\eps}\Big| N_{\rho_2}\big( \bigcup_{\tilde{T}\in\mathbb{T}_{\rho_1}}\pi_f (Y_2(\tilde{T})) \big)\Big|. 
\]

Iterate the above process and obtain correspondingly scales $\rho_3, \dots, \rho_N$ until $\rho_N \geq \delta^{\eps^2}$. Since $\rho_{k+1} \geq \rho_k^{1-\eps^2}$, $N\leq \eps^{-2}$. When $\eta>0$ is sufficiently small depending on $\eps$, 
\[
\Big| \bigcup_{T\in\tubes}\pi_f (Y(T))\Big| \geq (\delta/\rho_2)^{\eps}\Big| N_{\rho_N}\big( \bigcup_{\tilde{T}\in\mathbb{T}_{\rho_{N-1}}}\pi_f (Y_N(\tilde{T})) \big)\Big| \geq \delta^{\eps} \rho_N^{2} \geq \delta^{2\eps}. 
\]
We conclude with $\eps$ replaced by $\eps/2$. 
\end{proof}

\section{The consequences of Theorem \ref{mainThmDiscretized}}\label{proofOfAssouadDimThmSec}
In this section we will show how Theorem \ref{mainThmDiscretized} implies Theorem \ref{assouadDimThm}, Theorem \ref{equalHausdorffPackingKakeya}, and Corollary \ref{KakeyaTwoScales}.

\begin{proof}[Proof of Theorem \ref{assouadDimThm}] 
Suppose to the contrary that that there exists a Kakeya set $K\subset\RR^3$ with $\dimA(K)=3-\beta$ for some $\beta>0$, i.e.~there exists $C,r_0>0$ so that
\begin{equation}\label{KakeyaSmallBallEstimate}
\sup_{x\in \RR^3} \mathcal{E}_{\rho}\big( K \cap B(x,r)\big)\leq C(r/\rho)^{3-\beta}\quad\textrm{for all}\ x\in\RR^3,\ 0<\rho<r\leq r_0. 
\end{equation}
Replacing $C$ by $(1 +100r_0^{-3})C$, we may suppose that \eqref{KakeyaSmallBallEstimate} holds with $r_0=1$.

Let $\eta,\delta_0$ be the values from Theorem \ref{mainThmDiscretized}, with $\eps=\beta/2$. Let $\delta\in (0,\delta_0]$ be chosen sufficiently small (we will describe the choice of $\delta$ in further detail below), and let $\tubes$ be a set of $\delta^{-2}$ $\delta$-tubes contained in $N_{\delta}(K)$ that point in $\delta$ separated directions; in particular the tubes in $\tubes$ are essentially distinct and satisfy the Convex Wolff Axioms with error $100$; we will select $\delta\in (0,\delta_0]$ sufficiently small so that $100\leq\delta^{-\eta}$. For each $T\in\tubes,$ let $Y(T)=T$; this is a 1-dense shading. 

Apply Theorem \ref{mainThmDiscretized} with $\eps=\beta/2$. We conclude that there exists $\rho,r\in [\delta,1]$ with $\rho\leq\delta^{\eta}r$, and a ball $B$ of radius $r$ so that \eqref{rhoRNbhdFullVol} holds, and hence
\begin{equation}\label{bigRhoCoveringNumberInBall}
\mathcal{E}_{\rho}\big( K \cap B\big)\gtrsim \rho^{-3}\Big| B \cap N_{\rho}\Big(\bigcup_{T\in\tubes}Y(T)\Big)\Big| \gtrsim (\rho/r)^{\eps} (r/\rho)^3\geq \delta^{-\beta\eta/2}(r/\rho)^{3-\beta}. 
\end{equation}
Comparing \eqref{KakeyaSmallBallEstimate} and \eqref{bigRhoCoveringNumberInBall}, we obtain a contradiction, provided $\delta>0$ is selected sufficiently small depending on $C,\beta,\eta$, and the implicit constant in \eqref{bigRhoCoveringNumberInBall}. We conclude that every Kakeya set in $\RR^3$ has Assouad dimension 3. 
\end{proof}

\begin{rem}
As noted in the introduction, the conclusion of Theorem \ref{assouadDimThm} continues to hold if we remove the hypotheses that the lines point in different directions, and replace this with the requirement that the lines satisfy a variant of the Wolff axioms. This is most easily stated in the discretized setting: For all $\eps>0$, there exists $\eta>0$ so that the following holds. Let $\tubes$ be a set of $\delta$-tubes that satisfy the Convex Wolff Axioms with error $\delta^{-\eta}$. Then $\bigcup T$ has discretized Assouad dimension $\geq 3-\eps$ at scale $\delta$ and scale separation $\delta^{-\eta}$. 
\end{rem}

\begin{proof}[Proof of Theorem \ref{equalHausdorffPackingKakeya}]
Suppose to the contrary that there exists a Kakeya set $K$ with stably equal Hausdorff and packing dimension $\alpha < 3$. Without loss of generality, we can assume $K\subset B(0,1)$. Let $\eps = (3-\alpha)/4$; in what follows, all implicit constant may depend on $\eps$. Let $\eta>0$ and $\delta_0>0$ be the quantities from Theorem \ref{mainThmDiscretized} for this choice of $\eps$. Let $\eta_1=(3-\eps)\eta/6$. Since $K$ has stably equal Hausdorff and packing dimension, we can find a measure $\mu$ supported on $K$ so that
\begin{equation}\label{ballEstimateForMu}
	 r^{\alpha+\eta_1}\lesssim \mu(B(x,r))\lesssim r^{\alpha-\eta_1},\quad x\in\supp(\mu),\ 0\leq r\leq 1,
\end{equation} 
and there exists $\tau>0$ and a set $\Omega\subset S^2$ of positive measure, so that for each $e\in \Omega$ there is a line $\ell_e$ pointing in direction $e$, with  $|\ell_e\cap\supp(\mu)|\geq\tau$.

Let $\Omega_\delta$ be a $\delta$-separated subset of $\Omega$ of cardinality $\#\Omega_\delta\gtrsim \delta^{-2}|\Omega|$. For each $e\in\Omega_\delta$, let $T_e$ be a $\delta$-tube with coaxial line $\ell_e$, which satisfies $|T_e \cap \ell_e\cap\supp(\mu)|\geq\tau/2$. Define the shading
\[
Y(T_e) = \bigcup_{x\in T_e \cap \ell_e\cap\supp(\mu)}B(x,\delta).
\]
We have $|Y(T_e)|\gtrsim\tau|T|$ for each tube $T_e$; we will select $\delta\in(0,\delta_0]$ sufficiently small so that $|Y(T_e)|\geq\delta^{\eta}|T|$. Let $\tubes = \{T_e\colon e\in\Omega_\delta\}$. Since the tubes in $\tubes$ point in $\delta$-separated directions, the tubes are essentially distinct and satisfy the Convex Wolff Axioms with error $O(\delta^{-2}/(\#\tubes))=O(|\Omega|^{-1})$; we will select $\delta\in (0,\delta_0]$ sufficiently small so that $\tubes$ satisfies the convex Wolff Axioms with error at most $\delta^{-\eta}$. Thus we can apply Theorem \ref{mainThmDiscretized} to conclude that there exist scales $\rho,r \in[\delta,1]$ with $\rho\leq\delta^{\eta}r$, and a ball $B$ of radius $r$ so that \eqref{rhoRNbhdFullVol} holds.

Let $E\subset B$ be a $\rho$-separated subset of $B \cap N_{\rho}\Big(\bigcup_{T\in\tubes}Y(T)\Big)$ of cardinality $\gtrsim (r/\rho)^{3-\eps}$. Then by \eqref{ballEstimateForMu} we have 
\[
r^{\alpha-\eta_1}\gtrsim \mu(B) \gtrsim \sum_{x\in E}\mu(B(x,\rho))\gtrsim (r/\rho)^{3-\eps} \rho^{\alpha+\eta_1}.
\]
Re-arranging and using the fact that $\rho/r \leq\delta^\eta,$ as well as the definition of $\eps$ and $\eta_1$, we conclude that
\[
\delta^{-\eta(3-\eps)/2}= \delta^{-\eta(3-\eps-\alpha)} \leq (r/\rho)^{-\eta(3-\eps-\alpha)}\lesssim (r\rho)^{-\eta_1}\lesssim\delta^{-2\eta_1}=\delta^{-\eta(3-\eps)/3}.
\]
If $\delta\in(0,\delta_0]$ is selected sufficiently small, then this is impossible. We conclude that every Kakeya set in $\RR^3$ with stably equal Hausdorff and packing dimension must have dimension 3.
\end{proof}

\begin{proof}[Proof of Corollary \ref{KakeyaTwoScales}]
Let $\eps>0$. Let $\eta_0>0$ and $\delta_0>0$ be the quantities from Theorem \ref{mainThmDiscretized} with $\eps/2$ in place of $\eps$. Let $\eta=\eta_0/2$, let $\tubes$ be a set of essentially distinct $\delta$-tubes in $B(0,1)$ that satisfies the Convex Wolff Axioms with error $\delta^{-\eta}$, and let $Y(T)\subset T$ be a shading with $\sum_{\tubes}|Y(T)|\geq \delta^{\eta}\sum_{\tubes}|T|$. 

Let $Y_0(T)=Y(T)$. For each $i=1,\ldots$,  apply Theorem \ref{mainThmDiscretized} to $\tubes$ with the shading $Y_{i-1}$; let $r_i$, $\rho_i$ and $B_i$ be the output from this theorem. Define $Y_i(T) = Y_{i-1}(T)\backslash B_i$. We repeat this process until $\sum_{\tubes}|Y_i(T)|<\frac{1}{2}\delta^{\eta}\sum_{\tubes}|T|$, at which point we have
\[
\sum_{T\in\tubes}\Big| Y(T) \cap \bigcup_i B_i\Big| \geq \frac{1}{2}\sum_{T\in\tubes}|Y(T)|.
\]
After dyadic pigeonholing, we can select $\rho,r\in[\delta,1]$ with $\rho\leq\delta^\eta r$, and a set of indices $I$ with the following properties:
\begin{itemize}
\item For each index $i\in I$, we have $\rho_i\in [\rho, 2\rho)$ and $r_i\in [r, 2r)$.
\item $\sum_{T\in\tubes}\Big| Y(T) \cap \bigcup_{i\in I} B_i\Big| \gtrsim|\log\delta|^{-2}\sum_{\tubes}|Y(T)|.$
\end{itemize}

Next, we can greedily select a set of indices $I'\subset I$ so that the above properties continue to hold (with the implicit constant in the second item worsened slightly), and the sets $\{B_i\colon i\in I'\}$ are pairwise disjoint. 

We claim that the values of $\rho$ and $r$ described above, plus the shading $Y'(T) = Y(T)\cap \bigcup_{i\in I'} B_i$ satisfy the conclusions of Corollary \ref{KakeyaTwoScales}. Indeed, we can compute 
\begin{align*}
\Big| N_{\rho}\Big(\bigcup_{T\in\tubes}Y'(T)\Big)\Big| & =  \Big| N_{\rho}\Big(\bigcup_{i\in I'}B_i \cap \bigcup_{T\in\tubes}Y'(T)\Big)\Big| \\
&\gtrsim \sum_{i\in I'}\Big| B_i \cap N_\rho\Big(\bigcup_{T\in\tubes}Y'(T)\Big)\Big|\\
&\gtrsim \sum_{i\in I'} (r/\rho)^{\eps/2} |B_i| \\
&\gtrsim (r/\rho)^{\eps/2} \Big|N_{r}\Big(\bigcup_{T\in\tubes}Y'(T)\Big)\Big|\\
& \geq  \delta^{-\eta_0\eps/2}(r/\rho)^{\eps} \Big|N_{r}\Big(\bigcup_{T\in\tubes}Y'(T)\Big)\Big|.
\end{align*}
After decreasing $\delta_0$ (and hence $\delta$) if necessary, the quantity $\delta^{-\eta_0\eps/2}$ dominates the implicit constants in the above inequality, and we obtain \eqref{KakeyaSameTwoScalesIneq}.
\end{proof}

\end{document}